\documentclass[oneside]{amsart}
\usepackage{amssymb}
\usepackage{amsthm}
\usepackage[dvipsnames]{xcolor}
\usepackage[colorlinks = true,
linkcolor = Blue,
urlcolor  = Blue,
citecolor = Blue,
anchorcolor = Blue]{hyperref}
\usepackage{cleveref}
\usepackage{tikz-cd} 
\usepackage{todonotes}
\usepackage{xstring}
\usepackage[mathscr]{euscript}   
\usepackage[normalem]{ulem}
\newcommand{\cat}[1]{
	\StrLen{#1}[\mystrlen]
	\ifnum\mystrlen=1 \mathscr{#1}
	\else \mathrm{#1}
	\fi}
 
\newcommand{\mytimes}[1]{\mathbin{\operatorname*{\times}_{#1}}}
\newcommand{\isomto}{\overset{\simeq}{\rightarrow}}
\newcommand{\isomfrom}{\overset{\sim}{\leftarrow}}
\newcommand{\longisomto}{\overset{\sim}{\longto}}
\newcommand{\surjto}{\twoheadrightarrow}

\newcommand{\injto}{\hookrightarrow}

\newcommand{\longto}{\longrightarrow}

\DeclareMathOperator{\Ran}{Ran}
\DeclareMathOperator{\Spet}{{Sp\acute{e}t}}
\DeclareMathOperator{\sm}{sm}

\newcommand\cA{\mathcal{A}}
\newcommand\cB{\mathcal{B}}

\newcommand\cK{\mathcal{K}}
\newcommand\cL{\mathcal{L}}
\newcommand\cM{\mathcal{M}}
\newcommand\cO{\mathcal{O}}

\newcommand\cW{\mathcal{W}}

\newcommand\bC{\mathbf{C}}
\newcommand\bE{\mathbf{E}}
\newcommand\bF{\mathbf{F}}
\newcommand\bG{\mathbf{G}}
\newcommand\bN{\mathbf{N}}

\newcommand\bZ{\mathbf{Z}}

\newcommand\bS{\mathbf{S}}

\newcommand\QQ{\mathbb{Q}}
\newcommand\EE{\mathbb{E}}

\newcommand\FF{\mathbb{F}}

\newcommand\fX{\mathfrak{X}}

\newcommand\fm{\mathfrak{m}}

\def\rR{\mathrm{R}}
\def\rH{\mathrm{H}}

\newcommand{\fil}{\mathrm{fil}}
\newcommand{\op}{\mathrm{op}}
\newcommand{\cris}{\mathrm{cris}}
\newcommand{\cn}{\mathrm{cn}}
\newcommand{\cl}{\mathrm{cl}}
\newcommand{\an}{\mathrm{an}}  
\newcommand{\fl}{\mathrm{fl}}  
\newcommand{\pd}{\mathrm{pd}}  
\newcommand{\aug}{\mathrm{aug}}  
\newcommand{\et}{\mathrm{\text{\'{e}t}}}  
\newcommand{\zar}{\mathrm{zar}}
\newcommand{\art}{\mathrm{art}}
\renewcommand{\mod}{\mathrm{mod}}
\newcommand{\loc}{\mathrm{loc}}
\newcommand{\cocart}{\mathrm{cocart}}
\newcommand{\heart}{\heartsuit}
\newcommand{\triv}{\mathrm{triv}} 

\newcommand{\can}{\mathrm{can}}

\DeclareMathOperator{\sHen}{sHen}  
\DeclareMathOperator{\sqz}{sqz}  
\DeclareMathOperator{\ff}{ff} 
\DeclareMathOperator{\DM}{DM}  
\DeclareMathOperator{\ind}{ind}  
\DeclareMathOperator{\ob}{ob}  
\DeclareMathOperator{\Mod}{Mod}   
\DeclareMathOperator{\Env}{Env}  
\DeclareMathOperator{\Fun}{Fun}  
\DeclareMathOperator{\oper}{oper}  
\DeclareMathOperator{\Fin}{Fin}

\DeclareMathOperator{\Fil}{Fil}   
\DeclareMathOperator{\Zar}{Zar}

\DeclareMathOperator{\Shv}{Shv}
\DeclareMathOperator{\Cat}{Cat}
\DeclareMathOperator{\pr}{pr}  
\DeclareMathOperator{\rL}{L}  \DeclareMathOperator{\sch}{sch}  
  	
\DeclareMathOperator{\Cris}{Cris}  	
\DeclareMathOperator{\Sch}{Sch}

\DeclareMathOperator{\Sp}{Sp} 
 
\DeclareMathOperator{\SW}{SW} 
\DeclareMathOperator{\Ab}{Ab}  
\DeclareMathOperator{\Poly}{Poly}  
\DeclareMathOperator{\ArrPoly}{ArrPoly}  
\DeclareMathOperator{\PDPoly}{PDPoly}
\DeclareMathOperator{\SurjPoly}{SurjPoly}  
\DeclareMathOperator{\ArrPDPoly}{ArrPDPoly}   
  
\DeclareMathOperator{\surj}{surj}  
\DeclareMathOperator{\PolyProj}{PolyProj}
\DeclareMathOperator{\PolyMod}{PolyMod}

\DeclareMathOperator{\Topp}{Top}

\DeclareMathOperator{\Moduli}{\mathrm{Moduli}}

\DeclareMathOperator{\Sym}{Sym}  
\DeclareMathOperator{\Pic}{Pic}  
  
\DeclareMathOperator{\len}{len}  
  
\DeclareMathOperator{\der}{der}   
\DeclareMathOperator{\Gal}{Gal}

\DeclareMathOperator{\Def}{Def}  
\DeclareMathOperator{\Per}{Per}  
\DeclareMathOperator{\ST}{ST}  
\DeclareMathOperator{\per}{per}  
\DeclareMathOperator{\st}{st}  
\DeclareMathOperator{\Ext}{Ext}  

\DeclareMathOperator{\CAlg}{CAlg}  
\DeclareMathOperator{\CRing}{CRing}  
\DeclareMathOperator{\PDRing}{PDRing}  

\DeclareMathOperator{\RGamma}{\mathrm{R}\Gamma}

\newcommand{\und}{\mathrm{und}}

\DeclareMathOperator{\Lie}{Lie}

\DeclareMathOperator{\forget}{forget}

\DeclareMathOperator{\dR}{dR} 
\DeclareMathOperator{\CR}{CR}

\DeclareMathOperator{\Spec}{Spec} 
\DeclareMathOperator{\Spf}{Spf} 
 
\DeclareMathOperator{\cofib}{cofib}

\DeclareMathOperator{\Map}{Map} 
\DeclareMathOperator{\Hom}{Hom}

\DeclareMathOperator{\res}{res} 
 
\DeclareMathOperator{\fib}{fib} 
\DeclareMathOperator{\coker}{coker}

\DeclareMathOperator{\dlog}{dlog}

\DeclareMathOperator{\Alg}{Alg} 
\DeclareMathOperator{\PrL}{\mathrm{Pr}^{\mathrm{L}}}
\DeclareMathOperator{\Set}{Set}

\DeclareMathOperator{\id}{id}  

\newcommand{\gr}{\mathrm{gr}} 
\newcommand{\ev}{\mathrm{ev}}

\theoremstyle{definition}
\newtheorem{definition}{Definition}[section]

\newtheorem{setup}[definition]{Setup}
\newtheorem{warning}[definition]{Warning}
\newtheorem{construction}[definition]{Construction}
\newtheorem{example}[definition]{Example}
\newtheorem{notation}[definition]{Notation}
\newtheorem{convention}[definition]{Convention}
\newtheorem{remark}[definition]{Remark}

\theoremstyle{plain}
\newtheorem{theorem}[definition]{Theorem}
\newtheorem{proposition}[definition]{Proposition} 
\newtheorem{lemma}[definition]{Lemma}
\newtheorem{corollary}[definition]{Corollary}
\newtheorem{mainthm}{Theorem}

\begin{document}
\title{Deformations and lifts of Calabi--Yau varieties in characteristic $p$}
	
	\author[Lukas Brantner]{Lukas Brantner}
	\address{Oxford University,   Universit\'{e} Paris--Saclay (CNRS)} 
	\email{lukas.brantner@maths.ox.ac.uk}

	\author[Lenny Taelman]{Lenny Taelman}
	\address{University of Amsterdam}
	\email{l.d.j.taelman@uva.nl} 
	
	\begin{abstract}
		We study deformations of Calabi--Yau varieties in characteristic~$p$ using techniques from derived algebraic geometry. We prove a mixed characteristic analogue of the Bogomolov--Tian--Todorov theorem (which states that Calabi--Yau varieties in characteristic~$0$ are unobstructed), and we show that ordinary Calabi--Yau varieties admit canonical lifts to characteristic~$0$, generalising the Serre--Tate theorem on ordinary abelian varieties.
	\end{abstract}
	
	\maketitle
	
	\tableofcontents 
	
	\section{Introduction}
	We call a  smooth and proper variety $X$  \textit{Calabi--Yau} if its canonical bundle \mbox{$\omega_X$ is trivial.}
	The Bogomolov--Tian--Todorov theorem \cite{Bogomolov78,Tian87,Todorov89} asserts that (formal) deformations of  complex
	Calabi--Yau varieties are unobstructed, equivalently,  that their Kuranishi spaces are smooth.

	In characteristic~$p$, this result is  generally false: Hirokado constructed the first obstructed Calabi--Yau  threefolds in \cite{Hirokado1999}, and 
	numerous  other  examples  have been discovered  since \cite{EkedahlHylandShepherdBarron2012,Schroeer2004,AchingerZdanowicz2021,CynkVanStraten2009}.
	Some of these  fail to lift to characteristic~$0$, while others even have obstructed equal characteristic deformations. Counterexamples do not just occur in `small characteristic' --- indeed, Cynk--van Straten's obstructed threefolds ~\cite{CynkVanStraten2009} live in characteristic up to~$p=9001$.

	Our principal aim is to prove  that, 
	in spite of these examples, many
	Calabi--Yau varieties in characteristic $p$ have unobstructed mixed characteristic deformations. In the ordinary case, we will even construct canonical lifts to characteristic zero, thereby generalising classical Serre--Tate theory for ordinary abelian varieties.
	
	To prove these results, we analyse the effect of derived period maps between formal moduli problems on (Koszul dual) tangent fibres.
	\newpage 
	\subsection{Statement of Results}
	We begin by summarising our main results.
	
	\subsubsection*{Unobstructedness}
	First, we establish the following mixed characteristic analogue of the classical Bogomolov--Tian--Todorov theorem:
	
	\begin{mainthm}[Unobstructedness]\label{mainthm:btt}
		Let $k$ be a perfect field of characteristic $p$ with ring of Witt vectors \mbox{$W(k)$.} Let $X$ be a smooth, proper, and geometrically irreducible scheme over $k$  with trivial canonical bundle. Assume that
		\begin{enumerate}
			\item \label{item:BTT-HdR} the Hodge--de Rham spectral sequence for $X$ degenerates at $E_1$;
			\item \label{item:BTT-no-crys-tors}
			the crystalline cohomology $\rH^\ast_{\cris}(X/W(k))$ is torsion-free.
		\end{enumerate}
		Then the mixed characteristic formal deformations of $X$ are  unobstructed.
	\end{mainthm}
	\begin{remark}
		Our proof in fact gives unobstructedness under slightly weaker hypotheses; we refer to \Cref{thm:unobstructed} of the main text for a sharper statement. 
		
		If $\dim(X) \leq p$, then condition (\ref{item:BTT-HdR}) in \Cref{mainthm:btt} is necessary by the work of Deligne--Illusie \cite{DeligneIllusie1987}. We do not know of an example of a variety with trivial canonical bundle and degenerating Hodge--de Rham spectral sequence that has obstructed deformations (and hence violates condition (\ref{item:BTT-no-crys-tors})).
	\end{remark}

	Ekedahl--Shepherd-Barron~\cite{EkedahlShepherdBarron2005}
	and Schr\"{o}er~\cite{Schroeer2003} have shown related unobstructedness theorems,
	albeit under more restrictive hypotheses. The main 
	difference is that they \emph{assume} the existence of a lift to characteristic~$0$ to prove unobstructedness.
	Derived deformation theory allows us to eliminate such a condition.

	\subsubsection*{Serre--Tate coordinates}
	Our second main theorem is a generalisation of classical Serre--Tate theory for ordinary abelian varieties. First, let us recall some notation:

	\begin{definition}[Bloch--Kato ordinary] 
		Let $X$ be smooth and proper variety over a perfect field $k$ of characteristic~$p$. 
		Denote the image of\mbox{ $d\colon \Omega^{r-1}_{X/k} \rightarrow \Omega^r_{X/k}$ by $B^r_{X/k}$.} Then $X$  is said to be \emph{Bloch--Kato $n$-ordinary} if $\rH^\ast(X,B^r_{X/k})\cong 0$ for all $r\leq n$, and   \emph{Bloch--Kato ordinary} if it is \mbox{$n$-ordinary for all $n$. }
	\end{definition} 
	\begin{remark}
		Assume that the first $n$ slopes of the Newton and the Hodge polygon of $X$ agree,  which means that  for all $i$ and all $r \leq n-1$, we have 
		\[\dim_{\QQ_p}\left( K \otimes_{W(k)} \rH^i_{\cris}(X/W(k))\right)^{F=p^r} = \dim_k \rH^{i-r}(X,\Omega^r_X) \]
		where $K =W(\overline{k})[1/p]$ and $F$ is the Frobenius.
		Then $X$ is Bloch--Kato $n$-ordinary. If $\rH^\ast_\cris(X/W)$ is torsion--free, the  \mbox{converse   holds true.}
	\end{remark}
	\begin{remark}
		An abelian variety or K3 surface is Bloch--Kato $1$-ordinary iff it is Bloch--Kato ordinary, which happens precisely if it is 
		ordinary in the classical sense.
	\end{remark}

	\begin{mainthm}[Serre--Tate coordinates]\label{mainthm:Serre-Tate}
		Let $k$ be a perfect field of characteristic $p$, with ring of Witt vectors $W(k)$.
		Let  $X$ be a smooth, proper, and geometrically irreducible scheme over $k$  of dimension $d$ and with $\omega_{X/k}\cong \cO_X$ such that
		\begin{enumerate}
			\item \label{item:ST-ordinary} $X$ is Bloch--Kato $2$-ordinary;
			\item \label{item:ST-no-torsion} the group $\rH^d(X_{\bar k,\et},\,\bZ_p)$ is torsion-free.
		\end{enumerate}
		Then the deformation functor of $X$ is canonically isomorphic to a formal group $\ST_X$ of multiplicative type over $\Spf W(k)$.
	\end{mainthm}
	\begin{remark}
		If the first two slopes of the Hodge and the Newton polygon of $X$ agree, then $X$ automatically satisfies conditions (1) and (2). If $X$ is Bloch-Kato $2$-ordinary and  $\rH^{d-1}(X,\cO_X)$ vanishes, then condition~(\ref{item:ST-no-torsion}) is automatically satisfied.
	\end{remark}

	A formal group $T$ is of multiplicative type if
	it is \'etale locally isomorphic to a finite product of groups of the form $\widehat{\bG}_{m}$ and $\mu_{p^n}$. The formal group is then determined by its character group $\Hom(T_{\bar k},\widehat{\bG}_{m,\bar k})$, which is a finite   $\bZ_p$-module equipped with a continuous $\Gal({\bar{k}/k})$-action.\vspace{3pt} 
	
	We give an explicit description of the formal torus in  \Cref{rmk:underived-ST}.  If the character group is torsion-free,    it also follows that deformations of $X$ are unobstructed.

	\begin{remark} \Cref{mainthm:Serre-Tate}   recovers results of Serre--Tate~\cite{Serre68,Katz81} for abelian varieties, and Deligne~\cite{Deligne81} and Nygaard~\cite{Nygaard83} for K3 surfaces.
	\end{remark}

	Achinger and Zdanowicz~\cite{AchingerZdanowicz2021} have shown a result related to \Cref{mainthm:Serre-Tate}  under  more restrictive hypotheses. The main difference is that they \emph{assume} that deformations of $X$ are unobstructed, and in particular that $X$ is already known to lift to $W(k)$. Our proof is   inspired by their proof. The principal novelty in our approach is again the use of derived deformation theory, which allows us to eliminate the condition that $X$ is unobstructed.
	
	\Cref{mainthm:Serre-Tate} immediately implies the following result:
	\begin{corollary}[Canonical lifts]
		Given $X/k$ as in \Cref{mainthm:Serre-Tate}, $X$ 
		admits a \emph{canonical} formal lift $X^\can$ to $W(k)$  corresponding to the unit section of $\ST_X$. 
	\end{corollary}
	This formal lift is often algebraisable:
	
	\begin{mainthm}[Algebraisation]\label{mainthm:line-bundles}
		Let $X/k$ be as in \Cref{mainthm:Serre-Tate}. If $X$ is projective, then  so is its canonical lift $X^\can$. 
	\end{mainthm}
	
	In fact, we show that 
	for $X/k$ as in \Cref{mainthm:Serre-Tate} and $\cL$ a line bundle  on $X$, the 
	deformations of the pair $(X,\cL)$  are also parametrised by a formal group $\ST_{X,\cL}$ of multiplicative type, see
	\Cref{thm:serre-tate}.

	\begin{remark}
		Unlike the case of ordinary abelian varieties or K3 surfaces, we generally do not expect the canonical lifts $X^\can$ to be of CM type, see \Cref{rmk:complex-multiplication}.
	\end{remark}

	Classical Serre--Tate theory admits an `orthogonal generalisation' to 
	abelian varieties over connective $\bE_\infty$-rings, 
	which is a  key ingredient to Lurie's moduli-theoretic interpretation of topological modular forms and Morava $E$-theory  \cite{LurieEllI, LurieEllII}.
	
	In this context, there is a $p$-complete $\bE_\infty$-ring spectrum $\SW(k)$, the \textit{spherical Witt vectors}, which is flat over the sphere and satisfies \[\pi_0(\SW(k)) \cong W(k).\]  It is then natural to ask if our canonical lifts $X^\can$ can further be lifted  to $\SW(k)$. We show that this rarely happens:  
	\begin{mainthm}\label{thm:no-spherical-lifts}
		Let $X$ be a smooth projective irreducible scheme over an algebraically closed field $k$ of characteristic $p$. Assume $\omega_{X/k}\cong \cO_X$.  If $X$ admits a lift to a formal spectral scheme over the spherical Witt vectors $\SW(k)$, then $X$ has a  finite \'etale cover by an ordinary abelian variety.
	\end{mainthm}  
	\newpage
	\subsection{Outline of proofs}
	Before outlining our proofs in characteristic $p$, we will briefly trace 
	through  the history of the 
	Bogomolov--Tian--Todorov theorem in characteristic $0$,  and review a derived proof in this regime.
	\subsubsection*{Background on the BTT theorem}
	The unobstructedness of complex Calabi--Yau varieties  was first established  by Tian~\cite{Tian87} and 
	Todorov~\cite{Todorov89}, building on Bogomolov's work   \cite{Bogomolov78} on holomorphic symplectic manifolds.
	
	Goldman and Milson~\cite{GoldmanMilson1990} later reinterpreted Tian--Todorov's transcendental argument 
	in terms of differential graded Lie algebras, using that   these control  formal  deformations in characteristic $0$.
	This  paradigm had been proposed shortly before by Deligne~\cite{DeligneLetter} and Drinfel'd~\cite{DrinfeldLetter}. It  was later 
	promoted by Lurie~\cite{LurieDAGX} and Pridham~\cite{Pridham2010} to an equivalence between derived formal  deformation functors,  so-called  \textit{formal moduli problems},    and differential graded Lie algebras.
	
	The first  algebraic proof of the Bogomolov--Tian--Todorov theorem in characteristic $0$ was given by Ran~\cite{Ran1992}.
	Iacono and Manetti  ~\cite{IaconoManetti2010} 
	(cf.\ also~\cite{FiorenzaMartinengo2012})  
	later found a  purely algebraic proof using differential graded Lie algebras,  in the spirit of Goldman--Milson. 
	
	\subsubsection*{Unobstructedness in characteristic $0$} We briefly sketch a proof of unobstructedness in characteristic $0$.
	Rather than considering deformations of $X$ over local Artinian $k$-algebras, one deforms $X$  over 
	simplicial  (i.e.\ animated) local Artinian $k$-algebras $R$.
	Given such an $R$, let us write \[\Def_X(R)\in \cat{S}\]  for the space (i.e.\ $\infty$-groupoid)  of derived schemes $\fX/R$ lifting $X$ to $R$.
	The functor  $R\mapsto \Def(R)$ is an example of a \emph{formal moduli problem}, i.e.\  a functor satisfying a derived version of the Schlessinger axioms (cf.\ \Cref{def:fmp}).

	One advantage of  formal moduli problems $F$ is that they carry a canonical and {functorial} obstruction theory. 
	This means that if $R'\surjto R$ is an elementary extension of classical local Artinian rings with kernel $I$ and $x\in \pi_0 F(R)$, then there is a canonical element $\ob(x,R'\surjto R)$ in $\pi_0 F(\sqz_k(I[1]))$ which vanishes if and only if $x$ lifts to $\pi_0 F(R')$. Here $\sqz_k(I[n])) = k \oplus I[n]$ is the trivial square-zero extension of $k$ by $I[n]$. For $I=k$, we obtain the `shifted dual numbers'.
	
The spaces $F(\sqz_k(k[0])) ,  F(\sqz_k(k[1])),   F(\sqz_k(k[2])),  \ldots  $ assemble to  a spectrum $T_F$,
	the  	\textit{tangent fibre} of $F$, which can be promoted to a chain 
	\mbox{complex over $k$.}
	In fact,   $T_F[-1]$ can even be promoted to a differential graded Lie algebra, which is Koszul dual to $F$.
	One can then recover the formal moduli problem $F$ from $T_F$:    the value $F(R)$ is given by the space of maps $\mathfrak{D}(R) \rightarrow T_F$, where  $\mathfrak{D}(R)$ is the differential graded Lie algebra Koszul dual to $R$, i.e.\ its Andr\'{e}--Quillen cohomology.

	For $F= \Def_X$, we obtain the Kodaira--Spencer Lie algebra 
	\[ T_{\Def_X} = R\Gamma(X,T_X)[1] 
	\]
	with the commutator bracket of vector fields, 
	and we have \[
	\pi_0 \Def_X(\sqz_k(k[n])) \cong  \pi_{-n}T_{\Def_X}  \cong \rH^{n+1}(X,T_X).\]
	
	This derived perspective gives rise to a   conceptually clear proof the  Bogomolov--Tian--Todorov theorem, which we briefly sketch for $k=\bC$. 
	
	First, one constructs a map of formal moduli problems
	\[
	\per_X\colon \Def_{X/\bC}\longrightarrow\Per_{X/\bC},
	\]
	which is a (formal) derived version of the classical period map \`a la Griffiths. It measures how the first part of the Hodge filtration varies when $X$ moves in a family. More precisely, the functor $\per_X$ sends a deformation $\fX/R$ to the morphism
	\[
	\Fil^1\RGamma_{\dR}(\fX/R)\longrightarrow \RGamma_{\dR}(\fX/R) 
	\simeq R\otimes_\bZ \RGamma(X(\bC),\bZ).
	\]
	Here the equivalence is an avatar of the Gauss--Manin connection.
	
	If $\omega_X$ is trivial, then one checks that the linear map
	\begin{equation}\label{injectivityintro}
		\pi_0 \Def_X(\sqz_k(k[n])) \longrightarrow\pi_0 \Per_X(\sqz_k(k[n]))
	\end{equation}
	is injective  for all $n \geq 0$.
	
	For $n=0$, this is the tangent  to the classical period map; its injectivity is the infinitesimal Torelli theorem for Calabi--Yau varieties. For $n=1$, we see  that obstruction classes for deformations of $X$  map injectively to \mbox{obstruction classes for $\Per_X$. }
	
	The period domain $\Per_X$ is a kind of `derived Grassmannian', and one easily verifies that it is unobstructed. It then follows from \eqref{injectivityintro} that the obstruction classes for deformations of $X$ must always vanish. 
	
	In fact,  $\Per_X$ is even     `derived unobstructed', i.e.\
	pro-represented by a  {free}
	commutative differential graded algebra. This is equivalent to the corresponding differential graded Lie algebra being  {abelian} (as its homotopy groups are finite-dimensional).\vspace{1pt}

	Using that $\pi_i \Def_X(\sqz_k(k[n]))\longrightarrow \pi_i\Per_X(\sqz_k(k[n]))$
	is injective  for all $i$ and all $n \geq 0$, one can show that $X$ is also derived unobstructed. Indeed, one  can postcompose the induced map on tangent fibres $T_{\Def_X} \rightarrow T_{\Per_X}$ with a map to an abelian differential graded Lie algebra such that the composite is an equivalence. This argument is closely related to the proof of Iacono--Manetti~\cite{IaconoManetti2010}, and its interpretation due to Fiorenza--Martinengo~\cite{FiorenzaMartinengo2012}.

	\subsubsection*{Formal moduli problems in  mixed characteristic}
	The theory of formal moduli problems and their tangent fibres admits a well-behaved generalisation to mixed characteristic \cite{BrantnerMathew2019}. When $k$ is a perfect field   of characteristic $p$, formal moduli problems $F$ are now defined on animated local Artinian $W(k)$-algebras with \mbox{residue field $k$.} Their tangent fibres $T_F$ are still  chain complexes over $k$, which can be equipped with 
	partition Lie algebra structures controlling the    formal moduli problems $F$.
	
	In particular, there is a formal moduli problem $\Def_{X/W}$ encoding the mixed characteristic derived infinitesimal deformations of a given   variety \mbox{$X$ over $k$.}
	
	When attempting to adapt the above proof to the mixed characteristic setting, one however encounters a fundamental obstacle first observed by Grothendieck~\cite{Grothendieck1968}:  the de Rham cohomology is no longer rigid, as the
	Gauss--Manin connection generally need not exist.
	The above definition of the period map therefore no longer makes sense. 
	We can circumvent this issue in two ways:
	\begin{enumerate}
		\item Construct a  period map   for deformations of $X$ over animated local Artinian rings \textit{with divided powers} on their maximal ideals;
		\item Construct a period map based on the unit root part of crystalline cohomology of $X$, which is an infinitesimal crystal (cf.\ \cite[Cor.4.11]{Ogus}).
	\end{enumerate}
	The first approach leads to a proof of our Unobstructedness \Cref{mainthm:btt}, while the second leads to a proof of our Serre--Tate \Cref{mainthm:Serre-Tate}.

	\subsubsection*{Unobstructedness in mixed characteristic}
	Let $R\rightarrow k$ be a local Artinian animated   ring  equipped with a divided power structure on its augmentation ideal (in the sense of Mao~\cite{Mao2021}). For every lift $\fX/R$, 
	the derived de Rham cohomology  $\RGamma_{\dR}(\fX/R)$ agrees with  the crystalline cohomology $\RGamma_\cris(X/R) $. 
	We  can therefore \mbox{send $\fX/R$ to}
	\[
	\Fil^1 \RGamma_{\dR}(\fX/R) \longrightarrow \RGamma_\cris(X/R )\simeq 
	R \otimes_{W(k)} \RGamma_\cris(X/W(k)),
	\]
	and this assignment  defines a map of `divided power formal moduli problems'
	\[
	\per^{\pd}_X\colon \Def^\pd_{X/W} \longrightarrow \Per^\pd_{X/W}.
	\]
	Here  $\Def^\pd_{X/W}$ and $\Per^\pd_{X/W}$ are obtained by restricting corresponding genuine  formal moduli problems $\Def_{X/W} $ and $ \Per_{X/W}$ to animated rings with divided powers, and $\Per_{X/W}$ is again a kind of `derived Grassmannian'.

	Under conditions~(\ref{item:BTT-HdR}) and~(\ref{item:BTT-no-crys-tors}) in \Cref{mainthm:btt}, we show that the induced maps 
	\[
	\pi_0 \Def_X^\pd(\sqz_k(k[n])) \longrightarrow \pi_0 \Per_X^\pd(\sqz_k(k[n]))
	\]
	are injective for all $n\geq 0$ and that and that $\Per_{X/W}$ is unobstructed. From this,  we deduce that deformations of $X$ are unobstructed along small maps $f\colon R'\rightarrow R$ of local Artinian divided power rings  with the property that the divided powers vanish on $\ker(f)$.
	Following an idea of Schr\"oer~\cite{Schroeer2003}  and using a result of Achinger--Suh \cite{AchingerSuh2023}, we then bootstrap from this partial unobstructedness to show that $\Def_X$ is unobstructed on all local  Artinian (commutative) rings with residue \mbox{field $k$.}
	
	\begin{remark} For this bootstrapping argument, it is crucial to consider mixed characteristic deformations, even if one is only interested in unobstructedness along maps of $k$-algebras. \end{remark}

	\subsubsection*{Serre--Tate theory for ordinary Calabi--Yau varieties}
	Another advantage of derived deformation theory is that it is easy to detect when a map $F\rightarrow G$ of formal moduli problems is an equivalence, as the tangent fibre functor is conservative. This means that $F\rightarrow G$ is an equivalence  if and only if 
	 the induced map on tangent fibres $T_F \rightarrow T_G$ is an equivalence, i.e.\ if 
	for all $n$ and $i$, the induced map
	\[
	\pi_i F(\sqz_k(k[n])) \longrightarrow \pi_i G(\sqz_k(k[n]))
	\]
	is an isomorphism of $k$-vector spaces. 
	
	This observation is key to our proof of \Cref{mainthm:Serre-Tate}. We briefly outline our construction for $k$ algebraically closed, and refer to  the main text for the general case.

	 Given a Bloch--Kato
	$1$-ordinary Calabi--Yau variety $X$ over $k$ and a lift $\fX$ to an animated local Artinian ring $R$ with residue field $k$, we construct a morphism 
	\[ R\Gamma(X,W\Omega_X^1)^{F=1} \longrightarrow H^d(X, W\cO_X)^{F=1} \otimes_{\bZ_p} \hat{\bG}_m(R)[1-d]\]

	Varying $R$, these arrows assemble to  a map of formal moduli problems
	\[
	\st_X\colon \Def_{X/W}  \longrightarrow \ST_X,
	\]
	where $\ST_X$ has the structure of a (derived) formal group of multiplicative type.  
	This is a derived version of a construction of Nygaard~\cite{Nygaard83} and Achinger--Zdanowicz~\cite{AchingerZdanowicz2021}. 
	
	We then prove that   if $X$ is Bloch--Kato $2$-ordinary,   the induced maps \[\pi_i \st_X(\sqz_k(k[n]))\] are isomorphisms for all $i, n \geq 0$, which implies that $\Def_{X/W} $ is equivalent to $\ST_X$ by conservativity of the tangent fibre.

	\subsection{Acknowledgements}
	We are grateful to Piotr Achinger, Bhargav Bhatt, Jeremy Hahn,  Joost Nuiten,  	Sasha Petrov, Ananth Shankar, and No\'{e} Sotto for various discussions related to the content of this paper.
	
	L.B.\ was supported as a Royal Society University Research Fellow at Oxford University  through grant URF$\backslash$R1$\backslash$211075 and by the Centre national de la recherche scientifique (CNRS) at Orsay.

	L.T.\ was supported by the European Research Council (ERC), grant
	864145, and by the Charles Simonyi Endowment at the Institute for Advanced Study.

	\section{Preliminaries}
	To set the stage, we   briefly recall  the cotangent complex formalism  of Andr\'e~\cite{Andre} and Quillen~\cite{Quillen}  in the level of generality needed for our later applications.
	
	\subsection{Adjoining sifted colimits}
	Homological algebra replaces modules by chain complexes   to     refine   classical functors like the tensor product by more refined constructions like Tor groups.
	Quillen's   homotopical algebra \cite{Quillen2}  refines functors defined on categories which are not abelian,  e.g.\  the category of commutative rings.
	
	Higher category theory gives a  clean construction of the required nonabelian derived $\infty$-categories. Indeed,  let $\cat{C}$ be a small $\infty$-category with finite coproducts, and write $\cat{P}_{\Sigma}(\cat{C})$ for the $\infty$-category of functors $\cat{C}^{\op} \rightarrow \cat{S}$ which preserve finite products.
	By \cite[Section 5.5.8]{LurieHTT}, the $\infty$-category $\cat{P}_{\Sigma}(\cat{C})$ is 
	freely generated by $\cat{C}$ under sifted colimits. This means that $\cat{P}_{\Sigma}(\cat{C})$ admits sifted colimits and for 
	every other $\infty$-category $\cat{D}$ with sifted colimits, precomposing with the 
	(fully faithful) Yoneda embedding $
	j\colon \cat{C} \rightarrow \cat{P}_\Sigma(\cat{C})$ 
	gives an equivalence
	\[
	\Fun_\Sigma(\cat{P}_\Sigma(\cat{C}),\cat{D}) \xrightarrow{\simeq}  \Fun(\cat{C},\cat{D}).
	\]
	Here $\Fun_\Sigma$ is the full subcategory of functors that preserve sifted colimits. 
	
	Hence every functor $f\colon \cat{C}  \rightarrow \cat{D} $ to an $\infty$-category with sifted colimits has an essentially unique sifted-colimit-preserving extension $F\colon \cat{P}_\Sigma(\cat{C}) \rightarrow \cat{D}$.

	\begin{example}[Connective chain complexes]\label{ex:connchain}
		For  $\cat{C}=\Mod_{\bZ}^{\ff} \subset \Ab$  the full subcategory of   finite free abelian groups,  $\cat{P}_\Sigma(\cat{C})$ is  equivalent to the $\infty$-category $\Mod_{\bZ}^\cn $
		of connective module spectra over the Eilenberg--MacLane spectrum of $\bZ$, which is in turn equivalent to the derived $\infty$-category 
		$ \cat{D}_{\geq 0}(\bZ)$ of  connective chain complexes  (cf.\ \cite[Remark 7.1.1.16, Corollary 7.1.4.15.]{LurieHA}).
		The full subcategory  of discrete objects is   equivalent to the ordinary category $\Ab$ of abelian groups.

	\end{example} 
	
	\begin{proposition}\label{prop:forgetful-chain}
		The functor $\Mod_{\bZ}^{\ff} \rightarrow \cat{S}$ sending $A$ to the underlying set   extends to a sifted-colimit-preserving  functor $\Omega^{\infty}: \Mod_{\bZ}^{\cn} \rightarrow \cat{S}$ 
		which is conservative and creates small limits.
	\end{proposition}

	We recall some basic properties of 
	the $\cat{P}_{\Sigma}$-construction (cf.\ \cite[Proposition 5.5.8.10]{LurieHTT}, \cite[Lemma 5.5.8.14]{LurieHTT}):
	
	\begin{proposition}Let $\cat{C}$ be a small $\infty$-category that admits finite coproducts. Then 
		\begin{enumerate}
			\item $\cat{P}_\Sigma(\cat{C})$ is presentable;
			\item for every $x\in \cat{C}$ the functor $\Map_{\cat{P}_\Sigma(\cat{C})}(x,-)$ preserves sifted colimits;
			\item every object in $\cat{P}_\Sigma(\cat{C})$ can be written as a sifted colimit of objects in $\cat{C}$.
		\end{enumerate}
	\end{proposition}

	\subsection{Animated commutative rings}
	Let us write  $\CRing$  for the ordinary category of (commutative) rings, and $\Poly\subset \CRing$ for its full subcategory spanned by polynomial rings $\bZ[x_1,\ldots,x_n]$.

	\begin{definition}[Animated   rings]
		The $\infty$-category of \emph{animated   (commutative) rings}  $\CAlg^{\an} :=\cat{P}_\Sigma(\Poly)$ is obtained by freely adjoining sifted colimits to $\Poly$. 
	\end{definition}

	Equivalently,  $\CAlg^{\an}$ can be obtained by
	inverting the weak equivalences in the category of simplicial commutative rings with its Quillen model structure. We  briefly review the basic properties of  $\CAlg^\an$,   referring to \cite[Chapter 25]{LurieSAG} for   details.
	
	By \cite[Remark 25.1.1.4]{LurieSAG}, the full subcategory $\CAlg^{\an, \heartsuit}\subset  \CAlg^{\an} $ spanned by all discrete objects is  equivalent to the ordinary category of commutative rings. 
	
	The $\infty$-category $\CAlg^{\an}$ is presentable; hence,  it admits small limits and colimits.  As every object in $\Poly$ is a   coproduct of finitely many \mbox{copies of $\bZ[x]$, we have:}
	\begin{proposition}[Forgetful functor]\label{prop:forgetful-an}
		The functor $\Poly \rightarrow \Mod_{\bZ}^\cn$ sending a polynomial ring to its underlying abelian group extends to a 
		sifted-colimit-preserving functor
		$
		U\colon \CAlg^{\an} \rightarrow \Mod_{\bZ}^\cn
		$
		which is conservative and it creates small limits and sifted colimits.
	\end{proposition}

	Given $A \in \CAlg^{\an}$, write  $\CAlg^{\an}_A := \CAlg^{\an}_{A/}$ the $\infty$-category of maps $A\rightarrow B$.

	\begin{construction}[$\CAlg^\an_A$ as functor of $A$]\label{constr:functor-CAlg}

		Consider the $\infty$-category
		\[
		\cat{C} := \Fun(\Delta^1,\CAlg^\an).
		\]
		The functor $q\colon \cat{C} \rightarrow \CAlg^\an$, $(A\rightarrow B) \mapsto A$ is evidently a cartesian fibration.
		Since	$\cat{C} $ admits pushouts, it is also a cocartesian fibration, and applying the straightening construction (cf.\ \cite[Section 3.2]{LurieHTT}) gives a functor 
		\[
		\CAlg^\an \rightarrow \widehat\Cat_\infty,\, A \mapsto \CAlg^\an_A.
		\]
		to the (very large) $\infty$-category $\widehat\Cat_\infty$ of not necessarily small $\infty$-categories. 
		
		By \cite[Corollary 5.2.2.5]{LurieHTT}, this  functor factors over the (non-full) subcategory $\PrL \subset \widehat\Cat_\infty$ of presentable $\infty$-categories and colimit-preserving functors.
	\end{construction}

	\subsection{Modules over animated   rings}
	Let $\Sp $  be the symmetric monoidal
	$\infty$-category of spectra $\Sp$   with the smash product $\otimes$ and its canonical $t$-structure   \cite{LurieHA}.
	Write $\CAlg$ for the $\infty$-category of   algebra objects in $\Sp$, i.e.\  $\bE_\infty$-ring spectra. For $A\in \CAlg$, we write $\Mod_A$ for the $\infty$-category of $A$-module spectra.

	Every ring gives  an $\bE_\infty$-ring spectrum; 
	the resulting functor $\Poly \rightarrow \CAlg$ extends to a sifted-colimit-preserving functor $\CAlg^{\an} \rightarrow \CAlg, A \mapsto A^\circ$. 
	
	If $A\in \CAlg^{\an}$,  we write $\Mod_A$ for the $\infty$-category of module spectra over  $A^\circ$.  We denote by $\Mod^\cn_A \subset \Mod_A$ the full subcategory of connective $A$-modules, and by $\Mod^\heart_A$ the full subcategory of discrete $A$-modules. 
	
	\begin{remark}
		If $A$ is an ordinary  ring,   the homotopy category of $\Mod_A$ is the usual derived category of $A$. 
		Writing  $  \Mod_A^{\ff}$ for the full subcategory of finite free $A$-modules $A^{\oplus n}$, the functor $\Mod_A^{\ff} \rightarrow \Mod^{\cn}_A$ gives an equivalence
		\mbox{$	\cat{P}_\Sigma(\Mod_A^{\ff}) \simeq \Mod^{\cn}_A.$}
	\end{remark}

	We briefly recall the formal construction of 
$\infty$-categories of pairs $(A,M)$ of associative algebra object $A$ and left  $A$-module object $M$. The details  will be particularly relevant 
 in our discussion of sheaves of modules in \Cref{sec:sheavesofmodules} below.
	\begin{construction}[Construction of module $\infty$-categories]\label{cons:LM}
		Assume we are given a symmetric monoidal structure on the $\infty$-category $\cat{C}$, i.e.\ a cocartesian fibration $\cat{C}^{\otimes} \rightarrow \Fin_\ast$ satisfying the Segal condition (cf.\ \cite[Definition 2.0.0.7]{LurieHA}) with  \mbox{$\cat{C}^{\otimes}_{\langle 1 \rangle} = \cat{C}$.}  
		
		There is a canonical functor $\Delta^{\op} \rightarrow \Fin_\ast$ sending $[n] = \{0 < \ldots < n\} $ to $\langle n \rangle = \{\ast, 1, \ldots n\}$, where we think of $i \in \langle n \rangle $ as the $i^{th}$ inequality sign in $[n]$. Note that a map of ordered sets  induces a map on inequality signs in the reverse direction.
		
		Pulling back $\cat{C}^{\otimes} \rightarrow \Fin_\ast$ along this map gives  a monoidal $\infty$-category 
		\mbox{$\cat{C}^{\circledast} \rightarrow \Delta^{\op}$}  
		(cf.\ \cite[Definition 1.1.2]{LurieDAGII}).
		By Section  2.1 of [op.cit],   $\cat{C}$ is left-tensored over itself, which is formalised by a categorical fibration $\cat{C}^{\circledast,L} \rightarrow \cat{C}^{\circledast}$ satisfying various properties.
	Definition 2.1.4 in [op.cit]
	then gives a module  $\infty$-category $$\Mod(\cat{C}) \subset \Fun_{\Delta^{\op}}(\Delta^{\op}, \cat{C}^{\circledast,L}).$$
		By \cite[Definition 2.3.3]{LurieDAGII}, the forgetful functor $\Mod(\cat{C}) \rightarrow \Alg(\cat{C})$ is a cartesian fibration, giving rise to a functor $\Alg(\cat{C}) \rightarrow\widehat{\Cat}_{\infty}$ sending an associative algebra object $A$ to   $\Mod_A$  and a map $f:A\rightarrow B$ to $f^\ast: \Mod_B\rightarrow \Mod_A$.
		
		Objects in $\Mod(\cat{C})$ are  pairs $(A,M)$  of associative algebra object $A\in \Alg(\cat{C})$ and left $A$-modules $M \in \Mod_A$. To give a morphism $(A,M) \rightarrow (B,N)$, we must specify a map  of associative algebras $f: A\rightarrow B$ and a map of $A$-modules $M \rightarrow f^\ast N$.
	\end{construction}
	\begin{remark}
		In the proof of \Cref{prop:moduleeq} below, it will be 
		helpful to directly realise $\Mod(\cat{C})$ as an $\infty$-category of algebras over an $\infty$-operad. To this end,  \cite[Section 4.2.1]{LurieHA} introduces  the $\infty$-operad of left modules $\cat{L}\cat{M}^{\otimes}$. The left-tensored structure  on $\cat{C}$ gives rise to
		a cocartesian fibration  $\cat{C}^{\circledcirc}\rightarrow \cat{L}\cat{M}^{\otimes}$, and  $\Mod(\cat{C})$ is equivalent to  $\cat{L}\cat{M}^{\otimes} $-algebras in $\cat{C}$, i.e.\ operadic sections of $\cat{C}^{\circledcirc}\rightarrow \cat{L}\cat{M}^{\otimes}$. 
	\end{remark}
	
	\begin{definition} \label{def:pairs}  Define the $\infty$-category of pairs $(A,M)$ with $A$ an animated   ring and $M$ an object in $\Mod_A$ as
		\[
		\CAlg^{\an}\Mod:=	\CAlg^{\an}\times_{\Alg(\Sp)}  \Mod(\Sp).
		\]
		Let $\CAlg^{\an}\Mod^\cn  $ be the full subcategory of those  $(A,M)$ with $M$ connective. \end{definition}

	The $\infty$-category $\CAlg^\an\Mod^\cn$ can alternatively be described as follows. Let $\PolyMod^{\ff}$ be the ordinary category of pairs $(A,M)$ with $A$ a polynomial ring  and $M$ a finite free $A$-module. A morphism in $\PolyMod^{\ff}$ is a pair consisting of a map of  rings $f\colon A\rightarrow A'$, and a map of $A$-modules $M\rightarrow M'$ (or equivalently a map of $A'$-modules $A'\otimes_A M \rightarrow M'$). Then by \cite[25.2.1.2]{LurieSAG}, we have an equivalence
	\[
	\cat{P}_\Sigma(\PolyMod^{\ff}) \simeq \CAlg^{\an}\Mod^\cn.
	\]
	The full subcategory $(\CAlg^{\an}\Mod^\cn)^{\heartsuit}$ of discrete objects is equivalent to the ordinary category of pairs $(A,M)$   of a discrete   ring $A$ and a discrete \mbox{$A$-module $M$.}
	
	\subsection{Cotangent complex}
	We   briefly review the construction of Andr\'e--Quillen's cotangent complex formalism, referring to 
	\cite[Section 25.3]{LurieSAG} for a more extensive treatment in the language of $\infty$-categories.
	
	Write $\ArrPoly \subset \Fun(\Delta^1,\CRing)$ for the full subcategory of arrows of the form
	\[
	\bZ[x_1,\ldots,x_n] \rightarrow \bZ[x_1,\ldots,x_n,y_1,\ldots,y_m].
	\]
	By \cite[Corollary 2.13]{Mao2021}, the functor $\ArrPoly \rightarrow \Fun(\Delta^1,\CAlg^{\an})$ gives \mbox{an equivalence}
	\[
	\cat{P}_\Sigma(\ArrPoly) \simeq \Fun(\Delta^1,\CAlg^{\an}).
	\]
	\begin{definition}[Cotangent complex]
		\label{def:cotangent}
		The sifted-colimit-preserving functor
		\[
		\underline{L} \colon \Fun(\Delta^1,\CAlg^{\an}) \rightarrow \CAlg^{\an}\Mod^{\cn}
		\]
	maps
		$A\rightarrow B$ in $\ArrPoly$ to the pair $(B,\Omega^1_{B/A})$. 
	\end{definition}
	\begin{notation} \label{not:extend}
		The composite
		$
		\Fun(\Delta^1,\CAlg^{\an}) \overset{\underline{L}}\longto
		\CAlg^{\an}\Mod \longto \CAlg^{\an}
		$
		is equivalent to the functor mapping $A\rightarrow B$ to $B$; this is clear on compact projective  generators and  extends under sifted colimits. \mbox{Hence, we can 
write $\underline{L}_{B/A}=: (B,L_{B/A}).$}\end{notation}
	
	\begin{definition}[Trivial square-zero extension]
		The functor
		\[
		\underline{\sqz} \colon \CAlg^{\an}\Mod^{\cn} \rightarrow 
		\Fun(\Delta^1,\CAlg^{\an})
		\]
		is the sifted-colimit-preserving extension of the functor mapping $(R,M)$ in $\PolyMod^{\ff}$ to $R \rightarrow R\oplus M$, where $R\oplus M$ is the   ring with $(r,x)(s,y)=(rs,ry+sx)$.
	\end{definition}
	
	We write $\underline{\sqz}(R,M) = (R \rightarrow \sqz_R M )$. Note that there is a natural augmentation map $\epsilon: \sqz_R(M)\rightarrow R$ in $\CAlg^{\an}$, functorial in $(R,M)$, obtained by applying $\sqz_R$ to the morphism $M \rightarrow 0$. It preserves sifted colimits and is given by the evident augmentation for objects in $\PolyMod^{\ff}$.
	
	\begin{proposition}\label{prop:cotangent-adjoint}
		The functor $\underline{\sqz}$ is right adjoint to $\underline{L}$.
	\end{proposition}
	
	\begin{proof}
		This is well-known. We include a proof that will serve as a model for the proof of \Cref{prop:pd-cotangent-adjoint}, which gives an analogous result for divided power rings. 
		
		By \cite[Proposition 5.2.2.8]{LurieHTT}, it suffices to construct a \textit{unit transformation} \[\underline{u}:\id \rightarrow \underline{\sqz}\circ \underline{L}.\] Note that every map of   rings  $A\rightarrow B$ sits in a commutative square
		\[
		\begin{tikzcd}
			A \arrow{d} \arrow{r} & B \arrow{d}{\iota + d} \\
			B \arrow{r}{\iota} & \sqz_B(  \Omega^1_{B/A})
		\end{tikzcd}
		\]
		where $\iota\colon B\rightarrow  \sqz_B(  \Omega^1_{B/A}) = B \oplus \Omega^1_{B/A}$ is the inclusion   and $d$ is the universal derivation.
		For $A\rightarrow B$ in $\ArrPoly$, the vertical maps  
		define a functorial  morphism
		\[
		\underline{u}: (A\rightarrow B) \rightarrow \underline{\sqz}(\underline{L}_{B/A}).
		\] Extending by sifted colimits, we obtain a transformation 
		$\underline{u}:\id \rightarrow \underline{\sqz}\circ \underline{L}.$  
		
		To verify  this is a unit transformation, we   check that for   $A\rightarrow B$ in $\Fun(\Delta^1,\CAlg^\an)$ and $(R,M)$ in $\CAlg^\an\Mod^\cn$, the following composite is an  {equivalence:}
		\[
		\varphi: \Map_{}\big(\underline{L}_{B/A},\,(R,M)\big) 
		\xrightarrow{	\underline{u} \circ \underline\sqz}
		\Map_{}\big((A\rightarrow B),\, \underline\sqz(R,M)\big).
		\]
		
		First, assume that $A\rightarrow B$ is of the form 
		\[
		\bZ[x_1,\ldots,x_n] \rightarrow \bZ[x_1,\ldots,x_n,y_1,\ldots,y_m].
		\]
		Then $L_{B/A}$ is the free $B$-module on generators $dy_i$. If  $(R,M)\in \PolyMod^{\ff}$, then the map $\varphi$  identifies with a map of discrete spaces
		$
		R^{n+m} \times M^m \rightarrow R^n \times (R\oplus M)^m$.
		Unravelling the definitions, we see that $\varphi$ is given by the obvious bijection.
		
		Writing  $(R,M)$ in $\CAlg^\an\Mod^\cn$ as a sifted colimit of objects in $\PolyMod^{\ff}$, we see that $\phi$ is an equivalence since both $A\rightarrow B$ and $\underline{L}_{B/A}$ are compact projective.
		
		Finally, we can write a general $A\rightarrow B$ in $\Fun(\Delta^1,\CAlg^\an)$ as a sifted colimit of morphisms in $\ArrPoly$ which shows that  $\varphi$ is an equivalence in general.
	\end{proof}
	\begin{notation}
		Denote the counit of the above adjunction by \mbox{$\underline{v}: \underline{L} \circ \underline{\sqz} \rightarrow \id$.}
	\end{notation}
	\begin{remark}\label{rem:firstcounit}
		The first component of the  counit  $ (\sqz_B(M), L_{\sqz_B(M)/B}) \xrightarrow{\underline{v}} (B,M)$ on some $(B,M) \in \CAlg^\an\Mod^\cn$
		is given by the augmentation $\epsilon:  \sqz_B(M) \rightarrow B$. 
		
		To prove this, let us observe  that the two functors $\CAlg^\an\Mod^\cn \rightarrow \CAlg^\an$ and {$\ev_1: \Fun(\Delta^1,\CAlg^{\an}) \rightarrow \CAlg^\an$} restrict to equivalences between 
		the  full subcategory of $\CAlg^\an\Mod^\cn$ spanned by all $(B,M)$ with $M \simeq 0$, respectively the full subcategory of $\Fun(\Delta^1,\CAlg^{\an})$ spanned by all equivalences, and $\CAlg^\an$. 
		
		Both $\underline{L}$ and  $\underline{\sqz}$ are equivalent to the  identity  on these two subcategories, and we deduce  that the counit on an object  $(B,0)$ in $\CAlg^\an\Mod^\cn$ is given by $\id_{(B,0)}$. Given a general object $(B,M)$, we use the naturality of the counit for the morphism $(B,M) \rightarrow (B,0)$ to deduce that the first component of the counit $\underline{v} :(\sqz_B(M), L_{\sqz_B(M)/B}) \rightarrow (B,M)$  is given by the augmentation $\epsilon: \sqz_B(M) \rightarrow B$.
	\end{remark}
	
	\begin{notation}
		We write $v: L_{\sqz_B(M)/B} \rightarrow \epsilon^\ast M$ for the map in $\Mod_{\sqz_B(M)}$ induced by the second component of the counit.
	\end{notation}

	\begin{corollary}\label{cor:derivations-sqz}
		Given $A\rightarrow B$ in $\CAlg^{\an}$ and $M\in \Mod_B^{\cn}$, there is an equivalence
		\[ 	\Map_{\CAlg^{\an}_{A//B}}(B,\,\sqz_B(M))    \simeq \Map_{B}(L_{B/A},\,M)\]
		sending  \vspace{2pt}  $\phi: L_{B/A} \rightarrow M$ to the composite $B \xrightarrow{\id \oplus d} \sqz_B(L_{B/A})  \xrightarrow{\sqz(\phi)} \sqz_B(M)$ and 
		sending $s: B \rightarrow \sqz_B(M)$ to
		\mbox{$L_{B/A} \rightarrow s^\ast L_{\sqz_B(M)/A} \rightarrow s^\ast L_{\sqz_B(M)/B}  \rightarrow s^\ast \epsilon^\ast M \simeq M $.}
	\end{corollary}
	
	\begin{proof}
		Let us denote the forgetful functor $\CAlg^\an\Mod^\cn \rightarrow \CAlg^\an$ by $F$. 
		
As in \Cref{not:extend}, the composite
		\[
		\Fun(\Delta^1,\CAlg^\an) \overset{\underline{L}}\rightarrow
		\CAlg^{\an}\Mod^\cn \overset{F}\rightarrow \CAlg^\an
		\]
		maps $A\rightarrow B$ to $B$. By \Cref{rem:firstcounit}, the map 
		\[
		(F \circ \underline{L}\circ \underline{\sqz})(B,M) \rightarrow F(B,M) 
		\]
		induced by the counit of adjunction is given by the augmentation $\epsilon: \sqz_B(M) \rightarrow B$.

		We obtain a commutative triangle
		\[
		\begin{tikzcd}[column sep=-32pt]
			\Map_{\Fun(\Delta^1,\CAlg^\an)}\big(
			A\rightarrow B,\, \underline{\sqz}(B,M)
			\big) \arrow{rr}{\simeq} \arrow{rd}
			&& \Map_{\CAlg^\an\Mod^\cn}\big(
			\underline{L}_{B/A},\, (B,M) 
			\big) \arrow{ld} \\
			& \Map_{\CAlg^{\an}}\big(B, B \big) &
		\end{tikzcd}
		\]
		where the horizontal map is the adjunction equivalence, the right diagonal is induced by $F$, and the left diagonal map is induced by applying $\ev_1$ and then composing with the augmentation 
		$\epsilon : \sqz_B(M)\rightarrow B$. 
		
		Passing to fibres of the diagonal maps over $\id_B$, we obtain the desired equivalence.
		Unravelling the definitions gives the explicit descriptions, using that for every morphism in 	$\Map_{\Fun(\Delta^1,\CAlg^\an)}\big(
		A\rightarrow B,\, \underline{\sqz}(B,M)
		\big)$, the induced map of $B$-modules $L_{B/A} \rightarrow s^\ast L_{\sqz_B(M)/B}$ factors canonically through 
		$L_{B/A} \rightarrow s^\ast L_{\sqz_B(M)/A}$.
	\end{proof} 
	
	We record several well-known properties of this construction (cf.\ \cite{Quillen, Quillen2}).
	\begin{proposition}[Transitivity sequence] \label{transitivity0}
		Given a sequence of animated rings $A \rightarrow B \rightarrow C$, there is a cofibre sequence $C\otimes_{B} L_{B/A} \rightarrow L_{C/A} \rightarrow L_{C/B}. $
	\end{proposition}
	\begin{proof} 
		As   $ \underline{L}: \Fun(\Delta^{1}, \CAlg^{\an}) \rightarrow \CAlg^{\an}\Mod$ preserves small colimits, this  follows   by applying it to the pushout of  
		$ ( A\rightarrow C) \leftarrow (A \rightarrow B ) \rightarrow (B \xrightarrow{\id} B).$
	\end{proof}
	\begin{proposition}[Base change] \label{cotcxbasechange0}
		Let $A\rightarrow A'$ be  a map in $\CAlg^{\an}$
		and		 \mbox{$B \rightarrow C$}   a map in $\CAlg^{\an}_A$. There is a natural equivalence
		\mbox{$ A'  \otimes_{A} L_{C/B} \xrightarrow{\simeq} L_{A'\otimes_{A} C/A'\otimes_{A} B}.$}
	\end{proposition}
	\begin{proof}
		This follows by applying the colimit-preserving functor  $ \underline{L}$ to the pushout of the arrows $	(B\rightarrow C ) \leftarrow (A\xrightarrow{\id } A )  \rightarrow (A'\xrightarrow{\id }A' ).$
	\end{proof}

	\begin{proposition}[Localisation]\label{localisation0}
		Given  $A \in \CAlg^{\an}$  and $ f\in \pi_0(A)$, we have  $L_{A[f^{-1}] /A } = 0, $ where 
		$A[f^{-1}] = A\otimes_{\bZ[t]}\bZ[t^{\pm 1}]$ and $\bZ[t]\rightarrow A$ sends $t$ to $f$.
	\end{proposition}
	
	\begin{proof}
		For any $A[f^{-1}] $-module $M$, we have an equivalence
		$A_1[f^{-1}]  \otimes_{A} M\simeq M $.
		Hence  $L_{A[f^{-1}] /A } \simeq A\left[f^{-1}\right]  \otimes_{A} L_{A[f^{-1}] /A }$. By  \Cref{cotcxbasechange0}, is is equivalent to 
		$L_{(A[f^{-1}]  \otimes_{A} A[f^{-1}]) /A\left[f^{-1}\right] } \simeq L_{   A[f^{-1}] /A[f^{-1}] }\simeq  0.$
	\end{proof}
	
	\subsection{Sheaves with values in an $\infty$-category}
	\label{sec:sheaves}
	We briefly review the theory of sheaves on sites,  (cf.\ \cite[Tome 1, Expos\'{e} II]{SGA4}, \cite[Definition 6.2.2.1]{LurieHTT}). The only examples strictly needed for our proofs are the small Zariski site $X_\zar$ and the small \'etale site $X_\et$ associated to a (classical) scheme $X$.

	Let $\cat{T}$ be a small $\infty$-category equipped with a Grothendieck topology (cf.\ \cite[Section 6.2.2]{LurieHTT}). 
	Given an $\infty$-category
	$\cat{C}$ with small limits, let \[\Shv(\cat{T}, \cat{C}) \subset \Fun(\cat{T}^{\op}, \cat{C})\] be  the full subcategory consisting of   the $\cat{C}$-valued sheaves on $\cat{T}$, see~\cite[Definition 1.3.1.1]{LurieSAG}. By \cite[Corollary 1.3.1.8]{LurieSAG}, there is an equivalence
	\[
	\Shv(\cat{T},\cat{C}) \simeq \Fun^{\rR}(\Shv(\cat{T},\cat{S})^{\op},\, \cat{C}).
	\]
	Here $\Fun^{\rR} \subset \Fun$ is the full subcategory of    functors which preserve small \mbox{limits.} Postcomposition with any  $F \in \Fun^{\rR}(\cat{C}, \cat{D})
	$ gives a functor $F\colon \Shv(\cat{T},\cat{C}) \rightarrow \Shv(\cat{T},\cat{D})$.
	
	By \cite[Proposition 4.8.1.15]{LurieHA}, $\infty$-category $\Pr^{\rL}$ of presentable $\infty$-categories and left adjoints admits a symmetric monoidal structure $\otimes$, the Lurie tensor product. We can use $\otimes$ to relate $\cat{C}$-valued sheaves and sheaves of spaces:
	\begin{proposition} \label{prop:presentable-cat-of-sheaves}
		If $\cat{C}$ is presentable then 
		\begin{enumerate}
			\item $\Shv(\cat{T},\cat{C})$ is presentable;
			\item $\Shv(\cat{T},\cat{C}) \simeq \Shv(\cat{T},\cat{S})\otimes \cat{C}$;
			\item the inclusion $\iota_{\cat{C}}: \Shv(\cat{T},\cat{C}) \rightarrow \Fun(\cat{T}^\op,\cat{C})$ admits a left adjoint, which is called 
			\emph{sheafification} and denoted by $\cat{F} \mapsto \cat{F}^\#$.
		\end{enumerate}
	\end{proposition}
	\begin{proof}
		By \cite[Proposition 6.2.2.7]{LurieHTT} the $\infty$-category $\Shv(\cat{T},\cat{S})$ is presentable,  and  \cite[Remark 5.5.2.10]{LurieHTT} shows that $\Fun^{\rR}(\Shv(\cat{T},\cat{S})^{\op},\, \cat{C})$ coincides with the full subcategory of functors admitting a left adjoint. 
		By \cite[Lemma 4.8.1.16]{LurieHA}, the $\infty$-category $\Shv(\cat{T},\cat{C}) \simeq \Fun^{\rR}(\Shv(\cat{T},\cat{S})^{\op},\, \cat{C})$ is therefore presentable, and by \cite[Proposition 4.8.1.17]{LurieHA}, it is equivalent to $\Shv(\cat{T},\cat{S}) \otimes  \cat{C}$. The third statement then follows as the inclusion $\Shv(\cat{T},\cat{S}) \rightarrow \Fun(\cat{T}^\op,\cat{S})$ admits a left adjoint.
	\end{proof}

	Given a general functor 
	$F\colon \cat{C}\rightarrow \cat{D}$  (not necessarily preserving small limits) with  $\cat{D}$ presentable,  we obtain an induced functor
	\[
	\Shv(\cat{T},\cat{C}) \subset \Fun(\cat{T}^{\op},\cat{C}) \overset{F}{\rightarrow}
	\Fun(\cat{T}^{\op},\cat{D}) \overset{\#}{\rightarrow} \Shv(\cat{T},\cat{D}),
	\]
	which we will abusively denote by $F\colon \Shv(\cat{T},\cat{C}) \rightarrow \Shv(\cat{T},\cat{D})$. 
	
	\begin{construction}\label{cons:sheavesadjunctions}
		Given an adjoint pair   $L\colon \cat{C} \leftrightarrows \cat{D} :\! R$
		between presentable $\infty$-categories. Postcomposing with $R$ preserves sheaves, and we obtain  adjoint squares   
		\[
		\begin{tikzcd}
			\Shv(\cat{T},\cat{C})\arrow{d}{\iota_{\cat{C}}}& \arrow{l}{R\circ} 	\Shv(\cat{T},\cat{D}) \arrow{d}{\iota_{\cat{D}}}\\
			\Fun(\cat{T}^{\op},\cat{C})&\arrow{l}{R\circ} \Fun(\cat{T}^{\op},\cat{D}).
		\end{tikzcd} \ \ \ \ \ \ 
		\begin{tikzcd}
			\Shv(\cat{T},\cat{C}) \arrow{r}{(L\circ-)^\#}&	\Shv(\cat{T},\cat{D}) \\
			\Fun(\cat{T}^{\op},\cat{C})\arrow{r}{L\circ} \arrow{u}{(-)^\#}& \Fun(\cat{T}^{\op},\cat{D})\arrow{u}{(-)^\#}.
		\end{tikzcd}
		\] 
		
		Note that the counit   $L \circ R \rightarrow \id_{ \Fun(\cat{T}^{\op},\cat{D})}$ of the adjunction on presheaves induces a counit transformation $(-)^{\#}\circ L \circ R \rightarrow \id_{ \Shv(\cat{T},\cat{D})}$
		for the adjunction on sheaves.
		
		We abuse notation and denote the resulting adjunction by 
		\[
		L\colon \Shv(\cat{T},\cat{C}) \leftrightarrows \Shv(\cat{T},\cat{D}) :\!R
		\]
		In fact, it follows from \Cref{prop:presentable-cat-of-sheaves} that $\Shv(\cat{T},-)$ defines a functor $\PrL \rightarrow \PrL$.
	\end{construction}
	
	For $X$ a topological space, write $\cat{U}(X)$ for the site of open subsets of $X$. By \cite[Proposition 1.1.4.4]{LurieSAG}, we can characterise sheaves on $X$ in terms of compact bases:
	
	\begin{proposition}\label{sheafesbases} Let  $\cat{U}_e \subset \cat{U}(X)$ be  a basis for the topology of $X$ such that $\cat{U}_e$ is closed under finite intersections and every open $U \in \cat{U}_e$ is compact.
		
		Then $\mathcal{F}: \cat{U}(X)^{\op} \rightarrow \mathcal{C}$ is a $\cat{C}$-valued sheaf 
		if and only if 
	\begin{enumerate}
		\item 
		$\mathcal{F}$ is right Kan extended from $\cat{U}_e$ and
		\item  \label{cond2} for all finite collections $U_1, \ldots, U_n$ of open subsets in $ \cat{U}_e$ with $\cup_{i=1}^n U_i  \in  \cat{U}_e$, the following  map  is an equivalence: $$\displaystyle \mathcal{F}\left(\cup_{i=1}^n U_i \right) \rightarrow \lim_{\emptyset \neq S \subset \{1,\ldots,n \}} \mathcal{F}\left(\cap_{i \in S} U_i \right).$$
		\end{enumerate}
	\end{proposition}
	\begin{notation} \label{basesheaf}
		We write  $\cat{P}_e(X,\cat{C}) =  \Fun(\cat{U}_e,\cat{C})$ and $\Shv_e(X,\cat{C})  \subset \cat{P}_e(X,\cat{C})$ for the full subcategory spanned by all functors satisfying condition  (\ref{cond2}) above. 
	\end{notation}
	Writing $\Ran: \cat{P}_e(X,\cat{C})  \rightarrow \cat{P}(X,\cat{C}) $ for the right Kan extension and $(-)^{\#}_e$ for the adjoint of the inclusion $\Shv_e(X,\cat{C}) \hookrightarrow  \cat{P}_e(X,\cat{C})$, we obtain by 
	\Cref{sheafesbases}:
	
	\begin{corollary}\label{sheafifyingonbases}
		In the setting of \Cref{sheafesbases}, there are adjoint squares 
		\[\begin{tikzcd}
			\Shv_e(X,\cat{C})   \arrow{r}{\simeq} \arrow[hookrightarrow]{d}{\iota^{\cat{C}}_{X,e}} &\Shv(X,\cat{C})    
			\arrow[hookrightarrow]{d}{\iota^{\cat{C}}_{X}}  \\
			\cat{P}_e(X,\cat{C})  \arrow{r}{\Ran } &\cat{P}(X,\cat{C}) .
		\end{tikzcd} \ \ \ \ \ \ \ 
		\begin{tikzcd}
			\cat{P}(X,\cat{C})  \arrow{r}{\res } \arrow{d}{\ (-)^{\#}}  &\cat{P}_e(X,\cat{C}) 
			\arrow{d}{\ (-)^{\#}_e}  \\
			\Shv(X,\cat{C})  \arrow{r}{\simeq } &\Shv_e(X,\cat{C}).
		\end{tikzcd}
		\] 
	\end{corollary}

	\subsection{Sheaves of modules}\label{sec:sheavesofmodules}
	We again fix a  small $\infty$-category $\cat{T}$ equipped with a Grothendieck topology. 
	The $\infty$-category of presheaves $\Fun(\cat{T}^{\op},\Sp)$ is equipped with the pointwise smash product. By \cite[Proposition 2.2.1.9]{LurieHA}, the $\infty$-category $\Shv(\cat{T},\Sp)$ of sheaves of spectra  
	inherits a  symmetric monoidal structure making the sheafification functor $ \Fun(\cat{T}^{\op},\Sp) \rightarrow  \Shv(\cat{T},\Sp)$ symmetric monoidal.

	\begin{proposition}\label{prop:calgsh}
		There is a canonical equivalence \[\Shv(\cat{T},\CAlg) \simeq\CAlg(\Shv(\cat{T},\Sp)) \] between sheaves of $\bE_\infty$-ring spectra and commutative algebras in sheaves of spectra.
	\end{proposition}
	
	\begin{proof}
		Currying  gives an equivalence (in fact isomorphism) of $\infty$-categories
		\[\Fun(\cat{T}^{\op},\CAlg) =\Fun(\cat{T}^{\op},\Fun_{\Fin_\ast}^{\oper}( \Fin_\ast, \Sp^{\otimes}))
		\]
		\[\ \ \simeq \Fun_{\Fin_\ast}^{\oper}( \Fin_\ast,  \Fun(  \cat{T}^{\op}, \Sp)^{\otimes})
		=\CAlg(\Fun(\cat{T}^{\op}, \Sp)), 
		\]
		where $\Fun^{{\oper}} $ denotes
		maps of $\infty$-operads and 
		$\Fun(\cat{T}^{\op}, \Sp)^{\otimes}$ encodes the 
		pointwise symmetric monoidal structure on  $\Fun(\cat{T}^{\op}, \Sp)$.
		
		As  the forgetful functor $\CAlg \rightarrow \Sp$ is conservative and preserves small limits (cf.\ \cite[Section 3.2.2]{LurieHA}), the above equivalence identifies $\Shv(\cat{T},\CAlg)$ with  the full subcategory of $  \Fun_{\Fin_\ast}^{\oper}( \Fin_\ast,  \Fun(  \cat{T}^{\op}, \Sp)^{\otimes})$ spanned by all  functors sending $\langle 1 \rangle$ to a sheaf of spectra. This is equivalent to $\CAlg(\Fun(\cat{T}^{\op}, \Sp))$ as 
		$\Shv(\cat{T},\Sp)^{\otimes } \rightarrow \Fun(\cat{T}^{\op}, \Sp)^{\otimes}$ is a map of $\infty$-operads by  \mbox{\cite[Proposition 2.2.1.9 (3)]{LurieHA}.}
	\end{proof}

	If $\cA \in \Shv(\cat{T},\CAlg^\an)$ is a sheaf of animated   rings, we let $\cA^\circ \in \Shv(\cat{T},\CAlg)$ be  the underlying sheaf of $\bE_\infty$-rings. It is obtained by 
	first postcomposing  with the   forgetful functor $\CAlg^\an\rightarrow \CAlg$
	and then sheafifying (cf.\ \cite[Proposition 1.3.5.10]{LurieSAG}).
	By \Cref{prop:calgsh}, we can interpret $\cA^\circ$ as an $\bE_\infty$-algebra in $\Shv(\cat{T},\Sp)$.

	Using \Cref{cons:LM},  we define
	\[
	\Mod(\cat{T},\cA) := \Mod_{\cA^\circ}(\Shv(\cat{T},\Sp)) = \{\cA\}\times_{\Alg(\Shv(\cat{T},\Sp))}\Mod(\Shv(\cat{T},\Sp)).
	\]
	This is a presentable stable $\infty$-category by \cite[Theorem 3.4.4.2]{LurieHA}. Write $\Mod^\cn( \cat{T},\cA)$ for the full subcategory spanned by those $\cM \in \Mod(\cat{T},\cA)$ such that the homotopy sheaves $\pi_n \cM \in \Shv(\cat{T},\Ab)$ vanish for all $n<0$.  This forms the connective part of a $t$-structure on the stable $\infty$-category $\Mod(\cat{T},\cA)$ by~\cite[Proposition 2.1.1.1]{LurieSAG}.
	
	Using \Cref{def:pairs}, we can give a   more  `algebraic' definition:
	\begin{proposition}\label{prop:moduleeq}
		There are canonical equivalences \begin{eqnarray*}
			\{\cA\} \times_{\Shv(\cat{T},\CAlg^\an)}
			\Shv(\cat{T},\CAlg^\an\Mod)     &\simeq& 	\Mod(\cat{T},\cA)\\
			\{\cA\} \times_{\Shv(\cat{T},\CAlg^\an)}
			\Shv(\cat{T},\CAlg^\an\Mod^\cn)   &\simeq& 	\Mod^\cn(\cat{T},\cA) ,
		\end{eqnarray*}
	\end{proposition}
	\begin{proof}
		Using the notation from \Cref{cons:LM}, we note that by  \cite[Proposition 2.2.1.9(3)]{LurieHA}, the inclusion  $\Shv(\cat{T},\Sp)^{\circledcirc} \rightarrow \Fun(\cat{T}^{\op},\Sp)^{\circledcirc}$ is a map of   $\infty$-operads over $\cat{L}\cat{M}^{\otimes}$. We can therefore conclude as in  
		\Cref{prop:calgsh} that there is an equivalence
		\[\Shv(\cat{T}, \Mod(\Sp))  \xrightarrow{\simeq } \Mod( \Shv(\cat{T}, \Sp)),\]
		which sits in a commutative square consisting of forgetful functors and the equivalence  $ \Shv(\cat{T},\Alg(\Sp))  \xrightarrow{\simeq }\Alg( \Shv(\cat{T}, \Sp))$.
		Passing to the fibres over the respective image of $\cA$ in $\Shv(\cat{T},\Alg(\Sp)) $ and $  \Alg( \Shv(\cat{T}, \Sp))$ gives the first equivalence.
		
		Replacing $\Sp$ by $\Sp^{\cn}$ in the above argument gives rise to an equivalence
		\[ \{\cA\}\times_{\Shv(\cat{T},\Alg(\Sp^{\cn})) }\Shv(\cat{T}, \Mod(\Sp^{\cn}))  \xrightarrow{\simeq }
		\{\cA\}\times_{\Alg( \Shv(\cat{T}, \Sp^{\cn}))}   \Mod( \Shv(\cat{T}, \Sp^{\cn})).\]
		The right hand side is equivalent to $\Mod^\cn(\cat{T},\cA)$ by \cite[Proposition 1.3.5.8]{LurieSAG}.
	\end{proof} 
	
	\begin{notation}
		We will write objects in $\Shv(\cat{T},\CAlg^{\an}\Mod^{\cn})$  as pairs $(\cA,\cM)$ where $\cA \in \Shv(\cat{T},\CAlg^{\an})$ is the image under the forgetful functor and $\cM \in \Mod^{\cn}(\cat{T}, \cA)$ is the image under the second equivalence in 
		\Cref{prop:moduleeq}.		
	\end{notation}
	
	\begin{warning}
		Even for $(X,\cO_X)$ a classical scheme,   $\Mod(X,\cO_X)$ is generally \textit{not} equivalent to $\cat{D}(\Mod_{\cO_X}^\heartsuit)$,  the derived $\infty$-category  of the classical abelian category of $\cO_X$-modules, as the latter   only contains the hypercomplete sheaves.
		
		This distinction disappears for example when $X$ is quasi-compact, quasi-separated, and of finite Krull dimension (cf.\ \cite[Theorem 2.1.2.2.]{LurieSAG}\cite[Theorem 3.12]{ClausenMathew2021}).
	\end{warning}
	
	\subsection{Cotangent complex for sheaves of animated   rings}
	\label{subsec:cotangent-sheaves}
	To extend the cotangent complex to maps of sheaves of animated   rings,  note that 
	by \S \ref{sec:sheaves}, the adjoint pair $(\underline{L},\,\underline{\sqz})$ of \Cref{prop:cotangent-adjoint} induces an adjunction
	\[
	\underline{L}: \Fun(\Delta^1,\Shv(\cat{T}, \CAlg^{\an})) 
	\leftrightarrows \Shv(\cat{T},\CAlg^{\an}\Mod^{\cn}) : \underline{\sqz}.
	\]
	The composite 
	$
	\Fun(\Delta^1,\Shv(\cat{T}, \CAlg^{\an}))  \overset{\underline{L}}\rightarrow
	\Shv(\cat{T},\CAlg^\an\Mod^{\cn}) \overset{F}\rightarrow 
	\Shv(\cat{T},\CAlg^\an)
	$
	maps $\cA\rightarrow \cB$ to $\cB$, and we can write $\underline{L}_{\cB/\cA} = (\cB,L_{\cB/\cA})$ with $L_{\cB/\cA} \in \Mod_{\cB}$. Similarly, we write $\underline\sqz(\cB,\cM) = ( \cB \rightarrow \sqz_{\cB} \cM )$. 
	
	We denote the counit of this adjunction by $\underline{v}\colon \underline{L}  \circ \underline{\sqz} \rightarrow \id$. 
	Given some $(\cB,\cM)$ in $\Shv(\cat{T},\CAlg^{\an}\Mod^{\cn})$, the counit $(\sqz_{\cB} \cM, L_{\sqz_{\cB} \cM/\cB}) \rightarrow (\cB, \cM)$
	is obtained by
	sheafifying 
	the domain of the counit of the corresponding adjunction of presheaves.  We deduce from the \Cref{rem:firstcounit}
	that the `first component' of $\underline{v}$ is the canonical augmentation $\epsilon: \sqz_{\cB}(\cM) \rightarrow \sqz_{\cB}(0) \simeq \cB$, and the second component gives a map $v \colon L_{\sqz_{\cB} \cM} \rightarrow \epsilon^\ast \cM$. We can now sheafify \Cref{cor:derivations-sqz}:
	
	\begin{proposition}\label{prop:cotangent-sheaf-adjunction}
		Given a map $\cA\rightarrow \cB$ in $\Shv(\cat{T},\CAlg^{\an})$ and $\cM\in \Mod^{\cn}_\cB$, there is an equivalence
		\[ 		\Map_{\Shv(\cat{T},\CAlg^{\an})_{\cA//\cB}}\big(
		\cB, \sqz_\cB (\cM)
		\big)  \simeq 	\Map_{{\cB}}(L_{\cB/\cA},\, \cM)\]
		sending  \vspace{2pt}  $\phi: L_{\cB/\cA} \rightarrow \cM$ to the composite $\cB \xrightarrow{\id \oplus d} \sqz_{\cB}(L_{\cB/\cA})  \xrightarrow{\sqz(\phi)} \sqz_{\cB}(\cM)$ and 
		$s: \cB \rightarrow \sqz_{\cB}(\cM)$ to
		\mbox{$L_{\cB/\cA} \rightarrow s^\ast L_{\sqz_\cB(\cM)/\cA} \rightarrow s^\ast L_{\sqz_\cB(\cM)/\cB}  \rightarrow s^\ast \epsilon^\ast \cM \simeq \cM $.}
	\end{proposition}
	
	\begin{proof}
		As in the proof of \Cref{cor:derivations-sqz},
		we obtain a 
		commutative triangle
		\[
		\begin{tikzcd}[column sep=-16pt]
			\Map\big(
			\cA\rightarrow \cB,\, \underline{\sqz}(\cB,\cM)
			\big) \arrow{rr}{\simeq} \arrow{rd}
			&& \Map\big(\,
			\underline{L}_{\cB/\cA},\, (\cB,\cM) 
			\big) \arrow{ld} \\
			& \Map\big(\,\cB, \cB \big) &
		\end{tikzcd}
		\]
		in which the horizontal map is the adjunction equivalence, and the diagonal maps are induced by the map $\sqz_{\cB}(\cM) \rightarrow \cB$ and by $F$ respectively.  Considering fibres over $\id_{\cB}$ then proves the proposition.
	\end{proof}

	\begin{notation} \label{def:cotangent-complex-X}
		Let $A$ be an animated   ring. Given a pair $X=(\cat{T},\cO_X)$ of a Grothendieck site $\cat{T}$ and a sheaf $\cO_X$ in $\Shv(\cat{T},\CAlg^\an_A)$, we   write $L_{X/A} := L_{\cO_X/A}.$
	\end{notation}

	\subsection{Cotangent complex for derived schemes}
	We will now check that \Cref{def:cotangent} and  \Cref{def:cotangent-complex-X}  are compatible with each other. 
	Let us write $\Topp_{{\CAlg^{\an}}}$ for the $\infty$-category of \textit{animated ringed spaces}, using \cite[Construction 1.1.2.2]{LurieSAG}.
	Its objects are  given by  pairs $(X \in \Topp, \ \cO_X \in  \Shv(X,{\CAlg^{\an}}) )$ and   morphisms $(X, \cO_X) \rightarrow (Y, \cO_Y)$  are
	given by pairs  
	$ ( f:X\xrightarrow{} Y \ ,   \  \phi:
	\cO_Y  \xrightarrow{}   f_\ast(\cO_X))$. 
	
	\begin{definition}[Derived schemes]\label{def:dersch} A \emph{derived scheme}  is a pair $X=(X,\cO_X)$ with $X$ a topological space and $\cO_X$ an object in $\Shv(X,\CAlg^\an)$ such that
		\begin{enumerate}
			\item $(X,\pi_0 \cO_X)$ is a scheme;
			\item the $\pi_0\cO_X$-module $\pi_n \cO_X$ is quasi-coherent for all $n\geq 0$;
			\item $\cO_X$ is hypercomplete.
		\end{enumerate}
		A morphism $ X\rightarrow Y$ of derived schemes   is a pair consisting of a continuous map $f\colon X\rightarrow Y$ and a map $ \phi: \cO_Y \rightarrow f_\ast(\cO_X)$ in $\Shv(X,\CAlg^\an )$ such that the induced map
		$(X,\pi_0 \cO_X) \rightarrow (Y,\pi_0 \cO_Y)$ is a morphism of schemes.  
		
		The $\infty$-category of derived schemes is the  subcategory $\Sch^{\der} \subset \Topp_{{\CAlg^{\an}}}$  of derived schemes and maps  of derived schemes between them.
	\end{definition}
	
	\begin{remark} If $(X,\cO_X)$ satisfies (1) and (2) in \Cref{def:dersch}, then  (3) is equivalent to the assertion that for every affine open $U\subset (X,\pi_0\cO_X)$ the natural map $\pi_n(\cO_X(U)) \rightarrow (\pi_n \cO_X)(U)$
		is an isomorphism, cf.\ the proof of \mbox{\cite[Proposition 1.1.6.2]{LurieSAG}.}\end{remark}

	\begin{example}[Affine derived schemes]\label{affinederived}
		Every animated ring $B\in \CAlg^{\an}$ gives  a derived scheme $\Spec(B) = (|\Spec(B)|, \cO_{\Spec(B)})$, whose  structure sheaf maps a distinguished open $U_f = \Spec(\pi_0(B)[f^{-1}])$ to  $B[f^{-1}]$  
		
		The resulting assignment $\CAlg^{\an} \rightarrow \Sch^{\der}$, $B \mapsto \Spec(B)$ is fully faithful.
	\end{example} 
 
	\begin{remark} Given a derived scheme $X$ over an animated ring $A$, we 
		define $L_{X/A}$ following \Cref{def:cotangent-complex-X}. Note that any map $A' \rightarrow A$ in $\CAlg^{\an}$  induces  a map $L_{X/A} \rightarrow L_{X/A'}$. With a bit more work, one  can in fact define $L$ as a functor on arrows of animated ringed spaces.
	\end{remark} 
	
	\begin{definition}[Quasi-coherence]
		A sheaf of $\cO_X$-modules $\mathcal{M}$  on a derived scheme $(X, \cO_X)$ is \textit{quasi-coherent} if for all inclusions of \textit{affine} open subsets $U \subset V$ of $X$, the following canonical morphism is an equivalence:\[ \cO_X(U) \otimes_{\cO_X(V)} \mathcal{M}(V) \longrightarrow\cO_X(U) \otimes_{\cO_X(U)} \mathcal{M}(U) \longrightarrow \mathcal{M}(U).\]
	\end{definition}
	
		\begin{warning}
				The assignment $U \mapsto \cO_X(U)$ is usually not a sheaf of $\cO_X$-modules, as the forgetful  functor $\CAlg^{\an} \rightarrow \Mod_{\bZ}$ does not preserve limits. However, its underlying sheaf of $\EE_\infty$-rings
				$\cO_X^{\circ}$ (obtained by sheafification) is a quasi-coherent sheaf of $\cO_X$-modules, which agrees with $\cO_X$ on affines opens.
	\end{warning}

	\begin{remark}
		By \cite[2.2.6.1]{LurieSAG}, the $\cO_X$-module $\mathcal{M}$ is quasi-coherent if and only if it is hypercomplete and all its homotopy sheaves are quasi-coherent $\pi_0\cO_X$-modules.
	\end{remark}
	
	\begin{proposition}\label{affineqccrit}
		Let $X=  \Spec(B)$ be an affine derived scheme.
		Then a sheaf $\mathcal{M}$ of $\cO_X$-modules is quasi-coherent if and only if for all  distinguished affine opens  $U' \subset V' $ of $X$, the    morphism $$ \cO_X(U') \otimes_{\cO_X(V')} \mathcal{M}(V') \longrightarrow \mathcal{M}(U')$$ is an equivalence.
	\end{proposition}  
	\begin{proof}
		Fix an inclusion of affine open subsets $U \subset V$ of $ \Spec(B)$. 
		
		First assume $V$ is already distinguished.  We can cover it by finitely many distinguished affine opens $U_1', \ldots , U_m'$.
		Proceeding as in   \cite[Proposition 1.1.4.4]{LurieSAG}, 
		we see that the natural map $\mathcal{M}(U) \rightarrow \lim_{\emptyset \neq S \subset \{1,\ldots,m \}}  \mathcal{M}(U_S' )$ is an equivalence,  where $U_S' = \bigcap_{i \in S} U'_i$, and a similar expression holds for  $\cO_X(U)$.
		As tensor products commute with finite limits, we can  write 
		$ \cO_X(U) \otimes_{\cO_X(V)} \mathcal{M}(V) \longrightarrow \mathcal{M}(U)$
		as a limit   of maps
		$   \cO_X(U_i' ) \otimes_{\cO_X(  V)} \mathcal{M}( V) \rightarrow \mathcal{M}(U_i')  
		$, which shows the result for   $V$  distinguished.

		We  cover general inclusions of affines by finitely many distinguished affines.
	\end{proof}

	\begin{proposition} \label{cotforaffine}
		Fix a  map of animated   rings $q: A \longrightarrow B $. Then  $L_{\Spec (B)/A }$ (cf.\ \Cref{def:cotangent-complex-X}) is the quasi-coherent sheaf on $\Spec(B) $ with global sections 
		$L_{B/A }$.
	\end{proposition} 
	\begin{proof} 
		Write $\cO_R = \cO_{\Spec(R)}$ for $R \in \CAlg^{\an}$.
		First,  note that for any open $V \subset \Spec(A)$, we have 
		$ L_{\cO_{A}(V)/A} \simeq  0 $. 
	 
		Indeed, we can cover  $V$ by distinguished affines  $U_f \subset V$, and it
		suffices to check that for any such $U_f$, we have $\cO_{A}(U_f) \otimes_{\cO_{A}(V)} L_{\cO_{A}(V)/A} = 0 $. 
		This follows from   \Cref{transitivity0},
		\Cref{localisation0}, and the equivalence $\cO_{A}(V)[f^{-1}] \simeq A[f^{-1}]$, which we can show using the limit cube for  $\cO_{A}(V)$ induced by the distinguished \mbox{cover of $V$.}

		Given  $U \subset \Spec(B)$ open,  
		$(q^{-1}\cO_{A})(U)$ is a filtered colimit of animated  rings $\cO_{A}(V)$, where $V$ is an open containing $q(U)$.
		Hence $L_{(q^{-1}\cO_{A})(U)/A} =0$,  
		and  \Cref{transitivity0}  gives an equivalence
		$ L_{\cO_{B}(U)/A} \xrightarrow{\simeq } L_{\cO_{B}(U)/(q^{-1}\cO_{A})(U) }.$

		Together with Propositions \ref{localisation0} and \ref{transitivity0}, this  shows that for any inclusion $U_f \subset U_g $ of distinguished affines, the following square consists of equivalences: 
		\[	\begin{tikzcd}
			B[f^{-1}]\otimes_{B[g^{-1}]}  
			L_{B[g^{-1}]/A }  \arrow{r}{  } \arrow{d}{  }   & L_{B[f^{-1}]/A }  \arrow{d}{  }  \\
			\cO_{B}(U_f) \otimes_{\cO_{B}(U_g)}  
			L_{\cO_{\Spec(B)}(U_g)/(q^{-1}\cO_{A})(U_g) } \arrow{r}{ } & L_{\cO_{\Spec (B)}(U_f)/(q^{-1}\cO_{A})(U_f) }.
		\end{tikzcd}
		\] 
		
		Let $\cat{U}_e   $ be the poset of distinguished   open subsets of $\Spec(B)$.
		Using the above formula for $g=1$, we see  that  the restriction of the  presheaf $U \mapsto  \underline{L}_{\cO_{B}(U_f)/(q^{-1}\cO_{A})(U_f) }$  to $\cat{U}_e$
		satisfies  \Cref{sheafesbases}(2).
		\Cref{sheafifyingonbases} shows that  sheafification preserves the values  on    distinguished  affines $U_f \subset \Spec(B)$, so we obtain:
		$$\underline{L}_{\Spec (B)/\Spec (A) } (U_f) \simeq  (\cO_{B}(U_f),  {L}_{\cO_{B}(U_f)/(q^{-1}\cO_{A})(U_f) }) \simeq (B[f^{-1}],  {L}_{B[f^{-1}]/A }).$$

		\Cref{affineqccrit} now  implies that $ {L}_{\Spec (B)/\Spec (A) }$ is quasi-coherent.
	\end{proof}

	\begin{corollary}[Quasi-coherence of $L_{X/A}$]
		Given a derived scheme $X$ over an animated ring $A$,  
		the $\cO_X$-module  
		$L_{X/A}$ is   quasi-coherent.
	\end{corollary}
	\begin{proof}
		As the assertion is local on $X$, this follows from \Cref{cotforaffine}.
	\end{proof}
	\begin{remark}
		In particular, this shows that $L_{X/A}$  is hypercomplete. 
		For $X$ a classical scheme, we deduce that $L_{X/A}$ belongs to $\cat{D}(\Mod_{\cO_X}^\heartsuit)$, and it is not hard to verify that it agrees with the classical cotangent complex of $X$.
	\end{remark}
	We will  compare the cotangent complex in the Zariski and the \'etale site.   
Let $X$ be a derived scheme over $A \in \CAlg^{\an}$. Fix an uncountably strong \mbox{limit cardinal $\kappa > |X|$.}
	\begin{definition}[Small \'etale site] \label{def:smet}
		The \textit{small \'etale site} $X_{\et}$ consists of all
		maps of derived schemes $U \rightarrow X$ with $  |U|< \kappa$ which are 
		 \textit{\'{e}tale}, i.e.\ locally obtained by applying $\Spec(-)$ to an \'etale map of animated rings	(see  \cite[Definition 3.4.7]{LurieDAG}).
		 The coverings of this site are  jointly surjective families.
	\end{definition}	
	
	Let us temporarily denote 
	the \'{e}tale structure sheaf of $X$ by 
$\cO_{X,\et}$. Given an \'{e}tale open $(U \rightarrow X)$, we have    $\cO_{X,\et} (U \rightarrow X) \simeq  \cO_{U}(U)$.
	Every  $\cO_X$-module 
	$\mathcal{F}$ on $X$ gives   an  \'{e}tale sheaf of $\cO_{X,\et}$-modules
	$\iota^{\ast} \mathcal{F} $ on $X_{\et}$,   by  sheafifying the presheaf $$(U \xrightarrow{h} X) \mapsto \cO_U(U) \otimes_{\cO_X(h(U))}\mathcal{F}(h(U)).$$
 \Cref{def:cotangent-complex-X} gives a priori different cotangent complexes $L^{\zar}_{X/A}$ and $L^{\et}_{X/A}$ for the Zariski and \'etale topology.

	\begin{proposition}  \label{prop:cotzaret}
		There is a canonical equivalence $ i^{\ast}L_{X/A}^{\zar} \xrightarrow{\simeq} L_{X/A}^{\et}.$
	\end{proposition}
	\begin{proof}
		For any \'etale open $h:U\rightarrow X$, the pullback  map $\cO_{X}(h(U)) \rightarrow \cO_{U}(U)$ is \'etale, so the canonical map $$\cO_{X, \et}(U) \otimes_{\cO_X(h(U))}L_{\cO_X(h(U))/A} \rightarrow L_{\cO_{X, \et}(U)/A} $$ is an equivalence by \Cref{transitivity0}. The claim follows by sheafifying.
	\end{proof}

\newpage

\section{Formal moduli problems} 

In this section, we will review the basics of derived deformation theory, and introduce several examples of  formal moduli problems along the way.

\subsection{Artinian animated   rings}
Fix a  local Noetherian   ring $\Lambda$ with \mbox{residue field $k$.}
\begin{definition}[Artinian animated $\Lambda$-algebras] \label{def:artinian-animated-ring}
An object $A $ of $\CAlg_{\Lambda//k}^{\an}$ is said to be \emph{Artinian} if 
\begin{enumerate}
	\item $\pi_0(A)$ is a local Artinian ring with residue field $k$;
	\item $\bigoplus_{i \geq 0}
	\pi_i(A)$ is a $	\pi_0(A)$-module of finite length.
\end{enumerate}
Let $\CAlg_{\Lambda}^{\an,\art} \subset \CAlg^{\an}_{\Lambda//k}$  be the full subcategory spanned by all Artinian objects. 
\end{definition}

If $A\rightarrow B$ is a map in $\CAlg^{\an,\art}_\Lambda$ with fibre $J$, then the $\Lambda$-module $\oplus_i \pi_i(J)$ has finite length. We denote its length by $\len(A\rightarrow B)$. 
Recall the following   notion:
\begin{definition}[Discrete elementary  extension]
Let $A\rightarrow B$ be a map of discrete rings in $\CAlg_{\Lambda}^{\an,\art}$. Denote the maximal ideal of $A$ by $\fm_A$ and set $I = \ker(A\rightarrow B)$. We say that $A\rightarrow B$ is a \emph{elementary  extension} if $A\rightarrow B$ is surjective and $\fm_A I = 0$. 
\end{definition}
The following two statements are well known \mbox{(\cite[Lemma 6.2.6]{LurieDAG}
\cite[Lemma 2.8]{GalatiusVenkatesh2018}):}

\begin{lemma}\label{lemma:inductive-small-extension}
Given a map $A\rightarrow B$  in $\CAlg_{\Lambda}^{\an,\art}$ which is not an equivalence and such that  $\pi_0(A) \rightarrow \pi_0(B)$ is surjective,    there is a finite-dimensional $k$-vector space $V$, an integer $n\geq 0$,
and a commutative diagram
\[
\begin{tikzcd}
	A \arrow[rd] 
	\arrow[rrd, bend left=15, "\text{ }"]
	\arrow[rdd, bend right=15, "\text{ }"'] & \\
	& B' \arrow[r] \arrow[d] \arrow[rd, phantom, "\lrcorner", at start] & B \arrow[d] \\
	& k \arrow[r] &\sqz_k(V[n+1])
\end{tikzcd}
\]
in $\CAlg^{\an,\art}_\Lambda$ such that
\begin{enumerate}
	\item the square is a pullback square;
	\item $\pi_0(A) \rightarrow \pi_0(B')$ is surjective;
	\item $\len(A\rightarrow B')< \len(A\rightarrow B)$.
\end{enumerate}
If $A\rightarrow B$ is a map of discrete rings, then one can take $n=0$ and $B'$ discrete.  If $A\rightarrow B$  is moreover an elementary  extension, then  one can take $B'=A$.
\end{lemma}

\begin{proof}
Let us denote the (connective) fibre of $A\rightarrow B$ by $J$.
Let $n$ be the smallest integer such that $\pi_n(J)$ is non-zero. The connectivity estimate of \cite[Proposition 25.3.6.1]{LurieSAG} shows that the natural map $B\otimes_A J[1] \rightarrow L_{B/A}$ induces an isomorphism
\[
M[n+1]:= \left(\pi_0(B)
\boldsymbol{\otimes}_{\pi_0(A)} \pi_n(J)\right)[n+1]\xrightarrow{\ \simeq \ }
\tau_{\leq n+1} L_{B/A},
\]
where $
\boldsymbol{\otimes}$ denotes the underived tensor product.
The Artinian $\pi_0(B)$-module $M$ admits a nonzero surjective $\pi_0(B)$-linear map $M\rightarrow V$ to a $k$-vector space $V$ (seen as $\pi_0(B)$-module via the augmentation $\pi_0(B)\rightarrow k$). The resulting composite 
\[
L_{B/A} \rightarrow \tau_{\leq n+1} L_{B/A} \simeq M[n+1] \rightarrow
V[n+1]
\]
induces a map of animated    $A$-algebras $B \rightarrow \sqz_B(V[n+1]) \rightarrow \sqz_k(V[n+1]) $. We now define $B'$ as  the pullback of this map against the unit $k\rightarrow \sqz_k(V[n+1])$ and obtain a diagram as in the statement of the lemma. 

Writing  $J'$ for the fibre of $A\rightarrow B'$,  we have a fibre sequence
$
J' \rightarrow J \rightarrow V[n]
$
and the map $J\rightarrow V[n]$ is surjective on $\pi_n$.  
It follows that $J'$ is connective (hence $\pi_0(A) \rightarrow \pi_0(B')$ surjective) and that $\len(A\rightarrow B')<\len(A\rightarrow B)$. If $A$ and $B$ are both discrete, then we have $n=0$ and it follows that $J'$ and hence $B'$ are also both discrete. If moreover $A\rightarrow B$ is an elementary  extension, then $M$ is a $k$-vector space, and so we can take $V=M$ and find  that $A\rightarrow B'$ is an isomorphism.
\end{proof}

\begin{corollary}\label{cor:artinian-small}
For every $A\rightarrow B$ in $\CAlg_{\Lambda}^{\an, \art} $ with $\pi_0(A)\rightarrow \pi_0(B)$ surjective, there exists a factorisation
\[
A=A_n \rightarrow \cdots \rightarrow A_1 \rightarrow A_0=B
\]
where each of the maps sits in a pullback square
\[
\begin{tikzcd}
	A_{i+1} \arrow{r} \arrow{d} \arrow[rd, phantom, "\lrcorner", at start] & A_{i} \arrow{d}\\
	k \arrow{r} &  \sqz_k (V_i[n_i+1])
\end{tikzcd}
\]
with $n_i\geq 0$. If $A, B$ are discrete, then we can pick $A_i$ discrete
and  $n_i=0$   for all $i$.
\end{corollary}

\subsection{Deformation functors and formal moduli problems} 
Let us denote the ordinary category of   local Artinian $\Lambda$-algebras with residue field $k$ by 
$\CRing^\art_\Lambda.$ We recall
Schlessinger's classical 
definition of a deformation functor from \cite{Schlessinger1968}:

\begin{definition}[Deformation functors]
\label{def:schlessinger}
A \emph{deformation functor} is a functor
\[
F\colon \CRing_\Lambda^\art \rightarrow \Set
\]
such that 
\begin{enumerate}
	\item \label{item:S1} $F(k)$ has exactly one element;
	\item \label{item:S2} for every pullback square  		
	\[
	\begin{tikzcd}
		A' \arrow{r}{}  \arrow{d}{}\arrow[rd, phantom, "\lrcorner", at start] &A \arrow{d}{} \\
		B'  \arrow{r}{} & B
	\end{tikzcd}
	\]
	in $\CRing_\Lambda^\art$ with $A\rightarrow B$ surjective, the following induced map 	is surjective: 
	\[
	\varphi\colon F(A') \rightarrow F(A)\times_{F(B)} F(B').
	\]
	
	\item \label{item:S3} if in (\ref{item:S2}) we take $B'=k[\epsilon]/\epsilon^2$ and $B=k$, then the map $\varphi$ is bijective. 
\end{enumerate}
\end{definition}

We define a derived version of a   deformation functor, following \cite{LurieDAGX,LurieDAGXII,BrantnerMathew2019}:

\begin{definition}[Formal moduli problems]
\label{def:fmp}
A functor
\[
F: \CAlg_{\Lambda}^{\an, \art} \rightarrow \cat{S}
\]
is called a \emph{formal moduli problem} if
\begin{enumerate}
	\item \label{item:FMP1}
	$F(k)$ is contractible;
	\item \label{item:FMP2} for every pullback square  		
	\[
	\begin{tikzcd}
		A' \arrow{r}{} \arrow{d} \arrow[rd, phantom, "\lrcorner", at start]&A \arrow{d} \\
		B'  \arrow{r} & B
	\end{tikzcd}
	\]
	in $\CAlg^{\an, \art}_\Lambda$ for which the maps $\pi_0 A \rightarrow \pi_0 B $ and $\pi_0 B' \rightarrow \pi_0 B$ are surjective,   applying $F$ gives a pullback square in $\cat{S}$.
\end{enumerate}
We write $\Moduli_{\Lambda}^{\an} \subset \Fun(\CAlg_{\Lambda}^{\an, \art}, \cat{S})$ for the full subcategory spanned by the formal moduli problems. 
\end{definition}

\begin{lemma}\label{lemma:fmp-underlying-classical}
If $F\colon \CAlg^{\an,\art}_\Lambda \rightarrow \cat{S}$ is a formal moduli problem, then 
\[
\CRing^{\art}_\Lambda \injto \CAlg^{\an,\art}_\Lambda \overset{F}\rightarrow \cat{S} \xrightarrow{\pi_0}   \Set
\]
is a deformation functor in the sense of \Cref{def:schlessinger}.
\end{lemma}
\begin{notation}\label{not:fmp-underlying-classical}
We will denote the deformation functor associated to a formal moduli problem $F$ by $F^\cl \colon \CRing^\art_\Lambda \rightarrow \Set$. We can think of $F$ as a derived enhancement of the deformation functor $F^{\cl}$.
\end{notation}

\begin{proof}[Proof of \Cref{lemma:fmp-underlying-classical}]
The first condition in \Cref{def:schlessinger} is immediate. For the second condition, one uses the reasoning of \cite[Proposition 12.1.3.7]{LurieSAG} to show that $F$ in fact satisfies the pullback property in \Cref{def:fmp} when only $\pi_0 A\rightarrow \pi_0 B$ is asked to be surjective.  For the third condition, note that if $B'\rightarrow B$ is the map $k[\epsilon]/\epsilon^2\rightarrow k$, then $F(B)$ is contractible, and we have an equivalence $F(A')\simeq F(A) \times F(B')$ inducing a bijection $F^\cl(A') \cong F^\cl(A) \times F^\cl(B')$.
\end{proof}
In light of \Cref{def:fmp}, 
\Cref{cor:artinian-small} immediately implies that we can detect equivalences of formal moduli problems  on trivial square-zero extensions:

\begin{proposition}[Conservativity] \label{prop:conservativity}
A map  $F\rightarrow F'$ in $\Moduli_{\Lambda}^{\an}$ is an equivalence if and only if  the map
\[
F(\sqz k[n]) \rightarrow F'(\sqz k[n])
\]
is an equivalence for all $n\geq 0$.   
\end{proposition}

\begin{remark}
The analogous statement for classical deformation functors is false: if $F\rightarrow F'$ is a map of deformation functors inducing an isomorphism of tangent spaces $F(k[\epsilon]/\epsilon^2) \rightarrow F'(k[\epsilon]/\epsilon^2)$, then $F\rightarrow F'$ need not be an isomorphism. 
\end{remark}

\begin{remark}[Tangent fibre] \label{rem:tangentfibre}
One can construct a functor
\[
T \colon \Moduli_\Lambda^{\an} \rightarrow \Mod_k,\,
F \mapsto T_F 
\]
such that for all discrete finite-dimensional $k$-modules $V$ and all $n\geq 0$, we have  
\begin{equation} \label{unravel} F(\sqz(V [n])) \simeq  \Map_{ \Mod_{k }} (V^\vee[n], T_F ) \simeq 
	\Omega^{\infty-n}( T_F       \otimes_{k} V)   . \end{equation}

Indeed, given a formal moduli problem $F \in \Moduli_{\Lambda}$, the sequence 
\[	
\left(\ \ \ldots , \ \  F(\sqz_k(-^\vee[2])), \ \ F(\sqz_k(-^\vee[1])),   \ \ F(\sqz_k(-^\vee))\right)\]
defines an  object $T_F 
\in  \Mod_k$  -- the \textit{tangent fibre} of $F$.

Here we have used that  $\Mod_{k}^{\cn}$ is equivalent to $\EuScript{P}_\Sigma(\mathrm{Vect}_k^\omega)$, the 
$\infty$-category of finite-product-preserving presheaves on the category of discrete
finite-dimensional $k$-modules $\mathrm{Vect}_k^\omega$, and that 
$\Mod_{k}$ arises as the inverse limit of $\infty$-categories
\[\Mod_k \simeq \lim \left( \ldots \rightarrow \Mod_{k }^{\cn}\xrightarrow{\Omega} \Mod_{k }^{\cn} \xrightarrow{ \Omega}   \Mod_{k}^{\cn} \right),\]
where the transition functor $\Omega$ is given by $M\mapsto \tau_{\geq 0} (M[-1])$.

\Cref{prop:conservativity} states that the functor $T$ is conservative. With more work, one can show that the above tangent fibre functor can be lifted to an equivalence  
\[
\Moduli_{\Lambda}^\an \isomto \Lie_{\Lambda/k, \Delta}^{\pi}
\]
between formal moduli problems and $(\Lambda,k)$-partition Lie algebras, cf.\ \cite[Theorem 6.27]{BrantnerMathew2019}. For $\Lambda=k$  a field of characteristic~$0$, this was established  by Lurie~\cite{LurieDAGX} and Pridham~\cite{Pridham2010}, building on  work of Hinich~\cite{Hinich2001}, Deligne~\cite{DeligneLetter}, \mbox{Drinfel'd \cite{DrinfeldLetter}, and  others.}
\end{remark}

\subsection{Obstructions}
Let us fix a formal moduli problem $F\colon \CAlg^{\art,\an}_\Lambda \rightarrow \cat{S}$.
By \Cref{lemma:inductive-small-extension},   
any elementary  extension $A'\surjto A$  of discrete local Artinian $\Lambda$-algebras with residue field $k$ arises as a pullback  
\[
\begin{tikzcd}
A' \arrow{d} \arrow{r} \arrow[rd, phantom, "\lrcorner", at start] & A \arrow{d} \\
k \arrow{r} & \sqz_k(I[1])
\end{tikzcd}
\]
in $\CAlg_{\Lambda}^{\an,\art}$, where $I=\ker(A'\rightarrow A)$. 
Applying $F$ gives a fibre sequence of spaces
\[
F(A') \rightarrow F(A) \rightarrow F(\sqz_k(I[1])).
\] 

\begin{definition}[Obstruction classes]Given $x\in \pi_0 F(A)$,
the \emph{obstruction class} $\ob_F(x,A'\rightarrow A)$ to lifting \mbox{$x$ to $A'$} is given 
by the image of $x$ in $\pi_0 F(k\oplus I[1])$.
\end{definition}
Clearly $x$ admits a lift to $A'$ if and only if $\ob_F(x,A'\rightarrow A)$ vanishes. 

We can now  introduce a  central notion studied in this paper: 
\begin{definition}[Unobstructedness] \label{unob}
A deformation functor $F$ is \emph{unobstructed} if for all surjections  $A' \surjto  A$ 
in $\CRing_\Lambda^\art$ the map $F(A') \surjto  F(A)$ is a surjection. 
\end{definition}

Since every surjection $A'\surjto A$ of discrete Artinian   $\Lambda$-algebras factors as a composition of  elementary  extensions (cf.\ \Cref{cor:artinian-small}), we deduce:

\begin{proposition}
\label{prop:unobstructed-via-obstruction-classes}
Let $F\colon \CAlg^{\an,\art}_\Lambda \rightarrow \cat{S}$ be a formal moduli problem. Then $F^\cl$ is unobstructed if and only if for every elementary extension $A'\rightarrow A$ in $\CRing^\art_\Lambda$  and  every $x\in \pi_0 F(A)$, the obstruction class
$\ob_F(x,A'\rightarrow A)$ vanishes.  
\end{proposition}

\subsection{Deformations of modules}
In the remainder of this section, we will construct several   formal moduli problems  required to prove our main theorems. 

To begin with, we consider derived deformations of modules, following \cite[Section 16.5.2]{LurieSAG}.
Let  $\cat{C}$ be  a $\Lambda$-linear stable $\infty$-category, i.e.\ a left $\Mod_{\Lambda}$-module in  $\Pr^{\rL,\st}$, the $\infty$-category of presentable stable $\infty$-categories  with the Lurie tensor product. By \cite[Remark 4.1.8, Corollary 2.6.6]{LurieDAGII}, we see that  $\cat{C}$ is left-tensored over $\Mod_{\Lambda}$, and we write  $\Mod(\cat{C})$ for the $\infty$-category of pairs $(A,M)$ of associative algebra object $A$  in $\Mod_{\Lambda}$ and left $A$-modules $M$ in $\cat{C}$, see    [op.cit, Definition 2.1.4].\vspace{3pt}

Let us fix a  left $k$-module object  $M$ in $  \Mod_k(\cat{C})$ that we want to deform.\vspace{-2pt} 

\begin{construction}[The functor $\Def^\mod_{M}$]\label{cons:modfmp}
Consider 	the   $\infty$-category
\[
\cat{D} := \CAlg^{\an}_\Lambda \times_{\Alg_\Lambda} \Mod(\cat{C}), 
\]
of pairs $(A,N)$ with $A$ in $\CAlg^{\an,\art}_\Lambda$ and $N$ a left $A$-module object in $\cat{C}$. 
The projection $q\colon \cat{D} \rightarrow \CAlg^{\an}_\Lambda$ defines a cocartesian fibration corresponding to the functor $\CAlg^{\an}_\Lambda \rightarrow \widehat{\cat{S}}$ sending $A$ to  $\Mod_A({\cat{C}})^{\simeq}$, the maximal Kan complex contained in the $\infty$-category $\Mod_A({\cat{C}})$ of left $A$-modules.  A map 
$(A,N)\rightarrow (A',N')$ is cocartesian  iff it induces an equivalence $A'\otimes_A N \xrightarrow{\simeq } N'$. 

Let $\cat{D}^{\cocart}$ be the \textit{non-full} subcategory of $\cat{D}$ consisting of all objects and only those morphisms which are $q$-cocartesian.
By \cite[Corollary 2.4.2.5]{LurieHTT}, the restriction of $q$ to $\cat{D}^\cocart$ is a left fibration, and one readily  verifies that
 that the induced map
\[
\cM^{\cocart}_{/(k,M)} \rightarrow \CAlg^{\an}_{\Lambda//k}
\]
is also a left fibration. Straightening  then gives rise to  a functor $\CAlg^{\an}_{\Lambda//k} \rightarrow \widehat{\cat{S}}$, and restricting to Artinian objects  defines the functor $\Def_{M}^\mod \colon \CAlg^{\an,\art}_\Lambda \rightarrow \widehat{\cat{S}}$.
\end{construction}

\begin{remark}[Informal description]
Objects in   $\Def_{M}^\mod(A)\in \widehat{\cat{S}}$  correspond to
pairs consisting of an $A$-module $N$ in $\cat{C}$ and an equivalence $k\otimes_A N \xrightarrow{\simeq} M$. The $1$-simplices    correspond to compatible equivalences of $A$-modules $N\xrightarrow{\simeq} N'$.

\end{remark}

Throughout this paper, we will make use of the following useful notation:

\begin{notation}[$\Map'$]\label{not:map-prime}
Given a map  $f\colon y\rightarrow x$ in some $\infty$-category $\cat{C}$, we define   \[\Map'_{\cat{C}}(x,y) := \Map_{\cat{C}_{/x}}(x,y)\] as the space of sections of $f$, which sits in a 
  pullback square 
\[
\begin{tikzcd}
	\Map_{\cat{C}}'(x,y) \arrow{r}\arrow{d}\arrow[rd, phantom, "\lrcorner", at start]  &
	\Map_{\cat{C}}(x,y) \arrow{d}{f} \\
	\{\ast\} \arrow{r}{\id} & \Map_{\cat{C}}(x,x).
\end{tikzcd}
\]
The map $f$ is implicit in the notation, and should be clear from the context.
\end{notation}

Given a functor $F\colon \CAlg^{\an,\art}_\Lambda\rightarrow \widehat{\cat{S}}$, the natural section of $\sqz_k(V)\rightarrow k$  equips $F(\sqz_k(V))$ with a base point. Write  $\Omega F(\sqz_k(V))$ for the corresponding loop space. 

Note that for $F = \Def_{M}^\mod$ as in \Cref{cons:modfmp}, the  base point in $\Def_{M}^\mod(\sqz_k(V))$  corresponds to the trivial deformation $\sqz_k(V)\otimes_k M$ of $M$ to $\sqz_k(V)$.

\begin{proposition}\label{prop:defmod-tangent-space}
Let $V\in \Mod^\cn_k$ be perfect. Then there is an equivalence
\[
\Omega\Def_{M}^\mod(\sqz_k( V)) \simeq
\Map'_{\Mod_{k}(\cat{C})}( M,\, \sqz_k( V)\otimes_k M )
\]
where the implicit map $\sqz_k( V)\otimes_k M \rightarrow M$ is the one induced by tensoring $M$ with the canonical augmentation $\sqz_k( V)\rightarrow k$.
\end{proposition}

\begin{proof}
Let us write $A=\sqz_k( V)$ for brevity. Unravelling the definitions, we see that $\Omega\Def_{M}^\mod(A)$ sits in a pullback square
\[
\begin{tikzcd}
	\Omega\Def_{M}^\mod(A) 
	\arrow{r} \arrow[rd, phantom, "\lrcorner",  at start]  \arrow{d}
	& \Map_{{A}(\cat{C})^\simeq}\big(
	A\otimes_k M,\, A \otimes_k M
	\big) \arrow{d}{k\otimes_A -} \\
	\{\ast\} \arrow{r}{\id} & \Map_{\Mod_k(\cat{C})^\simeq}( M, M)
\end{tikzcd}
\]
in $\widehat{S}$. Hence $\Omega\Def_{M}^\mod(A)$ is essentially small, and  the above square is equivalent to a pullback square in $\cat{S}$. Given $M'\in \Mod_A(\cat{C})$, tensoring over $A$ with the fibre sequence $V \rightarrow A \rightarrow k$ gives  a fibre sequence
\[
V\otimes_k (k \otimes_A M') \rightarrow M' \rightarrow k\otimes_A M'
\]
and hence the functor $M'\mapsto k\otimes_A M'$ is conservative. We obtain a pullback square
\[
\begin{tikzcd}
	\Map_{\Mod_{A}(\cat{C})^\simeq}\big(
	A\otimes_k M,\, A \otimes_k M
	\big)  \arrow[hook]{r} \arrow{d}{k\otimes_A -} \arrow[rd, phantom, "\lrcorner",  at start] 
	& \Map_{\Mod_{A}(\cat{C})}\big(
	A\otimes_k M,\, A \otimes_k M
	\big)  \arrow{d}{k\otimes_A -} \\
	\Map_{\Mod_k(\cat{C})^\simeq}( M, M) \arrow[hook]{r}
	& \Map_{\Mod_k(\cat{C})}( M, M). 
\end{tikzcd}
\]
Combining the two pullback squares  and using the  adjunction equivalence
\[
\Map_{\Mod_A(\cat{C})}(A\otimes_k M, A\otimes_k M) \simeq
\Map_{\Mod_k(\cat{C})}(M, A\otimes_k M)
\]
for the map $k\rightarrow A$  gives the asserted  equivalence.
\end{proof}

\begin{proposition}\label{prop:defmod}
If $\cat{C}$ admits a left complete $t$-structure for which $M\in \Mod_{k}(\cat{C})$ is connective, then $\Def_{M}^\mod$ is  equivalent to a formal moduli problem. In particular,  $\Def_{M}^\mod(A)$ is essentially small for all $A$.
\end{proposition}

\begin{proof} 
Since the $t$-structure on $\cat{C}$ is left complete, its connective part $\cat{C}_{\geq 0}$ is a separated prestable $\infty$-category in the sense of \cite[Definition C.1.2.12]{LurieSAG}. Just like in the proof of   Proposition 16.5.7.1 in [op.cit], we can therefore deduce from  Proposition 16.2.2.1   in [op.cit] that  $\Def_{M}^\mod$
satisfies \Cref{def:fmp} (2).

It remains to verify that each $\Def_{M}^\mod(A)$ is essentially small. 
For $A=\sqz_k(V)$ a trivial square-zero extension, this follows from  \Cref{prop:defmod-tangent-space} and the equivalence
\[
\Def_{M}^\mod (\sqz_k(V))\simeq \Omega\Def_{M}^\mod (\sqz_k( V[1]))
\]
induced by the pullback square
\[
\begin{tikzcd}
	\sqz_k(V) \arrow{r} \arrow{d} \arrow[rd, phantom, "\lrcorner",  at start] & k \arrow{d} \\
	k \arrow{r} & \sqz_k (V[1]).
\end{tikzcd}
\]
For general $A\in  \CAlg_{\Lambda}^{\an, \art}$, one uses \Cref{cor:artinian-small} to obtain  a factorisation 
\[
A = A_n \rightarrow \cdots \rightarrow A_0 = k,
\]
which implies  by induction that $\Def_{M}^\mod(A_i)$ is small for all~$i$. 
\end{proof}

We can now rephrase \Cref{prop:defmod-tangent-space} in its more traditional form:

\begin{corollary}\label{cor:defmod-tangent}
Assume that $\cat{C}$ admits a left complete $t$-structure for which $M\in \Mod_{k}(\cat{C}) $ is connective. Then for any perfect $V \in \Mod_k^{\cn}$, there is an equivalence
\[\Def_{M}^\mod (\sqz_k (V)) \simeq \Map_{\Mod_k(\cat{C})}(M, V[1]\otimes_k M).\]
\end{corollary}

\begin{proof}
Propositions~\ref{prop:defmod} and~\ref{prop:defmod-tangent-space} give equivalences
\[
\Def_{M}^\mod (\sqz_k (V)) \simeq \Omega\Def_{M}^\mod (\sqz_k (V[1])) \simeq
\Map'_{\Mod_{k}(\cat{C})}\big( M,\, \sqz_k( V[1])\otimes_k M \big).
\]
We  deduce the result using  the  decomposition $
\sqz_k (V)\otimes_k M \simeq M \,\oplus\, (V\otimes_k M)$.
\end{proof}

\begin{remark}[Tangent fibre of $\Def_{M}^\mod $] \label{rem:tanDefM}
We will mostly be interested in the case where $\cat{C}=\Mod_{\bZ}$ is the derived $\infty$-category of $\bZ$  and $M \in \Mod_{k} \simeq \Mod_k(\cat{C})$ is   eventually connective (i.e.\ cohomologically bounded above).
 Writing $\underline{\smash{\Map}}$ for the mapping complex, 
\Cref{cor:defmod-tangent}   shows that  the tangent fibre   of $\Def_M^{\mod}$ \mbox{is   given by}
\[T_{\Def_M} \simeq \underline{\smash{\Map}}_{\hspace{1pt} k}(M,M)[1].\]

Indeed,    the connective chain complexes 
$\tau_{\geq {0}}(T_{\Def_M}[n])$ and
\mbox{$\tau_{\geq {0}}( \underline{\smash{\Map}}_{\hspace{1pt} k}(M,M)[1+n])$}
are equivalent for all $n\geq 0$, as they 
correspond to equivalent  presheaves under the identification $\Mod_k^{\cn}\simeq \cat{P}_{\Sigma}(\Mod_k^{\ff})$ from \Cref{ex:connchain}. These  send $V \in \Mod_k^{\ff}$ to 
\[\Def_M(\sqz_k (V^{\vee}[n])) \simeq \Map_{k}(M,V^\vee[n+1] \otimes_k M ) \simeq
\Map_{k}(V, \tau_{\geq 0}(\underline{\smash{\Map}}_{\hspace{1pt} k}(M,M)[1+n])) \]

The above identification of $T_{\Def_M}$ is  related to the classical fact that the tangent space to the (classical) deformation functor of $M$ is given by $\Ext^1(M,M)$, while the obstruction classes  live in $\Ext^2(M,M)$.

\end{remark}

\subsection{Deformations of sheaves of animated rings} \label{sec:defsheavesanring}
Fix a pair $X=(\cat{T},\cO_X)$ with $\cat{T}$ a small Grothendieck site and $\cO_X \in \Shv(\cat{T},\CAlg^\an_k)$.

\begin{construction}[The functor $\Def_{X}$]
\label{constr:def-an}
Consider the $\infty$-category
\[
\cat{C} := \CAlg^\an \times_{\Shv(\cat{T},\CAlg^\an)} 
\Fun(\Delta^1,\Shv(\cat{T},\CAlg^\an))
\]
whose objects are pairs $(A,A\rightarrow\cB)$ with $A \in\CAlg^\an$ and $A\rightarrow\cB$ a map in $\Shv(\cat{T},\CAlg^\an)$. 
Here we use the functor $\CAlg^\an  \rightarrow \Shv(\cat{T},\CAlg^\an)$ sending $A$ to the sheafification of the corresponding constant presheaf.

Consider the projection $q\colon \cat{C} \rightarrow \CAlg^\an$
and write $\cat{C}^\cocart \subset \cat{C}$ for the \mbox{\textit{non-full}} subcategory consisting of all objects and only the $q$-cocartesian morphisms. These are precisely the morphisms in $\cat{C}$ for which the underlying square
\[
\begin{tikzcd}
	A \arrow{r}\arrow{d} & A' \arrow{d} \\
	\cB \arrow{r} & \cB'\arrow[lu, phantom, "\ulcorner", at start]
\end{tikzcd}
\]
in $\Shv(\cat{T},\CAlg^\an)$ is a pushout. By \cite[Corollary 2.4.2.5]{LurieHTT}, the restriction of $q$ to $\cat{C}^\cocart$ is a left fibration, and the induced functor  
\[
\cat{C}^\cocart_{/(k,\cO_X)} \rightarrow \CAlg^\an_{/k}
\]
is also a left fibration. Straightening then gives rise to a functor  $\CAlg^\an_{/k}\rightarrow \widehat{\cat{S}}$. Composing with the forgetful functor $\CAlg^{\an,\art}_\Lambda \rightarrow \CAlg^\an_{/k}$, we obtain the functor \[\Def_X\colon \CAlg^{\an,\art}_\Lambda \rightarrow \widehat{\cat{S}}.\]
\end{construction}

\begin{remark}[Informal description]
Unravelling the construction, we see that the points of the space $\Def_X(A)$ correspond to pushout squares 
\[
\begin{tikzcd}
	A \arrow{r} \arrow{d} & k \arrow{d} \\
	\cB \arrow{r} & \cO_X \arrow[lu, phantom, "\ulcorner", at start]
\end{tikzcd}
\]
in $\Shv(\cat{T},\CAlg^\an)$, i.e.\ to pairs consisting of a $\cB\in \Shv(\cat{T},\CAlg^\an_A)$ and an equivalence $k\otimes_A \cB \xrightarrow{ \simeq } \cO_X$ in $\Shv(\cat{T},\CAlg^\an_k)$.
The $1$-simplices of the space $\Def_X(A)$ are given by compatible equivalences $\mathcal{B} \rightarrow \mathcal{B}'$  in $\Shv(\cat{T},\CAlg^\an_A)$.
\end{remark}

For $V\in \Mod^\cn_k$,  $k \rightarrow \sqz_k(V)$   induces a map $\ast \rightarrow \Def(\sqz_k(V))$  picking  out the `trivial deformation' $\cB\simeq \cO_X \otimes_k \sqz_k(V)$. We compute the based loop space:

\begin{proposition}\label{prop:defan-tangent-space}
Let $V\in \Mod^\cn_k$ be perfect. Then there is an equivalence 
\[
\Omega\Def(\sqz_k(V)) \simeq 
\Map'_{\Shv(\cat{T},\CAlg_k)}( \cO_X, \sqz_k(V)\otimes_k\cO_X)
\]
where the implicit map $\sqz_k(V)\otimes \cO_X \rightarrow \cO_X$ is the one induced by $\sqz_k(V)\rightarrow k$.
\end{proposition}

\begin{proof}
Analogous to the proof of \Cref{prop:defmod-tangent-space}.
\end{proof}

\begin{proposition}\label{prop:defan}
The functor $\Def$ is a formal moduli problem. In particular,  $\Def(A)$ is essentially small for all $A$.
\end{proposition}

\begin{proof} 
First, we show that $\Def_X$ satisfies the infinitesimal cohesiveness property in \Cref{def:fmp} (2). The argument is essentially the same as in \cite[Theorem 16.3.0.1]{LurieSAG}. 
For the sake of conciseness, let us write $\cat{C}_A := \Shv(\cat{T},\CAlg^\an_A)$ and $\cat{D}_A := \Shv(\cat{T},\Mod^\cn_A)$. It suffices to show that for every pullback diagram
\[
\begin{tikzcd}
	A \arrow{r} \arrow{d} \arrow[rd, phantom, "\lrcorner",  at start] & A_0 \arrow{d} \\
	A_1 \arrow{r} & A_{01} 
\end{tikzcd}
\]
in $\CAlg^\an$ with $\pi_0 A_0 \rightarrow \pi_0 A_{01}$ and $\pi_0 A_1 \rightarrow \pi_0 A_{01}$ surjective, the resulting functor
\[
L: \cat{C}_A \rightarrow\cat{C}_{A_0} \times_{\cat{C}_{A_{01}}} \cat{C}_{A_1}
\]
is an equivalence. Objects in the target can be identified with diagrams\vspace{-1pt}
\[
\begin{tikzcd}
	A_0 \arrow{r} \arrow{d} 
	& A_{01}  \arrow{d} 
	& A_1 \arrow{l} \arrow{d} \\
	\cB_0 \arrow{r} & \arrow[lu, phantom, "\ulcorner", at start] \cB_{01} \arrow[ru, phantom, "\urcorner", at start]& \cB_1 \arrow{l} 
\end{tikzcd}
\]
in $\Shv(\cat{T},\CAlg^\an)$, where both squares are pushouts. The functor $L$ has a right adjoint $R$, which maps the diagram above to the pullback of $\cB_{0} \rightarrow \cB_{01} \leftarrow \cB_1$. 

Using   that $\cB\otimes_A -$ preserves pushouts in  $\cat{D}_A$, and the fact that the forgetful functor $\cat{C}_A \rightarrow \cat{D}_A$ is conservative, one sees that the unit $u\colon \cB \rightarrow RL\cB$ is an equivalence. To see that the counit   $v\colon LR \cB_\ast \rightarrow \cB_\ast$ is an equivalence, it suffices to verify that it is an equivalence in $\cat{D}_{A_0}\times_{\cat{D}_{A_{01}}} \cat{D}_{A_1}$, which follows from \cite[Theorem 16.2.0.2]{LurieSAG}.

To see that $\Def(A)$ is essentially small, we proceed as in the proof of \Cref{prop:defmod} 
to deduce it from \Cref{prop:defan-tangent-space}.
\end{proof}

\begin{corollary}\label{cor:defan-tangent-space}
$\Def_X(\sqz_k(V))\simeq 
\Map_{\Mod(\cat{T},\cO_X)}(L_{X/k},\,V[1]\otimes_k \cO_X)$.
\end{corollary} 
\begin{proof}
This follows  from  the equivalence $\Def_X(\sqz_k(V))\simeq \Omega \Def_X(\sqz_k(V)[1])$ together with Propositions~\ref{prop:defan-tangent-space} 
and~\ref{prop:cotangent-sheaf-adjunction}, as $\cO_X \otimes_k \sqz_k(V) \simeq \sqz_{\cO_X}(\cO_X \otimes_k V )$.
\end{proof}

\begin{remark}[Tangent fibre of $\Def_X$]
When $X=(X, \cO_X)$ is a smooth scheme, then arguing as in \Cref{rem:tanDefM} gives an equivalence
$ T_{\Def_X} \simeq R\Gamma(X,T_X)[1].$

\Cref{cor:defan-tangent-space} is therefore a derived enhancement of classical Kodaira--Spencer theory, asserting that the tangent space of the (classical) deformation functor of $X$ is given by $H^1(X,T_X)$, while its obstruction classes belong to $H^2(X,T_X)$.
\end{remark}

\subsection{Deformations of schemes}\label{sec:defschemes}
For $X$  a derived scheme  over  $k$, consider
\begin{enumerate} 
\item deformations of the structure sheaf $\cO_X$ in Zariski sheaves;
\item deformations of $X$ in derived schemes;
\item deformations of the \'etale structure sheaf $\cO_{X,\et}$
in \'etale sheaves;
\item deformations of $X$ in derived Deligne--Mumford stacks.
\end{enumerate} 
We will now show that these   contexts
lead to equivalent formal moduli problems.

\subsubsection*{Zariski sheaves  and  schemes } We  formalise deformations of $X$ in \mbox{derived schemes}:

\begin{construction}[Kodaira--Spencer moduli] \label{KSmod}
Consider the $ \infty$-category 
	\[\cat{D}:=
	\CAlg^{\an}  \times_{  \Sch^{\der,\op}}	\Fun(\Delta^1, \Sch^{\der,\op})\]
	of pairs of animated rings $A \in \CAlg^\an$ and maps of derived schemes
	$Y \rightarrow \Spec(A)$.
	
	Consider the projection $ \cat{D} \rightarrow \CAlg^{\an} $, which is cocartesian as derived schemes admit pullbacks  (cf.\ \cite[Section 1.3.3]{TV08}, \cite[Proposition 2.3.21]{LurieDAGV}).
	These  are  computed in $\Topp_{\CAlg^{\an}}^{\loc} \subset \Topp_{\CAlg^{\an}} $, the (non-full) subcategory of animated ringed spaces $(Y, \cO_Y)$ such that $(Y, \pi_0\cO_Y)$ is locally ringed, and local maps between them.
	Cocartesian arrows  correspond to morphisms in $\cat{D}$ for which the underlying square 
	\[
	\begin{tikzcd}
		Y\arrow{r}\arrow{d}\arrow[rd, phantom, "\lrcorner",  at start]   & Y' \arrow{d} \\
		\Spec(A) \arrow{r} & \Spec(A') 
	\end{tikzcd}
	\]
	 is a pullback in $\Sch^{\der}$.	
Let $\cat{D}^{\cocart}$ be the ({non-full}) subcategory   of all objects and only $q$-cocartesian morphisms. By \cite[Corollary 2.4.2.5]{LurieHTT}, restricting $q$ to $\cat{D}^\cocart$ gives a left fibration.
Then   $X$ gives   an object $(k,X) \in \cat{D}$, and  the induced functor  
	\[
	\cat{D}^\cocart_{/(k,X)} \rightarrow \CAlg^{\an}_{ / k}
	\]
	is also a left fibration. Straightening then gives rise to a functor  $\CAlg^\an_{/k}\rightarrow \widehat{\cat{S}}$, and  composing with the forgetful functor $\CAlg^{\an,\art}_\Lambda \rightarrow \CAlg^\an_{/k}$, we obtain the  functor \[\Def^{\sch}_X\colon \CAlg^{\an,\art}_\Lambda \rightarrow \widehat{\cat{S}}.\]
\end{construction}
To establish the  equivalence
 $\Def_X^{\sch} \simeq \Def_X$, we need several  \mbox{preliminary results.}
\begin{proposition}\label{prop:homeo}
	Given a  deformation $(Y,\cO_Y) \in \Def^{\sch}_X(A)$ of $X$ over some $A$ in $\CAlg^{\an,\art}_{\Lambda}$, the underlying map of topological spaces $X \rightarrow Y$ is a homeomorphism.
\end{proposition} 
\begin{proof}
The map $\Spec(k) \rightarrow \Spec(A)$  is a closed immersion and a homeomorphism on underlying topological spaces.  By the derived analogue of \cite[Corollary 10.5]{LurieDAGIX},  
	the pullback corresponding to $Y$ is   a pullback in animated ringed spaces.
\end{proof}

\begin{proposition}\label{lemma:lift-is-derived-scheme}
	Let $A \in \CAlg^{\an,\art}_{\Lambda }$ and 
	consider corresponding  squares 
	\[
	\begin{tikzcd}
		A \arrow{r} \arrow{d} & k \arrow{d} \\
		\cB \arrow{r} & \cO_X 
	\end{tikzcd} \ \ \ \  \ \ \ \  \ \ \ \  \begin{tikzcd}
		(X,\cO_X)\arrow{r} \arrow{d}  &  \arrow{d} 	(X,\cB) \\
		\Spec(k)\arrow{r} &\Spec(A)   .
	\end{tikzcd}
	\]   
The left square is a pushout of $\CAlg^{\an}$-valued sheaves on $X$
if and only if 
 the right square is a pullback of derived schemes.  
\end{proposition} 

\begin{proof}  
	The `if' direction follows as in \Cref{prop:homeo} from the fact that the pullback of  
	$\Spec(k) \rightarrow \Spec(A) \leftarrow (X,\mathcal{B})$ is computed in animated ringed spaces.  
	
	For the `only if' implication, we assume that the left square is a pushout. We begin by showing that $(X,\mathcal{B})$ is a derived scheme and that  $(X,\cO_X) \rightarrow (X,\cB)$ is a  local morphism.
	Choose a factorisation
	\[
	A = A_n \rightarrow \cdots \rightarrow A_r  \rightarrow \cdots \rightarrow A_0= k
	\]
	as in \Cref{cor:artinian-small}; we can arrange that  $n_i = 0$ for $ i \leq r$ and $n_i>0 $ for $i>r$.
	Set $\cB_n := A_n \otimes_A \cB$, so that we have fibre sequences
	\[
	V_i[n_i] \otimes_k \cO_X \rightarrow \cB_{i+1} \rightarrow \cB_i, \ \ n_i \geq 0.
	\]
	For each $i$, we obtain a short exact sequence
$
	0 \rightarrow I_i \rightarrow \pi_0 \cB_{i+1} \rightarrow \pi_0 \cB_{i} \rightarrow 0
$
	with $I_i$ a quasi-coherent sheaf  of $\pi_0\cO_X$-modules. It  follows by induction, using \cite[05YV]{stacks}, that $(X,\pi_0 \cB)$ is a scheme and that the map $(X,\pi_0\cO_X) \rightarrow (X,\pi_0\cB)$  is local. By \cite[01LC]{stacks}, $\cO_X$ is   quasi-coherent over $\pi_0\cB$.
	
	We also see by induction that  $ \cB$ is hypercomplete, and, using the long exact sequence on homotopy sheaves, that each $\pi_k \cB$  is a quasi-coherent $\pi_0 \cB$-module.  Hence $(X,\mathcal{B})$ is a derived scheme and the map $(X,\cO_X) \rightarrow (X,\mathcal{B})$  is local.
	
The right square is a pullback because as in the proof of \Cref{prop:homeo}, this pullback in derived schemes is computed in animated ringed spaces.  
\end{proof}

\begin{proposition}\label{lemma:lift-is-derived-scheme-morphisms}
Fix $A \rightarrow A'$   in $ \CAlg^{\an,\art}_{\Lambda }$ and  consider corresponding diagrams \[
\begin{tikzcd}
	A \arrow{r} \arrow{d} & 	\arrow[phantom, from=1-1, to=2-2, "\mathrm{(I)}", pos=0.5]	A' \arrow{r} \arrow{d} & k \arrow{d} \\
	\cB \arrow{r} &	\cB' \arrow{r}   &     \cO_X 
\end{tikzcd}   \ \ \ \  \ \ \ \   \ \    
\begin{tikzcd}
	(X,\cO_X)  \arrow{r} \arrow{d}  &  \arrow{d} 	(X,\cB') \arrow{r} \arrow{d}  &  \arrow{d} 	(X,\cB) \\
	\Spec(k)\arrow{r} &\Spec(A')   \arrow{r} &\Spec(A)   .	\arrow[phantom, from=1-2, to=2-3, "\mathrm{(II)}", pos=0.5]
\end{tikzcd}
\]   
 
\noindent
Assume that the right and outer square on the left are pushouts in $\Shv(X,\CAlg^{\an})$. 
Then $(I)$ is a pushout in $\Shv(X,\CAlg^{\an})$ if and only if  $(II)$ is a pullback in $\Sch^{\der}$.
\end{proposition} 
\begin{proof}Both left   and   outer square on the right are  pullbacks in $\Sch^{\der}$ {by \Cref{lemma:lift-is-derived-scheme}.  }
Square (I) is a pushout if and only if  (II) is a pullback in $\Topp_{\CAlg}$. It  suffices to check that 
the pullback of $\Spec(A') \rightarrow \Spec(A) \leftarrow (X,\cB)$ in \textit{locally} animated ringed spaces is computed in $\Topp_{\CAlg}$.
As $(X,\cB)$ is a derived scheme, we can assume   without restriction that \mbox{$(X,\cB) = \Spec(R)$} is  affine. Here, the claim   follows as the maps \mbox{$\pi_0(R) \rightarrow \pi_0(A') \otimes_{\pi_0(A)} \pi_0(R) \rightarrow k \otimes_{\pi_0(A)} \pi_0(R)$} induce homeomorphisms  on prime spectra, as the maximal ideals of $\pi_0(A)$ and $\pi_0(A')$ \mbox{are nilpotent. }
\end{proof}

\begin{corollary}
	For $X$ a derived scheme over $k$, there is a canonical equivalence  
	\[\Def^{\sch}_X \simeq {\Def}_X.\]
	In particular, $\Def^{\sch}_X$ is a formal moduli problem.
\end{corollary} 
\begin{proof}
Define $\cat{C}$ as in \Cref{constr:def-an} for $\cat{T}$ the Zariski site of $X$.
Write $\cat{C}^{\art}_{/(k,\cO_X)} \subset \cat{C}^{ }_{/(k,\cO_X)}$ for the full subcategory  on all objects over animated rings $A\in \CAlg_{/k}$ satisfying conditions (1) and (2) of \Cref{def:artinian-animated-ring}. We define $\cat{C}^{\cocart,\art}_{/(k,\cO_X)} $, $\cat{D}^{\art}_{/(k,X)}$, and 
$\cat{D}^{\cocart,\art}_{/(k,X)} $ \vspace{2pt} similarly.
By  \vspace{1pt} \Cref{lemma:lift-is-derived-scheme},  \Cref{lemma:lift-is-derived-scheme-morphisms}, and \Cref{prop:homeo}, 
the natural functor $ \cat{C}^{\cocart, \art}_{/(k,\cO_X)} \rightarrow  \cat{D}^{\art }_{/(k,X)} $ gives a \vspace{1pt} fully faithful and essentially surjective functor $ \cat{C}^{\cocart,\art}_{/(k,\cO_X)} \rightarrow  \cat{D}^{\cocart, \art }_{/(k,X)} $.  
\end{proof}
 
\begin{definition}
	A connective module $M$ over an animated   ring $A$  is \emph{flat}  if $\pi_0 M$ is flat as a $\pi_0 A$-module, and if for every $n$ the map
	$
	\pi_n A \otimes_{\pi_0 A} \pi_0 M \rightarrow \pi_n M
	$
	is an isomorphism. We say that a derived scheme $(X,\cO_X)$ over $A$ is \emph{flat} if for every affine open $U\subset (X,\pi_0 \cO_X)$ the $A$-module $\cO_X(U)$ is flat.
\end{definition}

\begin{lemma}\label{lemma:flatness-over-artinian}
	Let $A\rightarrow k$ be in $\CAlg_\Lambda^{\art,\an}$, and $M$ in $\Mod^{\cn}_{A}$. Then $M$ is flat if and only if $k\otimes_A M$ is discrete.
\end{lemma}

\begin{proof}
	By \cite[Theorem 7.2.2.15]{LurieHA}, the $A$-module $M$ is flat if and only if $M\otimes_A N$ is discrete for all discrete $A$-modules $N$. But every such $A$-module has a finite filtration whose graded pieces are discrete $k$-modules, so $M$ is flat if and only if $k\otimes_A M$ is flat. Since $k$ is a field, this is equivalent with $k\otimes_A M$ being discrete.
\end{proof} 

\begin{corollary}\label{cor:flat-scheme-over-artinian}
	Let $X$ be a quasi-compact and quasi-separated scheme and  assume 
 $\cB$ is as in \Cref{lemma:lift-is-derived-scheme}. Then $(X,\cB)$ is flat over $A$.
\end{corollary}

\begin{proof}
	We first show that 
	the natural map
	$
	k\otimes_A \cB(U) \longto (k\otimes_A \cB)(U)
	$ is an equivalence for  
 $U\subset X$ affine open.
As sheafification is local, we may assume   that $U = (X,\cB)$. Here,  the claim follows from \Cref{sheafifyingonbases}, as the presheaf $V \mapsto 	k\otimes_A \cB(V)$ satisfies condition  \Cref{sheafesbases}(\ref{cond2}) on  distinguished affines.
Hence $k\otimes_A \cB(U) \simeq \cO_X(U)$, and  \Cref{lemma:flatness-over-artinian} shows that \mbox{$\cB(U)$ is flat over $A$}.
\end{proof} 
Hence  $\Def_X(A)$ is the space ($\infty$-groupoid) of flat derived schemes over $A$ \mbox{lifting $X$.} If $A$ is discrete,  $\Def_X(A)$ is the $1$-groupoid of flat \mbox{schemes over $A$ lifting $(X,\cO_X)$.}

\subsubsection*{Zariski sheaves and  \'{e}tale sheaves}

We proceed to  the equivalence between \mbox{(1) and (3).}
Let $(X_\et,\cO_{X,\et})$ be the pair of small \'etale site 
and \'etale structure sheaf of $X$ (\Cref{def:smet}).
Subsection \ref{sec:defsheavesanring}  constructs two formal moduli problems \mbox{$\Def_{X_{\zar}}$ and $\Def_{X_\et} $.}

\begin{lemma}  For $X$ a derived scheme over $k$, there is a canonical equivalence  \[\Def_{X_\et} \rightarrow \Def_{X_{\zar}}.\]  
\end{lemma}

\begin{proof}
	By conservativity of the tangent fibre (cf.\ \Cref{prop:conservativity}), it suffices to verify that for every perfect $V$ in $\Mod_k^{\cn}$, the natural map 
	\[
	\Def_{X_{\et}}(\sqz_k(V) ) \rightarrow 
	\Def_{X_{\zar}}(\sqz_k(V) )
	\]
	is an equivalence.
	By \Cref{cor:defan-tangent-space}, this map is given by the   natural map
	\[
	\Map_{\Mod(X_\et, \cO_{X,\et})}(L^{\et}_{X/k},V\otimes_k \cO^{\et}_X) \rightarrow
	\Map_{\Mod(X, \cO_{X})} (L_{X/k},V\otimes_k \cO_X).
	\]
	We now consider the usual pair of adjoint functors $$\iota^\ast: \Mod(X_\et, \cO_{X}) \leftrightarrows \Mod(X_\et, \cO_{X,\et}): \iota_\ast, $$
	where $\iota_\ast$ is obtained by evaluating an \'etale sheaf on Zariski opens.
	
	It is not hard to see that the unit $\mathcal{F} \xrightarrow{\simeq} \iota_\ast \iota^{\ast} \mathcal{F}$ is an equivalence for all quasi-coherent Zariski sheaves. Indeed, since sheafification is local, it suffices to check this on affine opens $U\subset X$, where the claim is an instance of faithfully flat  descent (cf.\ \cite[D.6.3.1]{LurieSAG}).
	We deduce that   $i^{\ast}$ is fully faithful on quasi-coherent sheaves, and so the claim follows from \Cref{prop:cotzaret}. 
\end{proof}

\subsubsection*{\'{E}tale sheaves and Deligne--Mumford stacks}
To construct our Serre-Tate period map in \Cref{sec:serre-tate}, 
we  need  a variant of   $\Def_{X_\et}$. 
Let $k$ be a perfect field of characteristic $p$ and set $\Lambda = W = W(k)$. Write $S_{\et}$ for the small  \'etale site of \mbox{$\Spec(k)$.}
\begin{notation}[The sheaf $\cW$] \label{not:cW}
		The sheaf $\cW \in\Shv(S_{\et}, \CAlg^{\an})$ sends an \'{e}tale map 
	$\coprod_{i} \Spec(L_i) \rightarrow \Spec(k)$ with $L_i/k$ a finite separable extension to  $\prod_i W(L_i)$.  
\end{notation}

\begin{notation}The structure map $f: X \rightarrow \Spec(k)$ gives  an inverse image  functor
	$$f^{-1}:  \Shv(S_{\et}, \CAlg^{\an}) \rightarrow \Shv(X_{\et}, \CAlg^{\an}).$$
As before, we 
denote all constant sheaves corresponding to \mbox{some  $A \in \CAlg^{\an}$ by $A$.}
\end{notation}

\begin{construction}[The functor $\widetilde{\Def}_{X_\et}$]
	\label{constr:def-an-tilde}
	Consider the $\infty$-category
	\[
	\widetilde{\cat{C}}_{\et} := \CAlg^{\an}_{W} \times_{  \Fun(\Delta^1, \Shv(X_{\et},\CAlg^\an))} 
	\Fun(\Delta^2,\Shv(X_{\et},\CAlg^\an))
	\]
	whose objects are pairs $(A\ , \   A \rightarrow f^{-1}\cW\otimes_W {A}  \rightarrow \cB)$ consisting of an animated $W$-algebra  $A$ and a composition of maps of  $\CAlg^{\an}$-valued sheaves on $X_{\et}$.
	
	We have a natural projection $\widetilde{q}\colon 	\widetilde{\cat{C}}_{\et}  \rightarrow \CAlg^{\an}_{W}$, 
	and  write $	\widetilde{\cat{C}}^{\cocart}_{\et}  \subset 	\widetilde{\cat{C}}$ for the \mbox{\textit{non-full}} subcategory consisting of all objects and only the $\widetilde{q}$-cocartesian morphisms. 
	These are   the morphisms in $\widetilde{\cat{C}}$ for which the underlying diagram
	\begin{equation}\label{diagram:tilde}
		\begin{tikzcd}
			A \arrow{r}\arrow{d} & A' \arrow{d} \\	
			f^{-1}\cW\otimes_W {A} \arrow{r}\arrow{d} & f^{-1}\cW\otimes_W {A}' \arrow{d}   \arrow[lu, phantom,  " \ \ \  \ulcorner", at start]\\
			\cB \arrow{r} & \cB'\arrow[lu, phantom, "\ulcorner", at start]
		\end{tikzcd}
	\end{equation}
	consists of pushouts.
By  \cite[Corollary 2.4.2.5]{LurieHTT},   the   restriction of $\widetilde{q}$ to $\widetilde{\cat{C}}^{\cocart}$ is a left fibration.
	As $ \cO_{S,\et} =  \cW\otimes_W {k}$ is the \'{e}tale structure sheaf of $\Spec(k)$, we obtain a map \mbox{$k \rightarrow f^{-1}\cW\otimes_W {k} \rightarrow \cO_{X,\et}$}
	and hence an object in $\widetilde{\cat{C}}$, which \mbox{we  denote by $(k,\cO_{X,\et})$.}

	The induced functor  
	$
	(\widetilde{\cat{C}}^\cocart)_{/(k,\cO_X)} \rightarrow \CAlg^{\an }_{W//k}
	$
	is also a left fibration. Straightening and restricting, we obtain  a functor 
	$ \widetilde{\Def}_{X_{\et}}:\CAlg^{\an,
		\art}_W \rightarrow \widehat{\cat{S}}.$
\end{construction}

\begin{proposition}\label{prop:equivetale}
	Restricting along $\Delta^{\{0,2\}}\subset \Delta^2$ induces  a canonical equivalence 
	$$\widetilde{\Def}_{X_{\et}} \xrightarrow{\simeq}   {\Def}_{X_{\et}}  $$
	of formal moduli problems  sending a diagram as in (\ref{diagram:tilde}) to the outer pushout square.
\end{proposition} 	

The proof of \Cref{prop:equivetale} relies on a  universal property of  
 the morphism \mbox{$A \rightarrow f^{-1}\cW  \otimes_W A$}. It is naturally formulated in the language of  derived Deligne--Mumford stacks, which we will now recall. As it  is not needed anywhere else in this article,  we encourage the  impatient reader to  \vspace{-2pt}  proceed directly to \Cref{sec:animated-divided-power}. 
\begin{notation}[Sheaves on $\infty$-topoi]
For $\cat{X}$ an 
 $\infty$-topos (cf.\   \cite[Ch.\ 6]{LurieHTT}),   write $\Shv(\cat{X} , \CAlg^{\an})$ for the $\infty$-category 
of limit-preserving functors $\cat{X}^{\op} \rightarrow  \CAlg^{\an} $.
\end{notation}
\begin{remark}\label{rem:sheavesontopoi}
Every small  Grothendieck site $\cat{T}$ gives an $\infty$-topos $\cat{X} = \Shv(\cat{T}, \cat{S})$.
By \cite[Corollary 1.3.1.8]{LurieSAG}, $\Shv(\cat{X}, \CAlg^{\an}) $ is   canonically equivalent to $\Shv(\cat{T}, \CAlg^{\an})$.
Note that the $\infty$-topos $\cat{S}$ of spaces corresponds to the topological space $\ast$.
\end{remark}
 
 \begin{construction}[Animated ringed $\infty$-topoi]
The functor
 	$\cat{X} \mapsto \Shv(\cat{X}, \CAlg^{\an})^{\op}$
 from the  $\infty$-category $\infty\Topp$ of $\infty$-topoi to the $\infty$-category $ \widehat{\Cat}_{\infty}$ of (large) $\infty$-categories  is classified by a   cocartesian fibration 
 	$q: \infty\Topp_{\CAlg^{\an}}  \rightarrow \infty\Topp.$
The fibre over   $\cat{X}$ is the $\infty$-category of pairs $(\cat{X}, \cO)$, where   $\cO$ is   a \mbox{$\CAlg^{\an}$-valued sheaf on $\cat{X}$.}
 \end{construction}
 
 \begin{example}[\'{E}tale spectrum]
 	For every animated ring $R \in \CAlg^{\an}$,    
 	the forgetful functor $\CAlg^{\an,\et}_{R} \rightarrow \CAlg^{\an}$ defines a sheaf $\cO_{R,\et}$ on the \'{e}tale site $\CAlg^{\an,\et}_{R}$.
 	We obtain an animated ringed $\infty$-topos 
 	$ \Spet(R) := (\Shv(\CAlg^{\an,\et}_{R}, \cat{S})  ,  \cO_{R,\et})$ and a natural map  
 	\mbox{$\alpha_R: \Spet(R) \rightarrow (\cat{S}, R) $ in $\infty\Topp_{\CAlg^{\an}}$.}
 \end{example}
  
 \begin{definition}[Global sections]
 	The  natural inclusion 
 	$(\CAlg^{\an})^{\op}  \subset \infty\Topp_{\CAlg^{\an}}  $
mapping  $A \in \CAlg^{\an}$ to the pair $(\cat{S}, A)$  
 admits a left adjoint $(\cat{X}, \cO) \mapsto \Gamma(\cat{X}, \cO) $.
\end{definition}

\begin{notation}[Strictly Henselian objects] Write $\infty\Topp_{\CAlg^{\an}}^{\sHen} \subset \infty\Topp_{\CAlg^{\an}}$
for the (non-full) subcategory whose objects are those  $(\cat{X},\cO)$ for which 
$\pi_0\cO$ is a strictly Henselian sheaf of commutative rings on the underlying topos $\cat{X}^{\heartsuit}$ of $\cat{X}$, cf.\  \cite[Definition 1.2.2.5]{LurieSAG}. The morphisms are 
those $\phi: (\cat{X}, \cO) \rightarrow (\cat{X}', \cO')$
  for which   $\pi_0 \phi^\ast \cO' \rightarrow \pi_0 \cO$  is  local, see \cite[Definition 1.2.1.4]{LurieSAG}.
\end{notation}
We can now characterise $\alpha_R$ by a universal property    \mbox{(cf.\ \cite[1.4.2.4]{LurieSAG}  \cite[2.2]{LurieDAGV}):}
\begin{proposition}\label{prop:adjunction}The assignment $R \mapsto \Spet(R)$  is part of an adjunction  $$
	\Gamma : 
	\infty\Topp_{\CAlg^{\an}}^{\sHen} 
	\leftrightarrows 	(	\CAlg^{\an} )^{\op}
	:  	\Spet$$
	with counit $\alpha$. For $(\cat{Y}, \cB)\in \infty\Topp_{\CAlg^{\an}}^{\sHen}$, postcomposing with  $\alpha_R$ \mbox{gives an equivalence}
	$$ \hspace{-10pt} \Map_{\infty\Topp_{\CAlg^{\an}}^{\sHen}}((\cat{Y}, \cB), \Spet(R)) \xrightarrow{\simeq}
	\Map_{\infty\Topp_{\CAlg^{\an}}}((\cat{Y}, \cB), (\cat{S}, R)) \simeq	\Map_{\CAlg^{\an}}(R, \Gamma(\cat{Y}, \cB)). 
	$$ 
\end{proposition}
Following \cite[Section 1.4.4]{LurieSAG} and \cite[Definition 4.3.20]{LurieDAGV}, we define: 
\begin{definition}[Derived Deligne--Mumford stacks]
	A  \textit{derived Deligne--Mumford stack} is an animated ringed $\infty$-topos $(\cat{X}, \cO_{\cat{X}})$  which is locally an \'{e}tale  spectrum. More precisely, this means that 
	 there is a collection of objects $U_\alpha \in \cat{X}$  \mbox{such that}
	\begin{enumerate}
		\item the map  $\coprod_\alpha U_\alpha \rightarrow \mathbf{1}$ to the final object of $\cat{X}$ is an effective epimorphism;
		\item for all $\alpha$, there is an $R_\alpha \in \CAlg^{\an}$ such that  {$(\cat{X}_{/U_\alpha}, \cO_{\cat{X}}|_{U_\alpha})\simeq \Spet(R_\alpha)$.}
	\end{enumerate}
Let $\DM^{\der}\subset \infty\Topp_{\CAlg^{\an}}^{\sHen}$ be the full subcategory of \mbox{derived Deligne--Mumford stacks.}
\end{definition}
Given a   derived scheme $X$ over $k$,  write $(\cat{X},  \cO_{X,\et})$ for the corresponding derived Deligne--Mumford stack, where $\cat{X} = \Shv(X_{\et},\mathcal{S})$. 
We proceed to prove versions  of Propositions \ref{prop:homeo}, \ref{lemma:lift-is-derived-scheme}, and   \ref{lemma:lift-is-derived-scheme-morphisms}  in the setting of derived Deligne--Mumford stacks:
\begin{proposition}\label{prop:homeostackyx}
	Given 
	a pullback of derived Deligne--Mumford stacks
		\[
\begin{tikzcd}
	(\cat{X} , \cO_{X,\et})\arrow{r} \arrow{d}  \arrow[rd, phantom, "\lrcorner", at start]  &  \arrow{d} 	(\cat{Y},\cB) \\
	\Spet(k)\arrow{r} &	\Spet(A),
\end{tikzcd}, \]
	with   $A\in \CAlg^{\an,\art}_{\Lambda}$, the underlying geometric morphism   $\cat{X}  \rightarrow \cat{Y}$ is an equivalence.  
\end{proposition} 
\begin{proof}
	The lower horizontal map is a closed immersion by the derived analogue of \cite[Theorem 4.4]{LurieDAGIX}, and an equivalence on underlying $\infty$-topoi by  \cite[04DY]{stacks}.
By the  derived analogue of \cite[Corollary 10.5]{LurieDAGIX}, it is  a pullback  in  $\infty\Topp_{\CAlg^{\an}}$. The claim follows by combining Propositions 4.3.1.5, 4.3.1.9, and  4.3.1.10 in \cite{LurieHTT}. \end{proof} 
\begin{proposition}\label{lemma:lift-is-derived-scheme-etale}
	Let $A \in \CAlg^{\an,\art}_{W}$ and 
	consider  corresponding squares 
	\[
	\begin{tikzcd}
		A \arrow{r} \arrow{d} & k \arrow{d} \\
		\cB \arrow{r} &\cO_{X,\et}
	\end{tikzcd} \ \ \ \  \ \ \ \  \ \ \ \  \begin{tikzcd}
		(\cat{X}, \cO_{X,\et})\arrow{r} \arrow{d}  &  \arrow{d} 	(\cat{X},\cB) \\
		(\cat{S},k)\arrow{r} &	(\cat{S},A)
	\end{tikzcd}
	\]   
	The left square is a pushout in $\Shv(\cat{X}, \CAlg^{\an})$ 
	if and only if 
	the right square is a pullback  in $\infty\Topp_{\CAlg^{\an}}$.
	If this holds, then
	$(\cat{X}, \cO_{X,\et})\rightarrow 	(\cat{X},\cB)$   \mbox{belongs to $\DM^{\der}$.}
\end{proposition} 

\begin{proof}  
The  equivalence follows by combining 	Propositions 4.3.1.5, 4.3.1.9, and 4.3.1.10 in \cite{LurieHTT}.
	To see that $	(\cat{X}, \cO_{X,\et})\rightarrow 	(\cat{X},\cB)$  lies in $\DM^{\der}$, pick a 
\mbox{factorisation}
\[
A=A_n \rightarrow \cdots \rightarrow A_1 \rightarrow A_0=k
\]
as in \Cref{cor:artinian-small}. Then $A':=A_{n-1}$ sits in a pullback  	$A \xrightarrow{\simeq } k \times_{\sqz_k(V[n+1])} A'$
with $n \geq 0$. We may assume the claim has been proven for $A'$.
The induced square 
	\begin{equation}\label{anotherpushout}	\begin{tikzcd}
(\cat{X}, \cB\otimes_A \sqz_kV[n+1] )
   \arrow{r}\arrow{d} & (\cat{X},	\cB\otimes_A A') \arrow{d} \\
(\cat{X},	\cO_{X,\et})\arrow{r} & 	(\cat{X}, \cB)
\end{tikzcd} \end{equation}
of animated ringed $\infty$-topoi  is a  pushout, cf.\ e.g.\ \cite[Sec. 21.4.4]{LurieSAG}, and the maps 
$$	(\cat{X}, \cO_{X,\et}) \rightarrow (\cat{X}, \cB\otimes_A \sqz_kV[n+1] ) \rightarrow (\cat{X},\cO_{X,\et})$$ 
 both belong to $\DM^{\der}$ since $ \cB\otimes_A \sqz_kV[n+1] \simeq  \cO_{X,\et}\otimes_k \sqz_kV[n+1]$.

	By induction, 
	$(\cat{X}, \cO_{X,\et} )  \rightarrow (\cat{X},	\cB\otimes_A A')$ belongs to $\DM^{\der}$, hence the top map in \eqref{anotherpushout} is local, see   \cite[Remark 1.2.1.7(b)]{LurieSAG}.
This top map is also a closed immersion. By the derived variant of  \cite[Theorem 16.1.0.1]{LurieSAG},  square \eqref{anotherpushout}
	 is a  pushout in $\DM^{\der}$.
\end{proof} 

\begin{proposition}\label{lemma:lift-is-derived-scheme-morphisms-etale}
	Fix $A \rightarrow A'$   in $ \CAlg^{\an,\art}_{\Lambda }$ and  consider corresponding diagrams 	\[
	\begin{tikzcd}
		A \arrow{r} \arrow{d} & 	\arrow[phantom, from=1-1, to=2-2, "\mathrm{(I)}", pos=0.5]	A' \arrow{r} \arrow{d} & k \arrow{d} \\
		\cB \arrow{r} &	\cB' \arrow{r}   &     \cO_X 
	\end{tikzcd}   \ \ \ \  \ \ \ \  \ \    \ \ \ \  
	\begin{tikzcd}
		(\cat{X},\cO_{X,\et})  \arrow{r} \arrow{d}  &  \arrow{d} 	(\cat{X},\cB') \arrow{r} \arrow{d}  &  \arrow{d} 	(\cat{X},\cB) \\
		(\cat{S}, k)\arrow{r} &		(\cat{S}, A')    \arrow{r} &		(\cat{S}, A).	\arrow[phantom, from=1-2, to=2-3, "\mathrm{(II)}", pos=0.5]
	\end{tikzcd}
	\]   
Assume
 both   right  and  outer square
in the left diagram are pushouts  in $\Shv(\cat{X}, \CAlg^{\an})$.

Square (I) is a pushout   \textit{ if and only if  } (II) is a \mbox{pullback in $\infty\Topp_{\CAlg^{\an}}$.} If this holds, then the morphism $(\cat{X},\cB') \rightarrow (\cat{X},\cB)$ belongs to $\DM^{\der}$ and the induced square
\[	\begin{tikzcd}
   \arrow{d} 	(\cat{X},\cB') \arrow{r} \arrow{d}  &  \arrow{d} 	(\cat{X},\cB) \\
 	\Spet(A')  \arrow{r} &	\Spet(A) \arrow[phantom, from=1-1, to=2-2, "\mathrm{(III)}", pos=0.5]
\end{tikzcd}
\]
is a pullback of derived Deligne--Mumford stacks.
\end{proposition}
\begin{proof}
The asserted equivalence between (I) being a pushout and (II) being a pullback follows again by combining 	Propositions 4.3.1.5, 4.3.1.9, and \mbox{4.3.1.10 in \cite{LurieHTT}.}

Hence  assume  that
(II) is a pullback  in $\infty\Topp_{\CAlg^{\an}}$, and let us show that 
the map $(\cat{X},\cB') \rightarrow	(\cat{X},\cB)$ is local and the induced square \mbox{(III) is a pullback in $\DM^{\der}$.}
  \Cref{lemma:lift-is-derived-scheme-etale} implies that
in the right diagram, both   left and   outer square are pullback squares in $\infty\Topp_{\CAlg^{\an}}$ whose top horizontal maps belong to $\DM^{\der}$.

 First,  let us assume that  $(\cat{X},\cB) = \Spet(R)$ is affine, and set $R' : = A' \otimes_A R$.
 As $\Spet$ is a right adjoint, we obtain a pullback square   in $\infty\Topp_{\CAlg^{\an}}^{\sHen}$
 \begin{equation}\label{affinespetpullback}
 	\begin{tikzcd}
 	\Spet(R') \arrow{r} \arrow{d}  &  \arrow{d} 	\Spet(R) \\
	\Spet(A')    \arrow{r} &	\Spet(A),
 \end{tikzcd}
 \end{equation} 
 which is also a pullback in $\DM^{\der}$ by  \cite[Proposition 2.3.21]{LurieDAGV}.

 As in the proof of \Cref{lemma:lift-is-derived-scheme-morphisms}, the map  $\pi_0(R) \rightarrow \pi_0(A') \otimes_{\pi_0(A)} \pi_0(R)  $ induces a homeomorphism on prime spectra.
By \cite[Corollary 4.3.12]{LurieDAGV} and the topological invariance of the \'etale site (cf.\ \cite[Théorème 18.1.2]{EGA4}), the induced map of $\infty$-topoi $ \Shv(\CAlg^{\an,\et}_{R'}, \cat{S}) \rightarrow \Shv(\CAlg^{\an,\et}_{R}, \cat{S})$ is therefore an equivalence, which implies that  \eqref{affinespetpullback} is a pullback in $\infty\Topp_{\CAlg^{\an}}$.
Pasting with the pullback square 
 \begin{equation}\label{affinespetpullback}
	\begin{tikzcd}
		\Spet(A')  \arrow{r} \arrow{d}  &  \arrow{d} 	\Spet(A) \\
	(\cat{S}, A')    \arrow{r} &		(\cat{S}, A),
	\end{tikzcd}
\end{equation} 
in $\infty\Topp_{\CAlg^{\an}}$ shows that   $(X,\cB') \rightarrow (X,\cB)$ is local and (III) is a \mbox{pullback in  $\DM^{\der}$.}

If $(X,\cB)$ is  a general derived Deligne--Mumford stack, we can  pick an effective epimorphism  $\coprod_\alpha U_\alpha \rightarrow \mathbf{1}$ to the final object of $\cat{X}$   such that for each $\alpha$, 
there is an equivalence
$(\cat{X}_{/U_\alpha}, \cB|_{U_\alpha})\simeq \Spet(R_\alpha)$ with  $R_\alpha \in \CAlg^{\an}$. 

A square of sheaves on $\cat{X}$ is a pullback if and only if it is a pullback after restriction to each affine $U_\alpha$.
As the locality of $\cB \rightarrow \cB'$ is a pullback condition, it is therefore implied by our earlier discussion.

We obtain an induced square  (III) in $\infty\Topp_{\CAlg^{\an}}^{\sHen}$ and a comparison morphism to the pullback $\Spet(A')\times_{\Spet(A)} (\cat{X},\cB)$ in $\DM^{\der}$.
By \cite[2.3.22]{LurieDAGV}, this morphism is an equivalence as its pullback to each $\Spet(R_\alpha)$ is one by our earlier discussion.
\end{proof}  
We are now ready to prove  \Cref{prop:equivetale}. We will use the following notation:

\begin{notation} 
As before, we will use the superscript $(-)^{\cocart}$ to denote   subcategories containing all objects and only cocartesian morphisms. We use   $(-)^{\art}$ to denote full subcategories 
containing only those objects  lying over animated rings   $A\in \CAlg_{/k}$ which satisfy  conditions (1) and (2) of  \Cref{def:artinian-animated-ring}.
\end{notation}
\begin{proof}[Proof of \Cref{prop:equivetale}]
	Define $\cat{C}_{\et}$ as in \Cref{constr:def-an} for $\cat{T}$ the \'{e}tale {site of the derived scheme $X$ over $k$.}
Consider the $\infty$-category 
$$\cat{E} =  \CAlg^{\an} \mytimes{\infty\Topp_{\CAlg^{\an}}^{\op}} \Fun(\Delta^1, \infty\Topp_{\CAlg^{\an}}^{\op}) \mytimes{ \infty\Topp_{\CAlg^{\an}}^{\op}} \DM^{\der,\op}$$
  opposite to the $\infty$-category of arrows $$(\cat{Y}, \cB) \rightarrow (\cat{S}, A) $$ with 
$(\cat{Y}, \cB) \in \DM^{\der}$, as well as the $\infty$-category
$$\hspace{-12pt} \widetilde{\cat{E}}  =  \CAlg^{\an} \mytimes{\Fun(\Delta^1,\infty\Topp_{\CAlg^{\an}}^{\op})} \Fun(\Delta^2, \infty\Topp_{\CAlg^{\an}}^{\op}) \mytimes{ \Fun(\Delta^1,  \infty\Topp_{\CAlg^{\an}}^{\op})} 
\Fun(\Delta^1,  \DM^{\der,\op})$$
 opposite to the  $\infty$-category of arrows  
$$(\cat{Y}, \cB) \rightarrow \Spet(A) \rightarrow (\cat{S}, A) $$
with $(\cat{Y}, \cB) \rightarrow \Spet(A)$ a morphism in $\DM^{\der}$.  
Projection to the first coordinate turns $\cat{E}$ and $\widetilde{\cat{E}}$ into   cocartesian fibrations over $\CAlg^{\an}$. The derived scheme $X$ gives two objects  in $\cat{E}$ and 
$\widetilde{\cat{E}}$, which we will both denote by $(k,X)$.

\Cref{prop:homeostackyx}, 
\Cref{lemma:lift-is-derived-scheme-etale},  and 
\Cref{lemma:lift-is-derived-scheme-morphisms-etale} combine to give equivalences $$	(\widetilde{\cat{C}}_{\et})^{\cocart, \art}_{/(k,\cO_X)} \simeq 
	\widetilde{\cat{E}}^{\cocart, \art}_{/(k,X)} \ \ \ \ \  \ \ \ \ \ 	( {\cat{C}}_{\et})^{\cocart, \art}_{/(k,\cO_X)} \simeq 
	{\cat{E}}^{\cocart, \art}_{/(k,X)}.$$
Here, we have used 
the equivalence $\Spet(A) \simeq (\Shv(S_\et, \cat{S}), \cW \otimes_W {A})$ induced by the equivalence of sites $S_{\et} = \Spec(k)_{\et} \simeq \Spec(A)_{\et}$.

	The claim then follows since restriction along $\Delta^{\{0,2\}} \subset \Delta^2$ gives an equivalence of cocartesian fibrations $\widetilde{\cat{E}} \xrightarrow{\simeq} \cat{E}$ by 
  \Cref{prop:adjunction}.
\end{proof}

\newpage
\section{Animated divided power rings}
\label{sec:animated-divided-power}
In this section, we will discuss the basic theory of animated divided power rings, following Mao \cite{Mao2021}. These   play a key role in our definition of the  period map. \vspace{-4pt}

\subsection{Divided power rings} To begin with, recall the following classical definition:

\begin{definition}[Divided power rings, 
{\cite[\href{https://stacks.math.columbia.edu/tag/07GL}{Definition 07GL}]{stacks}}] \label{classicalpdrings}
A \emph{divided power structure} on an ideal $I $ in a (commutative) ring $R$
is given by a collection of maps $(\gamma_n\colon I\rightarrow I)_{n\geq 0}$ such that
for all $n,m \geq 0$ and all $r \in R$ and $x,y \in I$, we have:
\begin{enumerate}
	\item $\gamma_0(x) = 1, \gamma_1(x) = x$;\vspace{2pt} 
	\item $\gamma_n(x) \gamma_m(x) = \frac{(n+m)!}{n!m!} \gamma_{n+m}(x) $;\vspace{2pt} 
	\item $\gamma_n(rx) = r^n \gamma_n(x) $;\vspace{2pt} 
	\item $\gamma_n(x+y) = \sum_{i = 0}^n \gamma_i(x) \gamma_{n-i}(y) $;\vspace{2pt} 
	\item $\gamma_n(\gamma_m(x)) = \frac{(nm)!}{n! (m!)^n } \gamma_{nm}(x) $.
\end{enumerate}
A \emph{divided power ring} is a triple $R=(R,I,\gamma)$. A morphism of divided power rings is a morphism of rings $f\colon R\rightarrow R'$ mapping $I\rightarrow I'$ and such that $\gamma'_n(fx)=f\gamma_n(x)$ for all $x\in I$ and $n\geq 0$. 
Denote the category of divided power rings by
$\PDRing$.
\end{definition}

The axioms express that $\gamma_n(x)$ `behaves like $\frac{x^n}{n!}$'. 

\begin{example}[Free divided power algebras]
\label{exa:free-divided-power-ring}
Given a   ring $R$ and an $R$-module $M$, we define 
the \textit{free divided power algebra $\Gamma_R(M)$ of $M$ over $R$} (cf.\ e.g.\ \cite[Section 3.1]{LurieDAG}) as the quotient of the free commutative $R$-algebra on symbols \[ \{y^{[n]} \  |  \ y \in M, n\geq 0 \}\] by the following relations 
\begin{enumerate}
	\item $y^{[0]}=1$;
	\item $y^{[n]} y^{[m]} =  \frac{(n+m)!}{n!m!} y^{[n+m]}$;
	\item $(r y)^{[n]} =r^ny^{[n]}$;
	\item $(y+ z)^{[n]} = \sum_{i+j = n} y^{[i]} z^{[j]} $;
\end{enumerate}

There is a natural surjection $\Gamma_R(M) \twoheadrightarrow R$ whose kernel $I$ is generated by the symbols $y^{[n]}$ with $n\geq 1$.
The ideal $I$ carries a unique divided power structure $\gamma$
satisfying $\gamma_n(y) := y^{[n]}.$
Hence, we obtain a divided power ring
\[
\Gamma_R(M) := \big(\Gamma_R(M), I, \gamma\big).
\]
Note that $	\Gamma_R(M) = \bigoplus_n \Gamma_R^n(M)$ admits a   multiplicative weight decomposition placing $y^{[n]}$ in weight $n$.
If $M$ is a free $R$-module on generators $y_1,\ldots, y_i$, then 
$\Gamma_R(y_1,\ldots, y_i):= \Gamma_R(M)$ has 
underlying  $R$ module   freely generated by symbols 
$y_1^{[n_1]} \cdot \ldots \cdot y_i^{[n_i]}$ with $n_i\geq 0$.
\end{example}

\begin{example}[Trivial divided power structures]\label{exa:trivialpd}
Given a   ring $R$, the zero ideal carries a unique divided power structure, making $R$ into a divided power ring $(R,0,0)$. This construction identifies the category of   rings with a full subcategory of the category of divided power rings.
\end{example}

\subsection{Animated divided power rings}

Write $\PDPoly\subset \PDRing$ for the full subcategory 
spanned by all divided power rings of the form
\[
\Gamma_{\bZ[x_1,\ldots,x_n]}(y_1,\ldots,y_m)
\]
with $n,m\geq 0$, see \Cref{exa:free-divided-power-ring}. These are the divided power rings of the form $\Sym_\bZ(N) \otimes \Gamma_\bZ (M)$ with $N$ and $M$ finite free $\bZ$-modules.
Following \cite{Mao2021}, we define:

\begin{definition}[Animated  divided power rings]
The $\infty$-category of \emph{animated divided power rings}   $\CAlg^\pd := \cat{P}_\Sigma(\PDPoly)$ is obtained by freely adjoining sifted colimits to $\PDPoly$.
\end{definition}

The  inclusion $\Poly \rightarrow \PDPoly$
induces a fully faithful functor $\CAlg^{\an} \injto \CAlg^{\pd}$. We will often identify $\CAlg^{\an}$ with its essential image, thinking of animated   rings as being equipped with the trivial divided power structure (see also \Cref{exa:trivialpd}).

\begin{notation}[Animated surjections]\label{D0def}
An \textit{animated surjection} is a map of animated rings $R\rightarrow R'$  which induces a surjection after applying $\pi_0$.
We write 	$\Fun(\Delta^{1}, \CAlg^{\an})_{\surj}\subset \Fun(\Delta^{1}, \CAlg^{\an})$ for the full subcategory spanned by all
animated surjections. 
\end{notation}

\begin{proposition}[Forgetful functors]
\label{prop:forgetful-functor}
There is an essentially unique sifted-colimit-preserving functor
\[
\forget\colon \CAlg^{\pd} \rightarrow \Fun(\Delta^1,\CAlg^{\an})_{\surj}
\]
mapping a
divided power ring $(A,I,\gamma)$ in $\PDPoly$ to the ring map $A\rightarrow A/I$. The functor $\forget$ is conservative, and preserves small limits and colimits. 
\end{proposition}

\begin{proof} 
This can be found in \cite[\S 3]{Mao2021}. Since our notation differs from \textit{loc.\,cit.}, we guide the reader in extracting these statements.

By \cite[Theorem 3.23]{Mao2021}, we can write $\Fun(\Delta^1,\CAlg^{\an})_{\surj} \simeq  \cat{P}_{\Sigma}(\SurjPoly)$, where $\SurjPoly \subset  \Fun(\Delta^{1}, \CAlg^{\an})_{\surj}$ consists of  all `standard
surjections' of ordinary polynomial rings of the form   $ \bZ[X_1,\ldots,X_n, Y_1,\ldots,Y_m ] \xrightarrow{}\bZ[X_1,\ldots,X_n].$

The  `classical PD-envelope'  
gives a functor $ \SurjPoly \rightarrow \PDPoly$ which is 
essentially surjective and  preserves coproducts.  
Extending    \mbox{$\SurjPoly \rightarrow  \PDPoly \rightarrow \CAlg^{\pd}$} in a sifted-colimit-preserving way gives  a functor \mbox{$\Env: \Fun(\Delta^1,\CAlg^{\an})_{\surj} \rightarrow  \CAlg^{\pd}$.}
By  \cite[Proposition 2.2]{Mao2021}, $\Env$ admits a right adjoint $$\forget: \CAlg^{\pd} \rightarrow  \Fun(\Delta^{1}, \CAlg^{\an})_{\surj}$$ which is conservative and preserves  small colimits, cf.\ \cite[Proposition 3.34]{Mao2021}. 

Finally, $F$ sends  $(A,I,\gamma)$ in $\PDPoly$ to $A\rightarrow A/I$ by 
\cite[ Proposition 3.17]{Mao2021}.
\end{proof}

\begin{corollary}[Underlying animated rings]
\label{prop:underlying-functor}
There is an essentially unique sifted-colimit-preserving functor
\[
\und\colon \CAlg^{\pd} \rightarrow  \CAlg^{\an}
\]
mapping $(A,I,\gamma)$ in $\PDPoly$ to $A$. It preserves small limits and colimits.  
\end{corollary}

\begin{proof}
The functor is defined by postcomposing 	$F\colon \CAlg^{\pd} \rightarrow \Fun(\Delta^1,\CAlg^{\an})_{\surj}$ with $\ev_0: \Fun(\Delta^1,\CAlg^{\an})_{\surj} \rightarrow\CAlg^{\an} $, $(A\rightarrow A') \mapsto A$. The functor $\ev_0$ admits a left adjoint $A \mapsto (\id: A \xrightarrow{} A)$ and it preserves small colimits,  as this is true for the inclusion $\Fun(\Delta^1,\CAlg^{\an})_{\surj} \hookrightarrow \Fun(\Delta^1,\CAlg^{\an})$.
\end{proof}

\begin{remark}
The (underived) `underlying ring' functor from divided power rings to    rings does not preserve colimits, see \cite[07GY]{stacks} for an example.
\end{remark}

\begin{proposition}\label{prop:und-as-right-adjoint}
The functor $\und$ is right adjoint to $\CAlg^{\an} \injto \CAlg^{\pd}$.
\end{proposition}

\begin{proof}This follows from \cite[Corollary 2.2]{Mao2021} applied to the coproduct preserving functor $\Poly \injto \PDPoly$.
\end{proof}

\begin{notation}[$I_A$ and $\overline{A}$]\label{not:I-and-Q}
We will refer to $\und(A)$ as the \emph{underlying animated ring of $A$}, and will usually just denote it by $A$. We write $\forget(A)=(A\rightarrow \overline{A})$, and denote the (connective) fibre of $A\rightarrow \overline{A}$ by $I_A$. 
There is a fibre sequence
\[
I_A \rightarrow A \rightarrow \overline{A}
\]
in $\Mod_\bZ^\cn$, functorial in $A\in \CAlg^\pd$. For $A=(A,I,\gamma)$ in $\PDPoly$, it is given by the short exact sequence
$
0 \rightarrow I \rightarrow A \rightarrow A/I \rightarrow 0.
$
As its formation preserves sifted colimits,  this  determines the fibre \mbox{sequence in general. }
\end{notation}

\begin{proposition}\label{prop:idempotent}
The construction $A \mapsto \overline{A}$ determines a colimit-preserving idempotent endofunctor of $ \CAlg^{\pd}$.
\end{proposition}
\begin{proof}
The first claim holds by \Cref{prop:forgetful-functor} as the	inclusion $\Fun(\Delta^1,\CAlg^{\an})_{\surj} \rightarrow \Fun(\Delta^1,\CAlg^{\an})$  and the embedding  $\CAlg^{\an} \rightarrow \CAlg^{\pd} $ both
preserve small colimits. The idempotence follows as   $\CAlg^{\an} \rightarrow \CAlg^{\pd} \xrightarrow{\overline{(-)}} \CAlg^{\an}$ is the identity.
\end{proof}

By \cite[Lemma 3.13]{Mao2021} and \cite[Proposition 3.32]{Mao2021}, we have: 
\begin{lemma}[Classical and animated divided power rings]\label{classical}
The functor $\PDPoly \rightarrow \CAlg^{\pd}$ extends to a fully faithful embedding $\PDRing \rightarrow \CAlg^{\pd}$ whose image consists of all $A\in  \CAlg^{\pd}$ for which both $\und(A)$ and $\overline{A}$ are discrete.

\end{lemma}

\begin{remark}
The above functor factors over $ \tau_{\leq 0}\CAlg^{\pd}$, the category  of finite-product-preserving functors $\PDPoly \rightarrow \Set$. Note  that   $\PDRing \rightarrow \tau_{\leq 0}\CAlg^{\pd}$ is \emph{not} an equivalence and $\CAlg^{\pd}$ is \emph{not} the `animation' of  $\PDRing$, see also \cite[\S 3.2]{Mao2021}. 

This is related to the fact that the theory of divided power rings is not an algebraic theory in the style of Lawvere. A divided power ring consists of sets $I$ and $R$ together with various operations satisfying various relations, and the additional condition that the operation $I\rightarrow R$ is injective. If one drops the injectivity condition, then one obtains a Lawvere theory, and its animation is $\CAlg^\pd$. 
\end{remark}

\begin{notation}[Pushouts as tensor products]We will denote the pushout of maps $A\rightarrow B$ and $A\rightarrow C$ in $\CAlg^{\pd}$ by $B\otimes_A C$. By \Cref{prop:underlying-functor}, the underlying animated ring is the pushout in $\CAlg^{\an}$, whose underlying object in $\Mod^{\cn}_\bZ$ is the relative tensor product $B\otimes_A C$.
\end{notation}

\begin{construction}[$\CAlg^{\pd}_A$]\label{constr:functor-CAlg-pd}
For $A$ in $\CAlg^{\pd}$, we write $\CAlg^{\pd}_A := \CAlg^{\pd}_{A/}$. For every map $A\rightarrow A'$ in $\CAlg^{\pd}$, the functor $\CAlg^{\pd}_{A'} \rightarrow \CAlg^{\pd}_A$ admits a left adjoint, given by the pushout $A'\otimes_A -$. Proceeding as in \Cref{constr:functor-CAlg}, we obtain a functor $\CAlg^{\pd}\rightarrow \PrL$ mapping $A$ to $\CAlg^{\pd}_A$. 
\end{construction}

\subsection{Induced divided power structure}\label{subsec:induced-pd-structure}

If $(A,I,\gamma)$ is a classical divided power ring and $A\rightarrow B$ a \emph{flat} morphism of  rings, then the ideal $BI \cong B\otimes_A I$ inherits a divided power structure, given by
\[
\gamma_n(b\otimes x) = b^n \otimes \gamma_n(x),
\]
see \emph{e.g.} \cite[07H1]{stacks}. In this section, we give a derived analogue of this construction.

\begin{notation} [$A_{\triv}$] Consider the composite
\[
\CAlg^{\pd} \overset{\und}\longto \CAlg^{\an} \injto
\CAlg^{\pd},\, A \mapsto A_\triv.
\]
It is the sifted-colimit-preserving extension of the functor mapping an
$A=(A,I,\gamma)$ in $\PDPoly$ to the divided power ring $A_\triv=(A,0,0)$.
The counit of the adjunction of \Cref{prop:und-as-right-adjoint} gives a natural transformation $(-)_\triv \rightarrow \id$. For $(A,I,\gamma)$ in $\PDPoly$, it is given by the identity map $(A,0,0)\rightarrow (A,I,\gamma)$. 
\end{notation}
\begin{definition}[The functor $\ind$]\label{def:induced-divided-power-structure}
Consider the $\infty$-category
\[
\CAlg^{\pd} \times_{\CAlg^{\an}} \Fun(\Delta^{1},\CAlg^{\an})
\]
whose objects are pairs consisting of an animated divided power ring $A$ and a morphism of animated rings $\und(A)\rightarrow B$. We have a functor
\begin{eqnarray*}
	\ind\colon 
	\CAlg^{\pd} \times_{\CAlg^{\an}} \Fun(\Delta^{1},\CAlg^\an),
	&\rightarrow& \Fun(\Delta^1,\CAlg^{\pd}),\, \\
	(A,\und(A)\rightarrow B) &\mapsto& (A\rightarrow \ind_A(B))
\end{eqnarray*}
given by sending $(A,\und(A) \rightarrow B)$ to the following pushout in $\CAlg^{\pd}$
\[
\begin{tikzcd}
	A_\triv \arrow{d} \arrow{r} & B \arrow{d} \\
	A \arrow{r} & \ind_A(B), \arrow[lu, phantom, "\ulcorner", at start]
\end{tikzcd}
\]
where we identify $B$ with an object of $\CAlg^{\pd}$ via  $\CAlg^\an \injto \CAlg^{\pd}$ as usual.
\end{definition}

\begin{lemma}\label{lemma:ideal-of-induced}
Let $A\in \CAlg^{\pd}$ and let $\und(A) \rightarrow B$ be a map in $\CAlg^{\an}$. Then $B':=\ind_A(B)$  satisfies
\begin{enumerate}
	\item the map $B\rightarrow \und_A B'$ is an equivalence in $\CAlg^\an$;
	\item the map $B\otimes_A \overline{A} \rightarrow \overline{B'}$ is an equivalence in $\CAlg^\an$;
	\item the map $B\otimes_A I_A \rightarrow I_{B'}$ is an equivalence in $\Mod_B$.
\end{enumerate}
\end{lemma}

\begin{proof}
Follows from \Cref{prop:forgetful-functor}.
\end{proof}

Alternatively, one could define the functor $\ind$ as the sifted-colimit-preserving extension of the functor that maps a pair consisting of an $A=(A,I,\gamma)$ in $\PDPoly$, and a $B$ in $\Poly_A$ to the divided power ring
\[
\ind_A(B) := (B, B\otimes_A I, \tilde\gamma)
\]
with $\tilde\gamma_n(b\otimes x) = b^n\otimes \gamma_n(x)$,  equipped with the obvious map $A\rightarrow \ind_A(B)$.

\subsection{Divided power cotangent complex}
We recall 	{\cite[\href{https://stacks.math.columbia.edu/tag/07HQ}{07HQ}]{stacks}}:
\begin{definition}[Divided power derivations] \label{def:pd-derivation}
Let $A$ be a   ring and $B=(B,I,\gamma)$ a (classical) divided power $A$-algebra.  An $A$-linear \emph{divided power derivation} of $B$ into a $B$-module $M$ is an $A$-linear derivation $\theta: B \rightarrow M$ satisfying
\[
\theta(\gamma_n(x)) = \gamma_{n-1}(x) \theta(x)
\]
for all $n\geq 1$ and all $x \in I$.
\end{definition}

There is a universal $A$-linear divided power derivation
\[
d = d_{B/A}\colon B \rightarrow \Omega^{1,\pd}_{B/A}
\]
such that for all $B$-modules $M$, precomposition with $d$ identifies the set of $A$-linear divided power derivations $B \rightarrow M$ with $\Hom_B(\Omega^{1,\pd}_{B/A},  M)$, cf.~\cite[\href{https://stacks.math.columbia.edu/tag/07HQ}{Tag 07HQ}]{stacks}.

\begin{remark}
The module $\Omega^{1,\pd}_{B/A}$ is a quotient of the module $\Omega^{1}_{B/A}$ of K\"ahler differentials of the underlying   $A$-algebra $B$ by all relations $d(\gamma_n(x))-\gamma_{n-1} d(x)$.
If $B$ carries  the trivial divided power structure, $\Omega^{1}_{B/A}\rightarrow \Omega^{1,\pd}_{B/A}$ is an isomorphism. 
\end{remark}

\begin{example}\label{exa:pd-differentials-free}
If $B=\Gamma_{A[x_1,\ldots,x_n]}(y_1,\ldots, y_m)$ then $\Omega^{1,\pd}_{B/A}$ is the free $B$-module generated
by $dx_i$ and $dy_i$, 
and the universal derivation satisfies $d y_i^{[n]} = y_i^{[n-1]} dy_i$.
\end{example}

Any commutative square of  rings
\[
\begin{tikzcd}
A   \arrow{r}{ } \arrow{d}{ } & B \arrow{d}{ }  \\
A'   \arrow{r}{ }  & B' 
\end{tikzcd}
\]
with $B\rightarrow B'$ a map of divided power rings induces
a homomorphism \[\Omega^{1,\pd}_{B/A} \rightarrow \Omega^{1,\pd}_{B'/A'}\]
and this makes $\Omega^{1,\pd}_{B/A}$ into a functor of $A\rightarrow B$.  To derive this functor, we will use  compact projective  arrows of animated divided power rings:
\begin{definition}[$\ArrPDPoly$] 
Write  $\ArrPDPoly$ for the ordinary category of maps of divided power rings
of the form
\[
\Gamma_{\bZ[x_1,\ldots, x_n]}(y_1,\ldots, y_m) \longrightarrow
\Gamma_{\bZ[x_1,\ldots, x_{n+n'}]}(y_1, \ldots, y_{m+m'}).
\]
\end{definition}

\begin{lemma}The natural functor $\ArrPDPoly \rightarrow \Fun(\Delta^{1}, \CAlg^{\pd})$ induces an equivalence $\cat{P}_\Sigma(\ArrPDPoly) \simeq  \Fun(\Delta^{1}, \CAlg^{\pd})$.
\end{lemma}

\begin{proof}See \cite[Proposition A.16]{Mao2021}.  
\end{proof}

\begin{definition}[Cotangent complex for animated divided power rings]\label{def:cotanipd}
Let us write $\CAlg^{\pd}\Mod^{\cn}$ for the $\infty$-category of pairs $(A,M)$ consisting of an animated divided power ring $A$ and an object $M$ of $\Mod_A^{\cn}$. More formally, we have
\[
\CAlg^{\pd}\Mod^{\cn} := 
\CAlg^{\pd} \times_{\CAlg^{\an}} \CAlg^{\an}\Mod^{\cn}
\]
The \textit{divided power cotangent complex functor}
\[
\underline{L}^\pd: \Fun(\Delta^{1}, \CAlg^{\pd}) 
\rightarrow \CAlg^{\pd}\Mod^{\cn}
\]  
is the unique sifted-colimit-preserving 
extension of the  functor  sending
$A \rightarrow B$ in $\ArrPDPoly$ to the pair $(B, \Omega^{1,\pd}_{B/A})$.
\end{definition}

\begin{remark}	\label{thirdrem}
As  in \Cref{not:extend}, 
we see that  the composition
\[
\Fun(\Delta^{1}, \CAlg^{\pd}) \xrightarrow{\underline{L}^\pd} \CAlg^{\pd}\Mod^{\cn} \xrightarrow{} \CAlg^{\pd}
\]
sends $A\rightarrow B$ to $B$.  
Hence, we can  write $\underline{L}^\pd_{B/A} = (B, L_{B/A}^\pd)$ for some   \mbox{$L_{B/A}^\pd\in \Mod_B^{\cn}$.}
\end{remark}

\begin{definition}[Hurewicz Map] \label{rem:hurewicz}
	Given a morphism $A \rightarrow B$ in $\CAlg^{\pd}$, there is a canonical map $B \rightarrow L_{B/A}^{\pd}$, the \textit{universal derivation},   together with a nullhomotopy of the composite $A \rightarrow B \rightarrow L_{B/A}^{\pd}$. Both are constructed by extending from $\ArrPDPoly$  under sifted colimits. The \textit{Hurewicz map} is the induced map  of $A$-modules $B/A \rightarrow L_{B/A}^{\pd}$, where $B/A$ denotes the cofibre of $A\rightarrow B$.
\end{definition} 
\begin{proposition}\label{prop:non-pd-cotangent-complex}
If $A\rightarrow B$ is a map of animated rings, interpreted as a map of animated divided power rings, then ${L}_{B/A}^\pd \simeq L_{B/A}$. 
\end{proposition}

\begin{proof}
By construction, the composition
\[
\Fun(\Delta^1,\CAlg^{\an}) \injto \Fun(\Delta^{1},\CAlg^{\pd})
\overset{\underline{L}^\pd}\rightarrow \CAlg^{\pd}\Mod^{\cn} \rightarrow \CAlg^{\an}\Mod^{\cn}
\]
agrees with $\underline{L}$ on $\ArrPoly$ and both functors preserve sifted colimits.
\end{proof}

\begin{proposition}\label{prop:cotangent-preserves-colimits}
The functor $\underline{L}^\pd$ preserves small colimits.
\end{proposition}

\begin{proof}
It suffices to show that the functor preserves finite coproducts in $\ArrPDPoly$. This follows from an explicit computation using \Cref{exa:pd-differentials-free}.
\end{proof}

As observed by Mao (see \cite[Examples ~2.32,2.25]{Mao2021}), 
\Cref{transitivity0} and \Cref{cotcxbasechange0} admit the following immediate generalisations:
\begin{corollary}[Transitivity sequence] \label{transitivity}
Given a diagram $A \rightarrow B \rightarrow C$ of animated divided power rings, there is a cofibre sequence in $\Mod_C$:
\[
C \otimes_{B} L_{B/A}^\pd \rightarrow L_{C/A}^\pd 
\rightarrow L_{C/B}^\pd.
\]
\end{corollary}

\begin{corollary}[Base change] \label{cor:cotangent-base-change}
Let 
\[
\begin{tikzcd}
	A \arrow{r} \arrow{d} & B \arrow{d} \\
	A' \arrow{r} & B' \arrow[lu, phantom, "\ulcorner", at start]
\end{tikzcd}
\]
be a pushout square in $\CAlg^{\pd}$. Then following natural map is an equivalence:
\[
A'\otimes_A L_{B/A}^\pd \rightarrow L_{B'/A'}^\pd.
\]

\end{corollary}
Together with \Cref{localisation0}, these two propositions imply:

\begin{proposition}[Localisation]\label{localisation1}
Given  $A \in \CAlg^{\pd}$  and $ f\in \pi_0(A)$, we have  $L_{A[f^{-1}] /A } = 0, $ where 
$A[f^{-1}] = A_{\triv}[f^{-1}]\otimes_{A_{\triv}}A$.
\end{proposition}

We  exhibit a right adjoint to the functor $\underline{L}^\pd$, analogous to \Cref{prop:cotangent-adjoint}.

\begin{definition}[Divided power trivial square-zero functor]
\label{def:pd-square-zero}
The functor
\[
\underline{\sqz}^\pd \colon \CAlg^{\pd}\Mod^{\cn} \rightarrow 
\Fun(\Delta^1,\CAlg^{\pd})
\]
is the sifted-colimit-preserving extension of the functor that maps a pair consisting of $(R,I,\gamma)$ in $\PDPoly$ and $M$ a finite free $R$-module to
\[
(R,I,\gamma) \rightarrow (R\oplus M, I\oplus M, \gamma'),
\]
where the multiplication vanishes on $M$, and the divided powers are given by
\[
\gamma'_n(x+y) := \gamma_n(x) + \gamma_{n-1}(x) y, \ \ x \in I, y\in M.
\] 
\end{definition}

We write $\underline{\sqz}^\pd(R,M) = (R \rightarrow \sqz_R^\pd M)$.

\begin{proposition}\label{prop:pd-cotangent-adjoint}
The functor $\underline{\sqz}^\pd$ is a right adjoint to $\underline{L}^\pd$. 

\end{proposition}

\begin{proof}
The proof is completely analogous to the proof of \Cref{prop:cotangent-adjoint}. For every map $A\rightarrow B$ of (discrete) divided power rings, we have a commutative square
\[
\begin{tikzcd}
	A \arrow{d} \arrow{r} & B \arrow{d}{\iota + d} \\
	B \arrow{r}{\iota} & \sqz_B^{\pd}( \Omega^{1,\pd}_{B/A})
\end{tikzcd}
\]
of divided power rings, where $\sqz_B^{\pd}( \Omega^{1,\pd}_{B/A}) = B\oplus \Omega^{1,\pd}_{B/A}$ has the divided power structure used in \Cref{def:pd-square-zero}. 
For $A\rightarrow B$ in $\ArrPDPoly$, the vertical maps in the square 
define a map
\[
\underline{u}\colon (A\rightarrow B) \rightarrow \underline{\sqz}^\pd (\underline{L}^\pd_{B/A}),
\]
functorial in $A\rightarrow B$, and extending by sifted colimits, we obtain a   transformation  $\id \rightarrow \underline{\sqz}^\pd \circ \underline{L}^{\pd}$  defined on all of $\Fun(\Delta^1,\CAlg^\pd)$.

It now suffices to show for all $A\rightarrow B$ in $\Fun(\Delta^1,\CAlg^\pd)$ and $(R,M)$ in $\CAlg^\pd\Mod^\cn$ the induced map
\[
\varphi\colon \Map_{}\big(\underline{L}_{B/A},\,(R,M)\big) 
\xrightarrow{\underline{u}\circ \underline{\sqz}^\pd}
\Map_{}\big((A\rightarrow B),\, \underline\sqz^\pd(R,M)\big)
\]
is an equivalence. First assume that $A\rightarrow B$ is of the form 
\[
\Gamma_{\bZ[X_1,\ldots, X_n]}(Y_1,\ldots, Y_m) \longrightarrow
\Gamma_{\bZ[X_1,\ldots, X_{n+n'}]}(Y_1, \ldots, Y_{m+m'}),
\]
and also that $(R,I,\gamma)$ belongs to $\PDPoly$ and that $M$ is finite free. By \Cref{exa:pd-differentials-free}, we have $\Omega^{1,\pd}_{B/A} \simeq B^{n'+m'}$, and the source of $\varphi$ is the discrete space
\[
R^{n+n'} \times I^{m+m'} \times M^{n'+m'}.
\]
Similarly, the target computes to the discrete space
\[
R^n\times I^m \times (R\oplus M)^{n'} \times (I\oplus M)^{m'}.
\]
Chasing the definitions, we see that $\varphi$ corresponds to the obvious bijection. It   follows as in  \Cref{prop:cotangent-adjoint} that $\varphi$ is an equivalence 
for all $A\rightarrow B$ and $(R,M)$.
\end{proof}
\begin{notation}
Denote the counit of the above adjunction by $\underline{v}: \underline{L}^{\pd}\circ \underline{\sqz}^\pd \rightarrow \id$.  As in \Cref{rem:firstcounit}, its first component 
on some $(B,M) \in \CAlg^\pd\Mod^\cn$
is the augmentation $\epsilon:  \sqz^{\pd}_B(M) \rightarrow B$. We write $v\colon L_{\sqz^\pd_B(M)/A}^\pd \rightarrow \epsilon^\ast M$ for the map in $\Mod_{\sqz^{\pd}_B(M)}$ induced by the second component of the counit.
\end{notation}
With the same proof as in \Cref{cor:derivations-sqz},  we obtain:

\begin{corollary}\label{cor:pd-derivations-sqz}
Given $A\rightarrow B$ in $\CAlg^{\pd}$ and $M\in \Mod_B^{\cn}$, there is an equivalence
\[ 	\Map_{\CAlg^{\pd}_{A//B}}(B,\,\sqz^{\pd}_B(M))    \simeq \Map_{B}(L^{\pd}_{B/A},\,M)\]
sending  \vspace{2pt}  $\phi: L^{\pd}_{B/A} \rightarrow M$ to the composite $B \xrightarrow{\id \oplus d} \sqz^{\pd}_B(L^{\pd}_{B/A})  \xrightarrow{\sqz(\phi)} \sqz^{\pd}_B(M)$ and 
$s: B \rightarrow \sqz^{\pd}_B(M)$ to
\mbox{$L^{\pd}_{B/A} \rightarrow s^\ast L^{\pd}_{\sqz^{\pd}_B(M)/A} \rightarrow s^\ast L^{\pd}_{\sqz^{\pd}_B(M)/B}  \rightarrow s^\ast \epsilon^\ast M \simeq M $.}
\end{corollary}

\subsection{Cotangent complex of sheaves of divided power rings} Let $\cat{T}$  be a 
 small $\infty$-category   with a Grothendieck topology. 
Proceeding as in \S\ref{subsec:cotangent-sheaves}, we \mbox{obtain a functor}
\[
\underline{L}^\pd \colon
\Fun(\Delta^1,\Shv(\cat{T},\CAlg^\pd)) \rightarrow 
\Shv(\cat{T},\CAlg^\pd\Mod),
\]
and we have $\underline{L}^\pd_{\cB/\cA} = (\cB, L^\pd_{\cB/\cA})$. Completely analogously to \Cref{prop:cotangent-sheaf-adjunction}, we obtain the following sheafy version of \Cref{cor:pd-derivations-sqz} (using analogous notation):
\begin{proposition}\label{prop:pd-cotangent-sheaf-adjunction}
Let $\cA\rightarrow \cB$ be a map in $\Shv(\cat{T},\CAlg^{\pd})$, and $\cM\in \Mod^{\cn}_\cB$, there is an equivalence
\[ 		\Map_{\Shv(\cat{T},\CAlg^{\pd})_{\cA//\cB}}\big(
\cB, \sqz^{\pd}_\cB (\cM)
\big)  \simeq 	\Map_{{\cB}}(L^{\pd}_{\cB/\cA},\, \cM)\]
sending  \vspace{2pt}  $\phi: L_{\cB/\cA}^{\pd} \rightarrow \cM$ to the composite $\cB \xrightarrow{\id \oplus d} \sqz^{\pd}_{\cB}(L^{\pd}_{\cB/\cA})  \xrightarrow{\sqz(\phi)} \sqz^{\pd}_{\cB}(\cM)$ and 
$s: \cB \rightarrow \sqz_{\cB}^{\pd}(\cM)$ to
\mbox{$L^{\pd}_{\cB/\cA} \rightarrow s^\ast L^{\pd}_{\sqz^{\pd}_\cB(\cM)/\cA} \rightarrow s^\ast L^{\pd}_{\sqz^{\pd}_\cB(\cM)/\cB}  \rightarrow s^\ast \epsilon^\ast \cM \simeq \cM $.}
\end{proposition}

\begin{notation}\label{not:affinepdstructuresheaf}
Given a  map $A \rightarrow B $ in $\CAlg^{\pd}$, write $\cO^{ }_B$ for the $\CAlg^{\pd}_A$-valued sheaf on $\Spec(B)$ sending a distinguished affine $U_f$ to $B[f^{-1}]$, cf.\ \Cref{sheafifyingonbases}.
\end{notation}

We proceed as in \Cref{cotforaffine} to check consistency with \Cref{def:cotanipd}: 
\begin{proposition} \label{cotforaffinepd}
Fix a  map of animated  divided power  rings $A \rightarrow B $. Then  $L^{\pd}_{\cO^{ }_B/A }$ is the quasi-coherent sheaf on $\Spec(B) $ with global sections 
$L^{\pd}_{B/A }$ .
\end{proposition} 

\subsection{Divided power cotangent complex of a surjective map}
Let $I\subset R$ be an ideal in a  ring and $\gamma$ a divided power structure on $I$. Given  $n\geq 1$,    write $I^{[n]}$ for the ideal generated by all divided powers $\gamma_m(x)$ with $x\in I $ and $m\geq n$. 

\begin{proposition}\label{prop:pi1-of-pd-cotangent-complex}
Let $A\rightarrow B$ be a surjective map in $\PDRing$ such that the induced map $A/I_A\xrightarrow{\cong} B/I_B$ is an isomorphism. Denote the kernel of \mbox{$A\rightarrow B$ by $I$.}
Then $\pi_0 L_{B/A}^\pd\cong0$ and $\pi_1 L_{B/A}^\pd\cong I/I^{[2]}$, functorially in $A\rightarrow B$.  

Moreover, the map $I = \pi_1(B/A) \rightarrow \pi_1 L_{B/A}^\pd\cong I/I^{[2]}$
induced by the Hurewicz map
(see \Cref{rem:hurewicz})  is given by projection.  
\end{proposition}

The rest of this subsection is dedicated to the proof of \Cref{prop:pi1-of-pd-cotangent-complex}. 
\begin{notation}
Given a divided power ring $A=(A,I_A,\gamma)$, we write $\PDRing^\aug_A$ for the category of divided power algebras $A\rightarrow B$ such that $A/I_A\rightarrow B/I_B$ is an isomorphism.	
More generally, given  an animated divided power ring $A \in \CAlg^{\pd}$, we write $\CAlg^{\pd,\aug}_A$ for the $\infty$-category of maps of animated divided power rings $A\rightarrow B$ such that the induced map $\overline{A} \rightarrow \overline{B}$ is an equivalence.
\end{notation}

Given an animated divided power ring $A \in \CAlg^{\pd}$, the composite  \[R:\CAlg^{\pd,\aug}_A \xrightarrow{I}\Mod_{\bZ}^{\cn} \xrightarrow{\Omega^\infty} \cat{S}\]
is conservative and  moreover preserves small limits and sifted colimits. 
It therefore admits a left adjoint $L$, and by the Barr--Beck--Lurie theorem \cite[4.7.3.5]{LurieHA}, the resulting adjunction is in fact monadic. Let us denote the associated monad by $T = RL$. For every $B \in \CAlg^{\pd,\aug}_A$, we obtain an  augmented simplicial diagram 
\begin{equation} \label{eq:monadicres}( \  \ldots \  \substack{\longrightarrow \vspace{-4pt} \\ \leftarrow   \vspace{-4pt} \ \\ \longrightarrow  \vspace{-4pt} \\ \leftarrow  \vspace{-4pt} \ \\ \longrightarrow} \  T(T(B))\   \substack{  \longrightarrow  \vspace{-4pt} \\ \leftarrow  \vspace{-4pt} \ \\ \longrightarrow} \ T(B)) \rightarrow B  \end{equation} which exhibits $B$ as a geometric realisation of objects in the image of $L$.

We will now describe this resolution more explicitly when  $A$ is discrete.
\begin{construction} Let $A \in \PDRing$ be a   divided power ring. 
Given a set $S$, we write $A\{S\}$ for the free $A$-module with generators $\{s \in S\}$. We 
then  equip the ring $\Gamma_A (A\{S\})$ (see \Cref{exa:free-divided-power-ring}) with the divided power ideal $I_A\oplus \Gamma^+_A (A\{S\})$, and the natural divided power structure $\gamma'$ given by
\[
\gamma'_n(x , y) = 
\bigg(\gamma_n(x),\, \sum_{i=0}^{n-1} \gamma_i(x) \gamma_{n-i}(y)
\bigg) \ ,  \ \ x\in I_A, y\in \Gamma^+_A (A\{S\}).
\] 

\end{construction}

\begin{lemma}\label{lemma:augmented-divided-power-left-adjoint}
The functor $S   \mapsto \Gamma_A (A\{S\})$ is naturally equivalent to the composite 
\[\Set \rightarrow \cat{S} \xrightarrow{L} \CAlg^{\pd,\aug}_A.\]
\end{lemma}
\begin{proof} First,  note that the functor $\Fin \rightarrow \CAlg^{\pd}$, $S \mapsto \Gamma_{\bZ}(\bZ\{S\})$ preserves finite coproducts, where $\Fin $ is the category of finite sets.		Left Kan extending along $\Fin \rightarrow \cat{S} \simeq \cat{P}_{\Sigma}(\Fin)$
gives the left adjoint of an 
adjoint pair
$\cat{S} \leftrightarrows \CAlg^{\pd},$
and the right adjoint sends $B \in  \CAlg^{\pd}$ to $\Map_{ \CAlg^{\pd}}( \Gamma_{\bZ}(\bZ\{S\}), B)$, see e.g.\ \cite[Lemma 2.1]{Mao2021}.
But 	$\Map_{ \CAlg^{\pd}}( \Gamma_{\bZ}(\bZ\{S\}), B) $ is naturally equivalent to $ \Map_{\cat{S}}(S, \Omega^{\infty}IB)$; this is checked by hand for $B \in \PDPoly$, and extends under sifted colimits as each $ \Gamma_{\bZ}(\bZ\{S\}) \in \CAlg^{\pd}$ is compact projective.

To prove the result, we factor the right adjoint $R$ as a composite of right adjoints \[\CAlg_A^{\pd,\aug} \subset \CAlg_A^{\pd} \rightarrow \CAlg^{\pd} \xrightarrow{B \mapsto \Omega^{\infty} I_B}\cat{S}.
\]
Using that $\Gamma_{\bZ}(\bZ\{S\}) $ is a free $\bZ$-module, we deduce that the left adjoint sends a finite set $S\in \Fin$ to $A\otimes \Gamma_{\bZ}(\bZ\{S\}) \simeq \Gamma_{A}(A\{S\})$. Extending under filtered colimits, we obtain the result for general sets.\vspace{-2pt} 
\end{proof}

We prove an analogue of   \cite[08RA]{stacks}
in the divided power setting:\vspace{-2pt} 
\begin{lemma}\label{lemma:pd-ideals-key-computation}
Let $A\rightarrow P \rightarrow B$ be surjective maps in $\PDRing$, with $P\simeq \Gamma_A(M)$ for $M = A\{S\}$ a free $A$-module on a set $S$.
Assume that $A/I_A\rightarrow B/I_B$ is an isomorphism. Denote the kernel of $A\rightarrow B$ by $I$, and the kernel of $P\rightarrow B$ by $J$. Then the universal divided power derivation $d\colon P \rightarrow \Omega^{1,\pd}_{P/A}$ induces a short exact sequence of \vspace{-3pt} $B$-modules: 
\[
0 \rightarrow I/I^{[2]} \rightarrow J/J^{[2]} \xrightarrow{[x] \mapsto 1 \otimes dx} B\otimes_P \Omega^{1,\pd}_{P/A} \rightarrow 0.
\]
\end{lemma}

\begin{proof}

Given $s\in S$, we write $x_s \in P$ for the corresponding generator.

By the surjectivity assumption, we can pick an element $\lambda_s \in A  $ whose image under $A \rightarrow B$ agrees with the image of $x_s$ under $P \rightarrow B$.
Hence $(A \rightarrow P \rightarrow B)$ is isomorphic to $(A \rightarrow P \xrightarrow{x_s \mapsto 0} B)$ via the isomorphism $P \rightarrow P$ sending $x_s $ to $x_s + \lambda_s$, and we may assume without restriction that 
$P\rightarrow B$ sends each   $x_s$ to zero.

Note that $\Omega^{1,\pd}_{P/A} \cong P\otimes_A M$. 
We also observe that $
J^{[2]} = I^{[2]} \oplus IM \oplus \Gamma^{\geq 2}_A(M)$
inside
$
J = I \oplus M \oplus \Gamma^{\geq 2}_A(M).$
Hence the sequence in the lemma  is isomorphic to\vspace{-2pt} 
\[
0 \rightarrow I/I^{[2]} \rightarrow I/I^{[2]} \oplus (B\otimes_A M) \rightarrow
B\otimes_A M \rightarrow 0.\vspace{-14pt}
\] 
\end{proof}

\begin{proof}[Proof of \Cref{prop:pi1-of-pd-cotangent-complex}]
Given   $B \in \PDRing^\aug_A$, \Cref{lemma:augmented-divided-power-left-adjoint} shows that each term  in the monadic resolution   $P_\bullet \rightarrow B$ of $B$ in \eqref{eq:monadicres}  is of the form $\Gamma_A(M)$ for $M = A\{S\}$ a free $A$-module.
\Cref{lemma:pd-ideals-key-computation} gives a   short exact sequence
\begin{equation} \label{eq:auxses}
	0 \rightarrow I/I^{[2]} \rightarrow J_\bullet/J_\bullet^{[2]} \rightarrow B\otimes_{P_\bullet} \Omega^{1,\pd}_{P_\bullet/A} \rightarrow 0
\end{equation}
of simplicial $B$-modules. To prove $\pi_0(L_{B/A}^{\pd}) \cong 0$ and 
$\pi_1(L_{B/A}^{\pd}) \cong I/I^{[2]}$, 
it suffices to show that the geometric realisation $|J_\bullet/J_\bullet^{[2]}|$  is $2$-connective.

First,   note that  $|J_\bullet| \simeq 0$ as $|P_\bullet| \rightarrow B$ is an equivalence; in particular   $\pi_0(|J_\bullet| )\cong 0$.

Applying the normalised chain complex functor to $J_\bullet$, we deduce   that the map $d_1: \ker(d_0:J_1 \rightarrow J_0) \rightarrow J_0$ is surjective.
Using \Cref{exa:free-divided-power-ring}, this shows that \[\Gamma_A^{\geq 2}(d_1): \ker(\Gamma_A^{\geq 2}(d_0):\Gamma_A^{\geq 2}(J_1) \rightarrow \Gamma_A^{\geq 2}(J_0)) \xrightarrow{\ \ \ \  } \Gamma_A^{\geq 2}(J_0)
\]
is also surjective, which in turn implies that  $\pi_0(|\Gamma_A^{\geq 2}(J_\bullet)|) \cong 0$.

But for each $n$, we have a surjection $\Gamma_A^{\geq 2} J_n \rightarrow J_n^{[2]}$. Denoting its kernel by $K_n$, we obtain a short exact sequence of simplicial $A$-modules
\[
0 \rightarrow K_\bullet \rightarrow  \Gamma_A^{\geq 2} J_\bullet
\rightarrow J_\bullet^{[2]} \rightarrow 0.
\]
Using the associated long exact sequence, we deduce that 
$\pi_0|J_\bullet^{[2]}| \cong 0$.

Using the 
long exact sequence associated to  
\[
0 \rightarrow J^{[2]}_\bullet \rightarrow J_\bullet \rightarrow J/J^{[2]}_\bullet \rightarrow 0, 
\]
we conclude from $|J_\bullet| \simeq 0$ and  $\pi_0|J_\bullet^{[2]}| \cong 0$ 
that $|J/J^{[2]}|$ is $2$-connective. 

The final claim about compatibility with projection follows from a diagram chase as \eqref{eq:auxses} receives a map from the short exact sequence  $ 0 \rightarrow I \rightarrow J_\bullet \rightarrow P_{\bullet}/A 
\rightarrow 0$.
\end{proof}

\section{Derived divided power de Rham complex}
\label{sec:derived-pd-de-rham}
In this section, we define the derived de Rham cohomology of divided power rings and compare it with crystalline cohomology, following  Mao~\cite{Mao2021} \mbox{and Bhatt \cite{Bhatt2012}.}
We will  set out the theory only in the level of generality needed for our main theorems; a more comprehensive  study of the derived de Rham cohomology of derived divided power schemes will appear in \cite{PDPeriod24}.

\subsection{Definition and basic properties}
In  general, derived de Rham cohomology
can behave badly, 
and 	we cannot control it  via the cotangent complex as the 
Hodge filtration need not be complete.
We obtain a  nicer theory by    {either} completing with respect to the Hodge filtration {or}  working in a $p$-completed setting, where 
the Cartier isomorphism is at our disposal.  The former approach is more suitable  in characteristic zero, whereas the latter provides a well-behaved	  theory in the mixed characteristic setting of interest to us.

So let us fix a prime  $p$ and work in a $p$-completed setting.

\begin{notation}
Given an animated ring $A \in \CAlg^{\an}$, 
write \[\Mod_{A}^\wedge \subset \Mod_{A}\] for the full subcategory spanned by all (derived) $p$-complete $A$-modules, i.e.\ those  $M$ for which the morphism $M \rightarrow \lim_n (M/p^n)$ is an equivalence (see    
\mbox{\cite[Section 7.3]{LurieSAG}}).	The inclusion $\Mod_{A}^\wedge  \subset \Mod_{A}$ admits a left adjoint, the (derived) completion functor $(-)^{\wedge}$, and 
\cite[Corollary 7.3.5.2]{LurieSAG} shows that there  is an essentially unique symmetric monoidal structure  ${\otimes}_A$ on $\Mod_{A}^\wedge$ for which $(-)^{\wedge}$ is  symmetric monoidal.

Let $\CAlg^{\an} \Mod_{ }^{ \wedge } \subset \CAlg^{\an} \Mod_{ }^{ }$ be the full subcategory spanned by all $(A,  B)$ for which $B \in \Mod_{A}^\wedge$ is 	  $p$-complete.  
\end{notation}

\begin{convention}\label{convention:completion}
In Sections 5-7, we will  often suppress the completion functor $(-)^{\wedge}$ on objects from our notation, 
denoting $M \in \Mod_A$ and  $M ^{\wedge} \in \Mod_A^{\wedge}$  by the same name. 
We will, however,  carefully distinguish between  $ \Mod_A$ and $ \Mod_A^{\wedge}$. 
\end{convention}
\begin{remark}
If $\pi_0(A)$ is a local Artinian ring with residue field $k$  of characteristic $p$, then every $A$-module is automatically $p$-complete since $p$ is nilpotent.
Completing at $p$ is,  however, a nontrivial operation over  bases like $W(k)$, where it `corrects' pathological behaviours exhibited by the uncompleted derived de Rham complex.
\end{remark}

\begin{definition}[Filtered $p$-complete modules]
Let $\bN$ be the partially ordered set of non-negative integers. 
We consider  the $\infty$-category
\[
\CAlg^\an\Mod^{\wedge \fil} := 
\CAlg^\an \times_{\Fun(\bN^{\op},\CAlg^\an)} \Fun(\bN^{\op},\CAlg^\an \Mod^{\wedge})
\]
of pairs $(A,M)$ with $A$ in $\CAlg^\an$ and 
\[
M = \big[ \cdots \rightarrow \Fil^1 M \rightarrow \Fil^0 M \ ] \in \Mod_A^{\wedge \fil}:= \Fun(\bN^{\op}, \Mod^{\wedge}_A).
\] 
\end{definition}

\begin{definition}[de Rham complex of classical divided power rings]\label{classdr}
Let $A\rightarrow B$ be a morphism of classical divided power rings. Given $n\geq 0$, set $\Omega^{n,\pd}_{B/A} := \Lambda^n_B \Omega^{1,\pd}_{B/A}$. 
The (uncompleted) \emph{de Rham complex} of $A\rightarrow B$ is the	
  chain complex of $A$-modules
\[
\Omega^{\bullet,\pd}_{B/A} = \left[ \cdots \rightarrow 0 \rightarrow
B \overset{d}\rightarrow \Omega^{1,\pd}_{B/A} 
\overset{d}\rightarrow \Omega^{2,\pd}_{B/A} 
\overset{d}\rightarrow \cdots
\right]
\]
with the usual de Rham differential 
\[
d(f_0 \wedge df_1 \wedge \ldots \wedge df_n) = df_0 \wedge df_1 \wedge \ldots \wedge df_n.
\]
We equip this complex with the \emph{Hodge filtration}, whose $n^{th}$ piece $\Fil^m \Omega^{\bullet,\pd}$ is  
\[
\left[ \  0 \rightarrow \cdots 0 \rightarrow 
\Omega^{m,\pd}_{B/A} \rightarrow \Omega^{m+1,\pd}_{B/A} \rightarrow \cdots
\right].
\]
The filtered  chain complex $\Omega^{\bullet,\pd}_{B/A}$ is functorial in the morphism $A\rightarrow B$. 
\end{definition}

\begin{definition}[Derived de Rham complex for animated divided power rings \cite{Mao2021}]\label{deRhamaffine}
The ($p$-completed)  \textit{Hodge-filtered de Rham complex}
\[
\underline{\dR}^\pd  \colon  \Fun(\Delta^{1}, \CAlg^{\pd}) \rightarrow \CAlg^{\an}\Mod^{\wedge\fil}
\]
is the unique sifted-colimit-preserving extension of the composition 
\[
\ArrPDPoly \xrightarrow{} \CAlg^{\an}\Mod^\fil \rightarrow \CAlg^{\an}\Mod^{\wedge \fil}
\] 
of the functor $(A\rightarrow B) \mapsto (A,  \Omega^{\bullet,\pd}_{B/A})$ in \Cref{classdr} and the $p$-completion functor.
\end{definition}

\begin{remark}	\label{thirdremprime}
As in \Cref{thirdrem}, we check that the composition
\[
\Fun(\Delta^{1}, \CAlg^{\pd}) \rightarrow \CAlg^{\an}\Mod^{\wedge \fil} 
\rightarrow \CAlg^{\an}
\]
sends $A\rightarrow B$ to $A$. We can therefore write 
$\underline{\dR}^\pd_{B/A} = (A, \dR_{B/A}^\pd)$. 
\end{remark}

\begin{convention}\label{convention:filtration}
We denote the Hodge filtration on $\dR_{B/A}^\pd$ by $\Fil^\bullet \dR_{B/A}^\pd$.
We will sometimes suppress the  functor $\Fil^0$ on objects from our notation, 
but will 	carefully distinguish between $ \Mod_A^{\wedge \fil}$ and $ \Mod_A^{\wedge}$.
\end{convention}

\begin{warning}
In \cite{Mao2021}, the above filtration is referred to as the `stupid filtration' (and the term `Hodge filtration' is reserved for something else). Moreover, note that  the symbol $\dR_{B/A}^{\pd}$ in  \cite{Mao2021} stands for the \textit{un-$p$-completed} derived de Rham complex.
\end{warning}

Proceeding as in  the proof of \Cref{prop:non-pd-cotangent-complex}, we obtain:
\begin{proposition}\label{prop:non-pd-dR-complex}
If $A\rightarrow B$ is a map of animated rings (seen as animated divided power rings), then $\dR_{B/A}^\pd$ agrees with
Bhatt's $p$-completed variant of Illusie's 
 Hodge-filtered derived de Rham cohomology $\dR_{B/A}$ (cf.\ \cite[Definition 2.1]{Bhatt2012}, \cite[VIII.2]{Illusie2}). 
\end{proposition}

\begin{proposition}[Base change]\label{prop:dR-base-change}
Every pushout square 
\[
\begin{tikzcd}
	A \arrow{r} \arrow{d} & B \arrow{d} \\
	A' \arrow{r} & B'\arrow[lu, phantom, "\ulcorner", at start]
\end{tikzcd}
\]
in $\CAlg^{\pd}$ induces a natural equivalence $A'\   {\otimes}_A \dR_{B/A}^\pd \rightarrow \dR_{B'/A'}^\pd$  in $\Mod_{A'}^{\wedge \fil}$.
\end{proposition}

\begin{proof}The $\infty$-category of such pushout squares is generated under sifted colimits by those diagrams for which both $A\rightarrow A'$ and $A\rightarrow B$ belong to  $\ArrPDPoly$ (it then follows that $A'\rightarrow B'$ and $B\rightarrow B'$ also lie in $\ArrPDPoly$). On such diagrams, the claim follows from an explicit computation. Since everything in sight preserves sifted colimits, the proposition follows.
Alternatively, the proposition follows from \cite[Lemma 4.12]{Mao2021} and \cite[Example  2.25]{Mao2021}, see also \Cref{cor:cotangent-base-change}.
\end{proof}

Applying this to the construction of \Cref{def:induced-divided-power-structure}, we obtain:

\begin{corollary}\label{cor:dR-induced-PD-structure}
For $A\in \CAlg^{\pd}$ and $A\rightarrow B$ a map in $\CAlg^\an$, the natural map
\[
\dR_{B/A_\triv}^{\pd} \rightarrow \dR^{\pd}_{\ind_A(B)/A}
\]
is an equivalence in $\Mod_{A}^{\wedge \fil}$. 
\end{corollary} 
\begin{definition}[Derived exterior powers]
Let 
\[
\Lambda^n \colon \CAlg^\an\Mod^{\cn} \rightarrow \CAlg^{\an}\Mod^{\cn}
\]
be the unique sifted-colimit-preserving extension of the functor that maps a pair $(A,M)$ in $\PolyProj$ to $(A,\Lambda^n_A M)$. We write $\Lambda^n(A,M) = (A, \Lambda^n_A M)$.
\end{definition}

\begin{remark}
By \cite[Section 3]{BrantnerMathew2019}, this can be extended to non-connective modules.
\end{remark}
Using \Cref{convention:completion} as always, we have: 

\begin{proposition}[Associated graded of the Hodge filtration]
\label{prop:hodge-graded} 
Given a map of animated divided power rings $A \rightarrow B$, there is an equivalence
\[
\gr^n \dR_{B/A}^\pd \simeq   \Lambda^n(L_{B/A}^\pd)[-n]
\]
in $\Mod_A^{\wedge}$. This defines an equivalence of functors
$\Fun(\Delta^{1},\CAlg^{\pd}) \rightarrow \CAlg^{\an}\Mod^{\wedge}$.
\end{proposition}

\begin{proof}
Both functors preserve sifted colimits, and they agree on $\ArrPDPoly$.
\end{proof} 

In particular, we have maps $\pi^0\colon \dR_{B/A}^\pd \rightarrow B$ and $\pi^1\colon \Fil^1 \dR_{B/A}^\pd \rightarrow L_{B/A}^\pd[-1]$.

\begin{proposition}\label{prop:dR-differential}
There is a map of fibre sequences
\[
\begin{tikzcd}
	\Fil^1\dR_{B/A}^\pd \arrow{r} \arrow{d}{\pi^1}
	& \dR_{B/A}^\pd \arrow{r}{\pi^0} \arrow{d} 
	& B \arrow{d}{d} \\
	L_{B/A}^\pd[-1] \arrow{r} & 0 \arrow{r} & L_{B/A}^\pd
\end{tikzcd}
\]
in $\Mod_A^{\wedge}$,  functorial in $A\rightarrow B$ in $\Fun(\Delta^{1}, \CAlg^{\pd})$.
\end{proposition}

\begin{proof}
For $A\rightarrow B$ in $\ArrPDPoly$, we  construct this diagram by hand, using  that the differential $d\colon B\rightarrow \Omega^{1,\pd}_{B/A}$ in the cochain complex $\Omega^{\bullet,\pd}_{B/A}$ is the universal divided power derivation $B\rightarrow \Omega^{1,\pd}_{B/A}$.  We then extend under sifted colimits.
\end{proof}

The  de Rham complex 
in \Cref{deRhamaffine} admits a multiplicative refinement $$\dR^{\pd}:\Fun(\Delta^1,\CAlg^{\pd}) \rightarrow 
\CAlg^{\an} \CAlg^{\wedge \fil}_{\bZ_p}, $$
where $\CAlg^{\an} \CAlg^{\wedge \fil}_{\bZ_p} \subset \CAlg^{\an} \times_{\CAlg} \Fun(\Delta^1, \CAlg_{\bZ}^{\fil})$ denotes the 
 full subcategory spanned by pairs $(A,B)$ of an animated ring $A$ and a $p$-complete filtered $\EE_\infty$-$A$-algebra $B$.
We first construct this functor on objects in $\ArrPDPoly$ using the wedge product of forms, and then extend under sifted colimits as usual.

We will give a more refined treatment of multiplicative refinements in \cite{PDPeriod24}. All we need in this paper is the following result, which follows from \cite[Lemma 4.12]{Mao2021}:
\begin{proposition}\label{prop:refinement}
	The functor $\dR^{\pd}:\Fun(\Delta^1,\CAlg^{\pd}) \rightarrow 
	\CAlg^{\an} \CAlg^{\wedge \fil}_{\bZ_p}$ preserves small colimits.
\end{proposition}

\subsection{Global version}
Let $\cat{T}$ be a Grothendieck site. Then the divided power de Rham functor induces a functor
\[
\underline{\dR}^\pd\colon \Fun(\Delta^1,\Shv(\cat{T},\CAlg^{\pd})) \rightarrow
\Shv(\cat{T},\CAlg^\an\Mod).
\]
The image of $\underline{\dR}_{\cB/\cA}^\pd$ in $\Shv(\cat{T},\CAlg^\an)$ is $\cA$, and we can  write $\underline{\dR}_{\cB/\cA}^\pd = (\cA,\dR_{\cB/\cA}^\pd)$. 
\begin{notation}[$\dR^\pd_{X/A}$]\label{not:dR-X}
Let $A$ be an animated divided power ring. Given a 
 pair $X=(\cat{T},\cO_X)$ of a small Grothendieck site $\cat{T}$ and a sheaf $\cO_X$ in $\Shv(\cat{T},\CAlg^\pd_A)$, we set $\dR^\pd_{X/A} := \dR^\pd_{\cO_X/A}$. 
\end{notation}
Using \Cref{not:affinepdstructuresheaf}, we check that this  is consistent with \Cref{deRhamaffine}:
\begin{proposition}\label{affdeRhamchains}
Fix a map  $A \rightarrow B$ in $\CAlg^{\pd}$, 
and write \mbox{$q:\Spec(B) \rightarrow \Spec(A)$} for the induced map of affine derived schemes.	
Then $q_\ast \dR^{\pd}_{\cO_B/A}$ is the quasi-coherent $p$-complete filtered $\cO_A$-module  with  \mbox{global sections  $\dR^{\pd}_{B/A}$.}
\end{proposition} 
To prove this result, the following lemma will be useful:\vspace{-1pt}
\begin{lemma}[Cotangent criterion for de Rham complexes] \label{cotcriteriondR}
Fix    maps $A\rightarrow A' \rightarrow B$ in $\CAlg^{\pd}$.
If 	$L_{\overline{A'}/\overline{A}} \simeq 0 $ and
$L^{\pd}_{A'/A} \simeq  0 $, then   $ \dR_{B/A}   \xrightarrow{   }   \dR_{B/A'}  $
is   an equivalence. 
\end{lemma} 
\begin{proof} 
As $L^{\pd}_{A'/A}\simeq 0$,  the transitivity sequence \Cref{transitivity} and \Cref{prop:hodge-graded} show that the  map induces an equivalence on  graded pieces of the Hodge filtration.

It therefore suffices to prove that $\Fil^0 	\dR_{B/A}   \xrightarrow{   }  \Fil^0  \dR_{B/A'}  $ is an equivalence. As we work in a $p$-completed setting, it is by \Cref{prop:dR-base-change}
enough to check that  \[\Fil^0 \dR_{B\otimes_{ } \FF_p/A\otimes_{ } \FF_p}   \xrightarrow{   } \Fil^0   \dR_{B\otimes_{ } \FF_p/A'\otimes_{ } \FF_p}\] is an equivalence.
Domain and target admit compatible 
exhaustive  {conjugate filtrations} (\cite[3.4]{Mao2021}).
The associated graded pieces  of this filtration are a functor of the map $L_{\overline{B}/\overline{A} } \rightarrow L_{\overline{B}/\overline{A}'}$, which is an equivalence as 
 $L_{\overline{A'}/\overline{A}} \simeq 0 $ by assumption.
\end{proof}

\begin{proof}[Proof of \Cref{affdeRhamchains}]

We first check that  the
presheaf $U \mapsto \mathcal{F}(U) := 	   \mathbf{\dR}_{\cO_{B}(U) / A} $ already satisfies the sheaf property on 
distinguished affines $U_1, \ldots, U_n \subset \Spec(B)$, i.e.\  that for $U = \cup_{i=1}^n U_i$,  the following  map  is an equivalence: \begin{equation} \displaystyle \mathcal{F}\left(U \right) \rightarrow \lim_{\emptyset \neq S \subset \{1,\ldots,n \}} \mathcal{F}\left(\cap_{i \in S} U_i \right).\end{equation}
The equivalence on graded pieces of the Hodge filtration follows from 
\Cref{prop:hodge-graded} , Proposition 25.2.3.1 in  \cite{LurieSAG}, and \Cref{cotforaffinepd}.

By $p$-completeness, it is enough to check 
the equivalence for  $\mathcal{F}^0 \simeq \Fil^0\mathbf{\dR}_{\cO_{ B}(-) / A} $ after applying $\FF_p \otimes (-)$.	
Here, it follows by observing that for a distinguished affine open $U_f$, we can identify $\FF_p \otimes   \Fil^0\mathbf{\dR}_{\cO_{B}(U_f) / A}\simeq   \Fil^0\mathbf{\dR}_{ \FF_p \otimes  B[f^{-1}] / \FF_p \otimes A} $ with
$$   \phi^\ast_{A \twoheadrightarrow \overline{A}}(\FF_p \otimes  \overline{B}[f^{-1}]) \otimes_{\phi^\ast_{A \twoheadrightarrow \overline{A}}(\FF_p \otimes \overline{B})} \Fil^0 \dR_{ \FF_p \otimes  B / \FF_p \otimes A}, $$
where the maps 	$ \phi^\ast$ are defined as  in \cite[Section 4.3]{Mao2021}. Indeed, this is readily checked on associated gradeds of the (exhaustive) conjugate filtration.

\Cref{sheafifyingonbases} therefore shows that sheafification does not change the value on  {distinguished affines,} and we obtain an equivalence
$\dR_{\cO_B/A}(U_f) \simeq    \mathbf{\dR}_{B[f^{-1}] / A}$.

To prove that $q_\ast \dR_{\cO_B/A}$ is  quasi-coherent  on $\Spec(A)$, 
fix an element $g \in \pi_0(A)$ and write $\overline{g} \in \pi_0(B)$ for the image of $g$ under $\pi_0(A) \rightarrow \pi_0(B)$.  
Using \Cref{prop:dR-base-change} and the equivalence  $\cO_{B}(  U_{\overline{g}})  \simeq  B[\overline{g}^{-1}]   \simeq   B \otimes_A A[g^{-1}]$,  we compute
\[ q_\ast \dR_{\cO_B/A} (U_g) \simeq 
\dR_{\cO_B/A} (U_{\overline{g}}) \simeq 
\dR_{ B \otimes_A A[{g}^{-1}]/A}  
\simeq  \dR_{ B/A} \otimes_A A[{g}^{-1}].
\]
Quasi-coherence  then follows from \Cref{affineqccrit}. 
\end{proof}

\subsection{Comparison with de Rham and crystalline cohomology}
We begin by recalling the following result (see  \cite[Corollary 3.10]{Bhatt2012}): 

\begin{proposition}[Bhatt]\label{prop:dR-comparison-bhatt}
If $A\rightarrow B$ is a smooth map of discrete commutative
$\bZ/p^n$-algebras, then $\dR_{B/A}^\pd \simeq \dR_{B/A} \simeq \Omega^\bullet_{B/A}$ in $\Mod_A^{\wedge \fil}$. 
\end{proposition}

Next, we state a remarkable `crystalline' property of the functor $\dR$, due to Mao. Recall from \Cref{not:I-and-Q} that for any animated divided power ring $B$, we have an animated ring $\overline{B}$. Interpreting $\overline{B}$ as an object of $\CAlg^{\pd}$ (equipped with the trivial divided power ideal), we obtain a map $B\rightarrow \overline{B}$ in $\CAlg^{\pd}$. This defines a functor
\[
\Fun(\Delta^{1},\CAlg^{\pd}) \rightarrow \Fun(\Delta^{1},\CAlg^{\pd}),\,
(A\rightarrow B) \mapsto (A\rightarrow \overline{B}).
\]
More formally, this can be defined as the sifted-colimit-preserving extension of the functor that maps $(A,I,\gamma) \rightarrow (B,J,\delta)$ in $\ArrPDPoly$ to $(A,I,\gamma) \rightarrow (B/J,0,0)$.
There is a natural map $(A\rightarrow B) \rightarrow (A\rightarrow \overline{B})$, functorial in $A\rightarrow B$.

\begin{theorem}[{\cite[4.16]{Mao2021}}]\label{thm:crystalline-invariance-dR}
For every map $A\rightarrow B$ in $\CAlg^{\pd}$, the induced map
\[
\dR_{B/A}^\pd \rightarrow   \dR_{\overline{B}/A}^\pd
\]
is an equivalence in $\Mod_A^{\wedge}$.   
\end{theorem}
\begin{warning}
  \Cref{thm:crystalline-invariance-dR} is a statement about the unfiltered de Rham complex, and fails if we take Hodge  filtrations into account.
\end{warning}

Now let $k$ be a perfect field of characteristic~$p$, and denote by $W$ the ring of Witt vectors of $k$. Equip $W$ with the unique divided power structure on the ideal $pW$. 

\begin{definition}[Derived crystalline cohomology, {\cite[Definition 4.17]{Mao2021}}]
\emph{Derived crystalline cohomology} is the functor
$
\rR\Gamma_{\cris}(-/W)\colon \CAlg^{\an}_{k} \rightarrow \Mod_W^{\wedge}
$ 
defined as  
\[
\CAlg^{\an}_{k} \injto \CAlg^{\pd}_{k} \longto
\CAlg^{\pd}_{W} \xrightarrow{\ \dR^\pd\ }\Mod_W^{\wedge},
\]
where the middle map is obtained by precomposing with $W\rightarrow k$.
\end{definition}

We then have the following comparison  (see
\cite[Proposition 4.87, 4.90]{Mao2021}):

\begin{proposition}\label{prop:crystalline-comparison}
The restriction of $\rR\Gamma_{\cris}(-/W)$ to smooth $k$-algebras is isomorphic to the crystalline cohomology functor of Berthelot  \cite{Berthelot}. 
\end{proposition}

We can globalise this result:
\begin{proposition}\label{cor:crystalline-comparison}
If $X/k$ is a smooth and separated scheme over $k$, then $\RGamma(X,\dR_{X/W}^\pd)$ is quasi-isomorphic to Berthelot's  \cite{Berthelot} crystalline cohomology $\RGamma_{\cris}(X/W)$.
\end{proposition}

\begin{proof}
The affine crystalline site $\Cris(X/W)$ of $X$ has objects
\begin{equation}\label{crisdiag}
	(U/S) = \left( \begin{tikzcd}
		\Spec(\overline{A}) \arrow{r} \arrow[hookrightarrow]{d}    & \Spec(A) \arrow{dd} \\
		X  \arrow{d}  & \\
		\Spec(k) \arrow[hookrightarrow]{r} & \Spec(W) 
	\end{tikzcd}\right).
\end{equation}
Here  $(A,I,\gamma)$ is a divided power ring under $W$ with  $\overline{A} = A/I$, the underlying map of topological spaces of 
$\Spec(\overline{A}) \rightarrow \Spec(A)$ is a homeomorphism, and 
$  \Spec(\overline{A}) \subset X$ is a Zariski open subscheme.
The $0^{th}$ piece of the Hodge filtration  gives  a map
\begin{equation}\label{eq:trunc} (U/S) \ \ \mapsto \ \   (\dR_{A/W}^\pd \rightarrow A ) \end{equation}
of $\Mod_W^{\wedge}$-valued  sheaves on $\Cris(X/W)$.

We have a natural map $u_{X/S}: \Cris(X/S) \rightarrow \Zar(X)$ to the affine Zariski site, sending the diagram $(U/S)$ as in \eqref{crisdiag} to $\Spec(\overline{A})
\subset  X$.  
Right Kan extending along   $u_{X/S}$ gives a
functor 
$u_{X/S\ast}$ from crystalline to Zariski sheaves. It sends    $\mathcal{F} $ to the sheaf
$ U \mapsto \lim_{V/S \in \Cris(U/S)^{\op}} \mathcal{F}(V)$.

By \Cref{affdeRhamchains} and  \Cref{thm:crystalline-invariance-dR}, applying  $u_{X/S\ast}$ to \eqref{eq:trunc} gives a map of sheaves on the affine Zariski site sending $ U\subset X$ to  
$
\dR_{\cO_X(U)/W}^\pd  \xrightarrow{\simeq } \RGamma_{\cris}(U/W),
$ which is an equivalence by  \Cref{prop:crystalline-comparison}.
The result follows by separability. 
\end{proof}

\begin{remark}
This result in fact holds without the separability assumption, and  for locally quasi-syntomic morphisms of finite type; we refer to \cite{PDPeriod24} for   details.
\end{remark}

\section{The period map}
\label{sec:period-map} 
In this section, we will use deformations of  the Hodge filtration to 
construct a  period map, and moreover compute its effect on tangent fibres. This will play a crucial role in the proof of  \Cref{mainthm:btt}, where it will enable us to   `transport' 
divided power unobstructedness from one formal moduli problem to another.

\subsection{Divided power formal moduli problems}

Let $\Lambda$ be a local Noetherian ring with residue field $k$ and a  chosen divided power structure on \mbox{its augmentation ideal.}

\begin{definition}[Artinian divided power algebras]
We call $A\in \CAlg^{\pd}_\Lambda$ \mbox{ \emph{Artinian} if}
\begin{enumerate}
	\item the map $k \cong \overline{\Lambda} \rightarrow \overline{A}$ is an equivalence;
	\item the underlying object in $\CAlg^{\an}_{\Lambda//k}$ is Artinian (see \Cref{def:artinian-animated-ring}).
\end{enumerate}
Let  $\CAlg^{\pd,\art}_\Lambda \subset \CAlg^\pd_\Lambda$ be the full subcategory spanned by the Artinian objects 
\end{definition}

Informally, an object of $\CAlg^\pd_{\Lambda}$ is an Artinian animated   $\Lambda$-algebra with 
residue field $k$, equipped with a divided power structure on its augmentation ideal. 

\begin{definition}\label{def:pd-fmp}
A functor $F\colon \CAlg_\Lambda^{\pd,\art} \rightarrow \cat{S}$ is a \emph{divided power formal moduli problem} if it satisfies:
\begin{enumerate}
	\item $F(k)$ is contractible;
	\item for every pullback square  		
	\[
	\begin{tikzcd}
		A' \arrow{r}{} \arrow{d}  \arrow[rd, phantom, "\lrcorner",  at start]  &A \arrow{d} \\
		B'  \arrow{r} & B
	\end{tikzcd}
	\]
	in $\CAlg^{\pd, \art}_W$ for which the maps $\pi_0 A \rightarrow \pi_0 B$ and $\pi_0 B' \rightarrow \pi_0 B$ are surjective, applying $F$ gives a pullback square in $\cat{S}$.
\end{enumerate}
Write $ \Moduli^\pd_\Lambda \subset\Fun(\CAlg_\Lambda^{\pd,\art},\cat{S}) $ for the full subcategory of divided power formal moduli problems.
\end{definition}

Since the functor $\und\colon \CAlg^\pd \rightarrow \CAlg^\an$ preserves limits, we have:

\begin{lemma}\label{lemma:restriction-of-fmp-to-pd-rings}
If $F\colon \CAlg_\Lambda^{\an,\art}\rightarrow \cat{S}$ is a formal moduli problem, then 
\[
\CAlg_{\Lambda}^{\pd,\art} \rightarrow \CAlg_\Lambda^{\an,\art} \overset{F}\rightarrow \cat{S}
\]
is a divided power formal moduli problem.  
\end{lemma}

Deforming sheaves of animated divided power rings gives another example.

\begin{construction}[The functor $\Def_X^\pd$]\label{constr:def-pd}
Let $X=(\cat{T},\cO_X)$ with $\cat{T}$ be a small Grothendieck site and $\cO_X$ 
an object of $\Shv(\cat{T},\CAlg^\pd_k)$. Consider the $\infty$-cagegory
\[
\cat{D} := \CAlg^\pd \times_{\Shv(\cat{T},\CAlg^\pd)} 
\Fun(\Delta^1,\Shv(\cat{T},\CAlg^{\pd}))
\]
and the projection $q\colon \cat{D}\rightarrow \CAlg^{\pd}$. Analogously to \Cref{constr:def-an}, we obtain a left fibration 
$q\colon (\cat{D}^{\cocart})_{/(k,\cO_X)} \rightarrow \CAlg^\pd_{/k}$
and a functor $\Def^\pd_X \colon \CAlg^{\pd,\art} \rightarrow \widehat{\cat{S}}.$
\end{construction}

\begin{remark}[Informal description]
The objects of the space $\Def^\pd_X(A)$ are    pushout squares in $\Shv(\cat{T},\CAlg^\pd)$  of the form
\[
\begin{tikzcd}
	A \arrow{r} \arrow{d} & k \arrow{d} \\
	\cB \arrow{r} & \cO_X\arrow[lu, phantom, "\ulcorner", at start].
\end{tikzcd}
\]
\end{remark}
\noindent
Recall from \Cref{not:map-prime} that  $\Map'(x,y)$ is the space of sections of a given  $y\rightarrow x$.

\begin{proposition}\label{prop:defpd}
The functor $\Def^\pd_{X}$ is a divided power formal moduli problem, and for every perfect $V$ in $\Mod_k^\cn$ there is a natural equivalence
\[
\Omega\Def^\pd_{X}(\sqz_k^\pd V) \simeq 
\Map'_{\Shv(\cat{T},\CAlg^{\pd}_k)}\big( \cO_X,\,
\sqz_k^\pd V\otimes_k \cO_X \big),
\]
where   $\sqz_k^\pd V$ is defined as in \Cref{def:pd-square-zero}.
\end{proposition}

\begin{proof}
The proof is essentially the same as for Propositions~\ref{prop:defan-tangent-space} and~\ref{prop:defan}. We only elaborate on the point in which the argument differs. Consider a pullback diagram
\[
\begin{tikzcd}
	A \arrow{r} \arrow{d}  \arrow[rd, phantom, "\lrcorner",  at start]& A_0 \arrow{d} \\
	A_1 \arrow{r} & A_{01}
\end{tikzcd}
\]
in $\CAlg^\pd$ with $\pi_0 A_1 \rightarrow \pi_0 A_{01}$ and $\pi_0 A_1 \rightarrow \pi_0 A_{01}$ surjective, and such that all maps induce equivalences after applying the functor $\overline{(-)}$. 
Writing $\cat{C}_A := \Shv(\cat{T},\CAlg^\pd_A)$, we have an adjunction
\[
L\colon \cat{C}_A \leftrightarrows \cat{C}_{A_0} \times_{\cat{C}_{A_{01}}} \cat{C}_{A_1} : R,
\]
and we need to verify that the unit and counit are equivalences. By \Cref{prop:forgetful-functor}, it suffices to check that they are equivalences on underlying sheaves of animated rings, and that they are equivalences after applying the functor $\overline{(-)}$. The first follows as in the proof of \Cref{prop:defan}, and the second follows trivially since all maps involved are equivalences.
\end{proof}
Using \Cref{prop:pd-cotangent-sheaf-adjunction}, we deduce the following analogue of  \Cref{cor:defan-tangent-space}:

\begin{corollary}\label{cor:defpd-tangent}
$\Def^\pd_{X}(\sqz_k^\pd V)\simeq
\Map_{\Mod(X,\,\cO_X)}(L^{\pd}_{X/k},\,V[1]\otimes_k \cO_X)$.  
\end{corollary}

\begin{construction}[$\ind\colon \Def_X \rightarrow \Def_X^\pd$]
\label{constr:ind-an-to-pd}
Consider the $\infty$-category
\[
\cat{C} := \CAlg^\pd \times_{\Shv(\cat{T},\CAlg^\an)} 
\Fun(\Delta^1,\Shv(\cat{T},\CAlg^\an))
\]
whose objects are given by pairs $(A,\cB)$ with $A$ in $\CAlg^\pd$ and $\cB$ in $\Shv(\cat{T},\CAlg^\an_{\und (A)})$. Similarly, consider the $\infty$-category
\[
\cat{D} := \CAlg^\pd \times_{\Shv(\cat{T},\CAlg^\pd)} 
\Fun(\Delta^1,\Shv(\cat{T},\CAlg^\pd))
\]
whose objects are pairs $(A,\cB)$ with $A$ in $\CAlg^\pd$ and $\cB$ in $\Shv(\cat{T},\CAlg^\pd_A)$.

The functor $\ind$ defined in Section \ref{subsec:induced-pd-structure} induces a functor $\ind\colon \cat{C} \rightarrow\cat{D}$ over $\CAlg^\pd$ mapping $(A,\cB)$ to $(A,\ind_A \cB)$,  where $\ind_A \cB$ is the sheafification of the presheaf $U \mapsto \ind_A \cB(U)$. For every $A\rightarrow A'$ in $\CAlg^\pd$ and $\cB$ in $\Shv(\cat{T},\CAlg^\pd_A)$, the   map
\[
A'\otimes_A \ind_A(\cB) \rightarrow \ind_{A'}(\und(A')\otimes_{\und(A)}\cB)
\]
is an equivalence, so  $\ind\colon \cat{C} \rightarrow \cat{D}$ preserves  arrows that are cocartesian over $\CAlg^\pd$.

Now if $X=(\cat{T},\cO_X)$ with $\cO_X$ in $\Shv(\cat{T},\CAlg^\an_k)$, then using \Cref{constr:def-an} and \Cref{constr:def-pd}, we obtain a map $\ind\colon \Def_X \rightarrow \Def_X^\pd$ 
in   $\Moduli_W^\pd$. 
\end{construction}

\subsection{Construction of the period map}
The functor $\overline{(-)}$ induces a functor
\begin{equation}\label{eq:functorial-Q}
\Fun(\Delta^1,\CAlg^\pd) \rightarrow \Fun(\Delta^1,\CAlg^\pd)
\end{equation}
informally given by mapping $A\rightarrow B$ to the composite $A \rightarrow B \rightarrow \overline{B}$. More formally, this can be defined by extending by sifted colimits from $\ArrPDPoly$.

\begin{setup}\label{setup:coherentspace}
Throughout this section, we fix a   \textit{coherent} topological space $X$, i.e.\ a   topological space 
for which the collection of quasi-compact open subsets is closed under finite intersections and 
forms a basis of the topology.
\end{setup}

\begin{construction}[The functor $\theta$]
Write $\cK$ for the category $\bullet \overset{\sim}\leftarrow \bullet \rightarrow \bullet$ \mbox{and consider}
\[
\cat{E}:= \CAlg^\pd \times_{\Fun(\cK,\Shv(X,\CAlg^\an))}
\Fun(\cK,\Shv(X,\CAlg^\an\Mod^{\wedge})), 
\]
the $\infty$-category   of pairs consisting of an $A$ in $\CAlg^\pd$ and a diagram $N_0 \overset{\sim}\leftarrow N \rightarrow M$ in $\Shv(X,\Mod_A^{\wedge})$. Using  \eqref{eq:functorial-Q}, one constructs a functor $\theta\colon \cat{D} \rightarrow \cat{E}$
\mbox{informally given by}
\[
\theta\colon 
\big(A,\,\cB\big) \mapsto 
\big( A,\, \dR^{\pd}_{\overline{\cB}/A} \isomfrom 
\dR^{\pd}_{\cB/A} \rightarrow \cB \big),
\]
where $\overline{\cB}$ denotes the sheafification of the presheaf $U \mapsto \overline{\mathcal{B}(U)}$.

To see that the left map is an equivalence, we note that by 
\Cref{cons:sheavesadjunctions}
and \Cref{prop:refinement}, the map $
\dR^{\pd}_{\cB/A} \rightarrow  \dR^{\pd}_{\overline{\cB}/A} $ is obtained by sheafifying the map 
of presheaves $U \mapsto 
 (\dR^{\pd}_{\cB(U)/A} \rightarrow  \dR^{\pd}_{\overline{\cB(U)}/A})$, which is an equivalence by \vspace{2pt}
\Cref{thm:crystalline-invariance-dR}.

We then consider also the $\infty$-category
\[
\cat{F} := \CAlg^\pd \times_{\Fun(\cK,\CAlg^\an)}
\Fun(\cK,\CAlg^\an\Mod^{\wedge}),
\]
whose objects are pairs   of an $A$ in $\CAlg^\pd$ and a diagram
$N_0 \overset{\sim}\leftarrow N \rightarrow M$ in $\Mod_A^{\wedge}$. 

Taking global sections induces a functor $\RGamma\colon \cat{E}\rightarrow \cat{F}$, where we have precomposed with the canonical map  from  $R$ to  the global sections $R\Gamma(X,R)$ of the corresponding constant sheaf on $X$. We obtain two functors lying over $\CAlg^\pd$:
\[
\cat{D} \overset{\theta}\longto \cat{E}\overset{\RGamma}\longto \cat{F}.
\] 
\end{construction}

\begin{construction}[The fundamental diagram] 
Thinking of  $\overline{(-)}$  from  \Cref{not:I-and-Q} as a functor 
$: \CAlg^\pd \rightarrow \CAlg^\pd$,  we construct the $\infty$-category
\[
\cat{D}' := \CAlg^\pd \times_{\Shv(X,\CAlg^\pd)} 
\Fun(\Delta^1,\Shv(X,\CAlg^\pd))
\]
whose objects are pairs $(A,\cB)$ with $A\in \CAlg^\pd$ and $\cat{B}$ in $\Shv(X,\CAlg^{\pd}_{\overline{A}})$.	
We have a  functor 
$\overline{(-)} : \cat{D} \rightarrow \cat{D}'$
sending a pair $(A,\cB)$   to the pair 
$(A,\overline{\cB})$.
Let 
\[
\cat{E}':=  \CAlg^\pd \times_{\Shv(X,\CAlg^\an)}
\Shv(X,\CAlg^\an\Mod^{\wedge})
\]
be the    $\infty$-category of pairs $(A,N_0)$ with $A\in\CAlg^\pd$ and $N_0$  in $\Shv(X,\Mod_A^{\wedge})$, and 
\[
\dR^{\pd}: \cat{D}' \rightarrow \cat{E}'
\]
the functor informally given by $(A,\cB) \mapsto \dR^{\pd}_{\cB/A}$. Finally, consider the $\infty$-category
\[
\cat{F}' := \CAlg^\pd \times_{\CAlg^\an}
\CAlg^\an\Mod^{\wedge} \simeq \CAlg^\pd\Mod^{\wedge}
\]
and the functors $\ev_0\colon \cat{E}\rightarrow \cat{E}'$ and $\cat{F}\rightarrow\cat{F}'$ that map $(A,N_0 \overset{\sim}{\leftarrow} N \rightarrow M)$ to $(A,N_0)$. Summarising all constructions so far gives a \mbox{commutative diagram over $\CAlg^\pd$:}
\begin{equation}\label{eq:fundamental-diagram-of-fibrations}
	\begin{tikzcd}
		\cat{C}\arrow{r}{\ind} 
		& \cat{D} \arrow{d}{\overline{(-)}} \arrow{r}{\theta}
		& \cat{E}\arrow{d}{\ev_0} \arrow{r}{\RGamma}
		& \cat{F} \arrow{d}{\ev_0} \\
		& \cat{D}' \arrow{r}{\dR^{\pd}} 
		& \cat{E}' \arrow{r}{\RGamma}
		& \cat{F}'
	\end{tikzcd}.
\end{equation} 
\end{construction}

Tracing a pair $(A,\cB)$ in $\cat{D}$ under the maps in the left square in (\ref{eq:fundamental-diagram-of-fibrations}) gives
\[
\begin{tikzcd}
(A,\, \cB) \arrow[|->]{r} \arrow[|->]{d} & 
\big( A,\, \dR^{\pd}_{\overline{\cB}/A} \isomfrom
\dR^{\pd}_{\cB/A} \rightarrow \cB \big) \arrow[|->]{d} \\
(A,\, \overline{\cB}) \arrow[|->]{r} & (A,\, \dR^{\pd}_{\overline{\cB}/A} ).
\end{tikzcd}
\]

\begin{lemma}\label{lemma:fundamental-diagram-cocartesian}
If $f\colon A\rightarrow A'$ is a map in $\CAlg^\pd$ such that  $\overline{A} \rightarrow \overline{A'}$ is an equivalence, then all functors in diagram~(\ref{eq:fundamental-diagram-of-fibrations}) preserve cocartesian arrows over $f$.\vspace{-2pt}
\end{lemma}

In particular, this applies to all maps $f\colon A\rightarrow A'$ in $\CAlg^{\pd,\art}_W$. \vspace{-1pt}

\begin{proof}[Proof of \Cref{lemma:fundamental-diagram-cocartesian}]  For  $\ind$, this has already been observed in \Cref{constr:ind-an-to-pd}.	

Given any    $ \cB \in \Shv(X,\CAlg^{\pd}_A)$, we 
note that by  \Cref{prop:idempotent}, all maps in the following square become equivalences after applying the functor $\overline{(-)}$: 
\begin{equation} \label{sq:bars}
	\begin{tikzcd}[column sep=36pt]
		\arrow{r} \arrow{d}  
		A' \otimes_A  \cB	\arrow{r}{} &  A' \otimes_A \overline{ \cB}
		\arrow{d}{ } \\
		\overline{A'} \otimes_{\overline{A}}  \cB\arrow{r}{} 
		&  \overline{A'} \otimes_{\overline{A}}  \overline{\cB} \simeq  \hspace{-10pt}&\hspace{-32pt} \overline{A' \otimes_A \cB}
	\end{tikzcd}
\end{equation}
Together with \Cref{prop:dR-base-change} and \Cref{thm:crystalline-invariance-dR}, this implies the claim for    $\theta$.

Note that an arrow $(A,\cB)\rightarrow (A',\cB')$ in $\cat{D}'$ is cocartesian over $\CAlg^\pd$ if and only if the map $\overline{A'} \otimes_{\overline{A}} \cB \rightarrow \cB'$ is an equivalence.
\Cref{prop:idempotent} therefore implies that the functor $\overline{(-)}$ in~(\ref{eq:fundamental-diagram-of-fibrations}) preserves cocartesian arrows over $A\rightarrow A'$.

To prove the statement about $\dR^{\pd}$, we need to verify that the composition
\[
A'\otimes_A \dR^{\pd}_{\cB/A} \rightarrow \dR^{\pd}_{A'\otimes_A \cB/A'} \rightarrow
\dR^{\pd}_{\cB'/A'}
\]
is an equivalence  as soon as $(A,\cB)\rightarrow(A',\cB')$ in $\cat{D}'$ is cocartesian. The first map is an equivalence by \Cref{prop:dR-base-change},  the second  by \eqref{sq:bars} and  \Cref{thm:crystalline-invariance-dR}.

To see that the functors $R\Gamma$ preserve cocartesian edges, we 
note that as $A' \otimes_A -$ preserves finite limits and $X$ is coherent, the functor  $A' \otimes_A -$ preserves sheaves by  \Cref{sheafifyingonbases}. 	The functors $\ev_0$ evidently preserve cocartesian edges.
\end{proof}

Now fix a quasi-compact quasi-separated   scheme $X=(X,\cO_X)$ over $k$. By \cite[0A4G]{stacks}, the underlying topological space is coherent (cf.\  \Cref{setup:coherentspace}).

In analogy with \Cref{def:cotangent-complex-X} and~\ref{not:dR-X},  we write $\theta_{X/k}$ for $\theta_{\cO_X/k}$. Taking slice categories in (\ref{eq:fundamental-diagram-of-fibrations}), using \Cref{lemma:fundamental-diagram-cocartesian}, and applying the straightening functor we obtain a commutative diagram in $\Fun(\CAlg_\Lambda^{\pd,\art},\widehat{\cat{S}})$:
\begin{equation}\label{eq:fundamental-diagram-of-fmp}
\begin{tikzcd}
	\Def_X \arrow{r}{\ind}
	& \Def_{X}^\pd \arrow{r}{\theta} \arrow{d} 
	& \Def^\mod_{\theta_{X/k}} \arrow{d}{\ev_0} \arrow{r}{\RGamma} 
	& \Def^\mod_{\RGamma(X,\theta_{X/k})} \arrow{d}{\ev_0} \\
	& \{\ast \} \arrow{r}{\dR^\pd_{X/-}} 
	& \Def^\mod_{\dR^\pd_{X/k}} \arrow{r}{\RGamma}
	& \Def^\mod_{\RGamma(X,\dR^\pd_{X/k})}.
\end{tikzcd}
\end{equation}

When $X$ is moreover smooth and proper,   the two right corners of the above diagram 
are formal moduli problems by \Cref{prop:defmod}, since $\RGamma(X,\dR_{X/k})$ is  eventually connective   (cf.\ \cite[0FM0]{stacks}).  Under these conditions, we can define: 
\begin{definition}[Period domain and period map]\label{def:perdommap}
The \emph{period domain} of $X$
is the divided power formal moduli problem
$\Per_X \colon \CAlg^{\pd,\art}_W \rightarrow \cat{S}$ obtained as the following pullback square in $\Moduli_\Lambda^\pd$:
\[
\begin{tikzcd}[column sep=36pt]
	\Per_{X}^{}  \arrow{r} \arrow{d} \arrow[rd, phantom, "\lrcorner",  at start] 
	& \Def^{\mod}_{\RGamma(X,\theta_{X/k})} 
	\arrow{d}{\ev_0} \\
	\{\ast\} \arrow{r}{\RGamma(\dR^\pd_{X/-})} 
	& \Def^{\mod}_{\RGamma(X,\dR_{X/k})} .
\end{tikzcd}
\]
The \emph{period map} is the\vspace{-3pt} map  
\[
\per_X\colon \Def_X \rightarrow \Per_X
\]
in $\Moduli_\Lambda^\pd$ induced
by the commutative diagram (\ref{eq:fundamental-diagram-of-fmp}).
\end{definition}

\begin{remark}[Informal description]
\Cref{cor:crystalline-comparison} provides an equivalence   
\[
A\otimes_W \RGamma_\cris(X/W) \rightarrow \RGamma(X, \dR_{X/A}^\pd)
\]
for any $A$ in $\CAlg_\Lambda^{\pd,\art}$. Hence $\Per_{X}(A)$ is the space of maps \mbox{$
	A\otimes_W \RGamma_{\cris}(X/W) \rightarrow P
	$}
in $\Mod_A$ lifting the map
$
k\otimes_W \RGamma_\cris(X/W) \simeq \RGamma_{\dR}(X/k) 
\rightarrow \RGamma(X,\cO_X).$
The image of a lift $(X,\cB)$ under $\per_X$ can then be identified with the composite
\[
A\otimes_W \RGamma_\cris(X/W) \simeq \RGamma(X,\dR_{X/A}^\pd) 
\simeq \RGamma(X,\dR^\pd_{\cB/A}) \rightarrow \RGamma(X,\cB),
\]
where the second equivalence follows from \Cref{thm:crystalline-invariance-dR}. 
\end{remark}

\begin{remark}[Analogy with classical period map]
\label{rmk:classical-period-map}
The map $\per_X$ is an analogue of the (formal) classical period map, which associates to a lift $(X,\cB)$ of a smooth variety $X/\bC$ over some local Artinian $\bC$-algebra $A$ the composite
\[
A\otimes_{\bZ}\RGamma(X(\bC),\bZ) \simeq \RGamma(X,\Omega^\bullet_{\cB/A}) \rightarrow
\RGamma(X,\cB).
\]
Here the equivalence is an avatar of the Gauss--Manin connection.
\end{remark}

\begin{remark}[Higher parts of the Hodge filtration]
We have restricted ourselves to the variation of the first step of the Hodge filtration. One could similarly define period maps tracking the variation of the higher pieces of the Hodge filtration. 
\end{remark}

\begin{remark}[Non-formal deformations]
The period map can be extended to non-formal deformations over divided power rings; we will explore this 
\mbox{further in \cite{PDPeriod24} .}
\end{remark}

\subsection{Tangent to the period map}Let $X=(X,\cO_X)$ be a    smooth and proper scheme over $k$.
Given $V\in \Mod_k^\cn$ perfect, we constructed $\sqz_k^\pd V \in \CAlg^{\pd,\art}_W$ with underlying animated ring $\sqz_kV \in \CAlg^{\an,\art}_W$. In this section, we will compute 
\[
\per_X(\sqz_k^\pd V)\colon
\Def_{X}^{\an}(\sqz_k^\pd  V) \rightarrow \Per_{X}(\sqz_k^\pd V).
\] 

\begin{theorem}\label{thm:tangent-to-period-map}
The map $\per_X(\sqz_k^\pd V)$ sits in a commutative square
\[
\begin{tikzcd}
	\Def_X(\sqz_k^\pd  V) \arrow{d}{} \arrow{r}{\simeq} 
	& \Map_{\cO_X}(L_{X/k},\, V[1]\otimes_k \cO_X) \arrow{d} \\
	\Per_X(\sqz_k^\pd V) \arrow{r}{\simeq} 
	& \Map_{k}\!\big(\RGamma(X,\Fil^1\dR_{X/k}),\, 
	V\otimes_k \RGamma(X,\cO_X)\big),
\end{tikzcd}
\]
where the right map is induced by the map $\pi^1\colon\Fil^1\dR_{X/k} \rightarrow L_{X/k}[-1]$, and the horizontal maps are equivalences.
\end{theorem}

\begin{remark}The statement is entirely analogous to the classical computation of the derivative of the complex period map due to Griffiths \cite{Griffiths1968}.
\end{remark}

The rest of this section is devoted to the proof of \Cref{thm:tangent-to-period-map}.

\begin{notation}In   this section,  we write $A:=\sqz_k^\pd(V)$  for $V \in \Mod_k^{\cn}$ perfect and $\cB := A\otimes_k \cO_X$. Abusing notation,  write $A_\triv 
= \und(A)\otimes_k \sqz_k(V)  = A\otimes_k \sqz_k(V) $, interpreted as an object of either $\CAlg^\an$ or $\CAlg^\pd$. Set \mbox{$\cB_\triv := A_\triv \otimes_k \cO_X$.  }
Informally, $\cB$ is equipped with the divided power  ideal $V\otimes_k \cO_X$, while $\cB_\triv$ is equipped with the zero divided power ideal.
\end{notation}
 
The counit of the adjunction of \Cref{prop:und-as-right-adjoint} defines a map $v\colon A_\triv \rightarrow A$ in $\CAlg^\pd$ and hence a map $v\colon \cB_\triv \rightarrow \cB$ 
in $\Shv(X,\CAlg^{\pd}_k)$. 

\begin{lemma}\label{lemma:local-period-1}
There is a commutative square
\[
\begin{tikzcd}
	\Omega\Def_X(A) \arrow{d}{\ind} \arrow{r}{\simeq}
	& \Map'_{\Shv(X,\CAlg^{\pd}_k)}\big(
	\cO_X,\cB_{\triv} \big) \arrow{d}{v} \\
	\Omega\Def^{\pd}_X(A) \arrow{r}{\simeq} 
	& \Map'_{\Shv(X,\CAlg^{\pd}_k)}\big(
	\cO_X,\,\cB \big)
\end{tikzcd}
\]
in which the horizontal maps are the equivalences from \Cref{prop:defan-tangent-space} and \Cref{prop:defpd}, respectively, and the right map is induced by the map $v\colon \cB_\triv \rightarrow \cB$.
\end{lemma}

\begin{proof}This is rather formal from the definitions. 
We  have a commutative square
\[
\begin{tikzcd}
	\Map_{\Shv(X,\CAlg^{\pd}_{A_\triv})}(\cB_\triv,\cB_\triv)
	\arrow{d}{A\otimes_{A_\triv}-}
	\arrow{r}{\simeq}
	& \Map_{\Shv(X,\CAlg^\pd_k)}(\cO_X,\cB_\triv) \arrow{d}{v} \\
	\Map_{\Shv(X,\CAlg^\pd_A)}(\cB,\cB) \arrow{r}{\simeq} 
	& \Map_{\Shv(X,\CAlg^\pd_k)}(\cO_X, \cB).
\end{tikzcd}
\]
Here the horizontal equivalences are the base change  adjunctions for $k\rightarrow A_\triv$ and $k\rightarrow A$, respectively. The diagram lies over  $\Map_{\Shv(X,\CAlg^\pd_k)}(\cO_X,\cO_X)$, and taking the fibres over the identity map yields the desired commutative square.
\end{proof}	
 
\begin{lemma}\label{lemma:local-period-2}
There is a commutative diagram
\[
\begin{tikzcd}
	\Omega\Def^{\pd}_X(A) \arrow{d}{} \arrow{r}{\simeq}
	& \Map'_{\Shv(X,\CAlg^{\pd}_k)}\big(\cO_X,\cB \big)
	\arrow{d}{\RGamma\circ \theta} \\
	\Omega\Per_X(A) \arrow{r}{\simeq} 
	& \Map'_{\Fun(\cK,\Mod_k)}\big(
	\RGamma(X,\theta_{X/k}),\,
	\RGamma(X,\theta_{\cB/k}) \big),
\end{tikzcd}
\]
in which the top horizontal  map is the equivalence from \Cref{prop:defpd}, and the lower horizontal 
map is an equivalence. 
\end{lemma}

\begin{proof}\Cref{lemma:fundamental-diagram-cocartesian} gives an equivalence $\theta_{\cB/A} \simeq A\otimes_k \theta_{\cO_X/k}$. Consider
the \mbox{composite}
\[
\alpha: 	\theta_{\cB/k} \rightarrow \theta_{\cB/A} \simeq 
A\otimes_k \theta_{X/k}.
\]
in $\Fun(\cK,\Shv(X,\Mod_k))$  and write  
$\alpha_0:	\dR^{\pd}_{X/k} \rightarrow \dR^{\pd}_{X/A} \simeq A\otimes_k \dR_{X/k}$
for its  component in $\Shv(X,\Mod_k)$.  Under the equivalences
$
\Omega\Def_?(A) \simeq \Map'(?, A\otimes_k ?)
$
of Propositions~\ref{prop:defpd} and~\ref{prop:defmod-tangent-space}, the commutative square
\[
\begin{tikzcd}[column sep=36]
	\Omega\Def_X^{\pd}(A) 
	\arrow{d}{\RGamma\circ\theta} \arrow{r}
	& \{\ast\} \arrow{d}{\RGamma(\dR^{\pd}_{X/A})} \\
	\Omega\Def_{\RGamma(\theta_{X/k})}^\mod(A)
	\arrow{r}{\ev_0} 
	& \Omega\Def_{\RGamma(\dR_{X/k})}^\mod(A)
\end{tikzcd}
\]
 becomes equivalent to the outer square in the commutative diagram
\begin{equation*}
	\begin{tikzcd}[column sep=small]
		\Map'\big(\cO_X,\,\cB\big)
		\arrow{d}{\RGamma\circ \theta} \arrow{r}
		& \{*\} \arrow{d}{\simeq} \\
		\Map'\big(\RGamma(\theta_{X/k}),\,
		\RGamma(\theta_{\cB/k})\big)
		\arrow{d}{\alpha} \arrow{r}{\ev_0} \arrow[rd, phantom, "\lrcorner",  at start]
		&  \Map'\big(\RGamma(\dR_{X/k}),\,
		\RGamma(\dR_{X/k}) \big) 
		\arrow{d}{\alpha_0} \\
		\Map'\big(
		\RGamma(\theta_{X/k}), A\otimes_k \RGamma(\theta_{X/k})
		\big) \arrow{r}{\ev_0} 
		& \Map'\big(\RGamma(\dR_{X/k}),\, 
		A\otimes_k \RGamma(\dR_{X/k}) \big).
	\end{tikzcd}
\end{equation*}
Here  $\Map'$ is computed in $\Shv(X,\CAlg^{\pd}_k)$, $\Fun(\cK,\Mod_k)$, and $\Mod_k$, respectively.

Note that the top right map is an equivalence. 
We claim that the bottom square is a pullback.  To verify this, we compare the cofibres of $\alpha$ and $\alpha_0$.
As  $\dR_{\cB/k}\xrightarrow{\simeq}  \dR_{X/k}$, we find
\[
\cofib(\alpha) \simeq 
\Map\!\big(
\RGamma(\theta_{X/k}),\, 
V\otimes_k \RGamma(\dR_{X/k}) \isomfrom
V\otimes_k \RGamma(\dR_{X/k}) \rightarrow 0  
\big),
\]
with the mapping space computed in $\Fun(\cK,\Mod_k)$, and 
\[
\cofib(\alpha_0) \simeq 
\Map_{k}\!\big(
\RGamma(\dR_{X/k}),\, V\otimes_k \RGamma(\dR_{X/k})
\big).
\]
Since $0$ is final in $\Mod_k$, the induced map $\ev_0\colon \cofib(\alpha) \rightarrow \cofib(\alpha_0)$  is an equivalence, and we conclude that the lower square above is a pullback.

The claim now follows by comparing the above pullback diagram with the pullback square in \Cref{def:perdommap}. 
\end{proof}

By \Cref{prop:dR-differential}, we have a commutative square
\begin{equation}\label{eq:theta-to-L}
\begin{tikzcd}
	\dR_{X/k} \arrow{r} \arrow{d} & \cO_X \arrow{d}{d} \\
	0 \arrow{r} & L_{X/k},
\end{tikzcd}
\end{equation}
which induces a map
\[
\Psi_{X/k}\colon  \theta_{X/k} \rightarrow 
\big[ 0 \isomfrom 0 \rightarrow L_{X/k} \big]
\]
in $\Fun(\cK,\Mod(X,k))$. 
The trivial section $\cO_X \rightarrow \cB$ induces a decomposition
\begin{equation}\label{eq:theta-decomposition}
\theta_{\cB/k} \simeq \theta_{X/k} \, \oplus \,\,
\big[ 0 \isomfrom 0 \rightarrow V\otimes_k \cO_X \big]
\end{equation}
in $\Fun(\cK,\Mod(X,k))$. The following lemma is the key to the proof of \Cref{thm:tangent-to-period-map}:
\begin{lemma}\label{lemma:key-lemma}
There is a commutative square
\[
\begin{tikzcd}
	\Map'_{\Shv(X,\CAlg^\pd_k)}\!\big(\cO_X,\, \cB \big)
	\arrow{r}{\simeq} \arrow{d}{\RGamma\circ\theta} 
	& \Map\!_{\Mod(X,\cO_X)}\big(L_{X/k},\, V\otimes_k \cO_X \big)
	\arrow{d}{\Psi_{X/k}} \\
	\Map'\!
	\big( \RGamma(\theta_{X/k}),\, \RGamma(\theta_{\cB/k}) \big) 
	\arrow{r}{\simeq} 
	& \Map\!\big(
	\RGamma(\theta_{X/k}),\, 
	0 \overset{\sim}\leftarrow 0 \rightarrow V\otimes_k \RGamma(\cO_X) \big) 
\end{tikzcd}
\]
in which the top map is the equivalence from \Cref{prop:pd-cotangent-sheaf-adjunction}, and the bottom equivalence
is  induced by ~(\ref{eq:theta-decomposition}). The lower mapping spaces   are computed in $\Fun(\cK,\Mod_k)$.
\end{lemma}

\begin{proof} 
The formation of the square (\ref{eq:theta-to-L}) is functorial in the sheaf $\cO_X$. In particular, it
provides a commutative square
\begin{equation}\label{eq:key-lemma-sq1}
	\begin{tikzcd}
		\Map'_{\Shv(X, \CAlg^{\pd}_k)}\big(\cO_X,\, \cB \big) 
		\arrow{r}{L} \arrow{d}{\theta} 
		& \Map_{\Shv(X,\Mod_k)}\big(L_{X/k},\, L_{\cB/k}\big) 
		\arrow{d}{\Psi_{X/k}}  \\
		\Map'
		\big( \theta_{X/k},\, \theta_{\cB/k} \big) 
		\arrow{r}{\Psi_{\cB/k}} 
		& \Map
		\big(\theta_{X/k},\, 
		0 \isomfrom 0\rightarrow L_{\cB/k}\big)
	\end{tikzcd}
\end{equation}

By the adjunction of \Cref{prop:pd-cotangent-sheaf-adjunction}, the projection $\cB \rightarrow V\otimes_k \cO_X$ factors as
\[
\cB \overset{d}\longto L_{\cB/k} 
\overset{\delta}\longto V\otimes_k \cO_X,
\]
where we have used that $\cB\simeq \sqz^{\pd}_{\cO_X}(V\otimes_k \cO_X)$.
Composing the square (\ref{eq:key-lemma-sq1}) on the right with the map induced by $\delta$, we obtain a commutative square 
\[
\begin{tikzcd}
	\Map'_{\Shv(X, \CAlg^{\pd}_k)}\big(\cO_X,\, \cB \big) 
	\arrow{r} \arrow{d}{\theta} 
	& \Map_{\Mod(X,k)}\big(L_{X/k},\, V\otimes_k \cO_X \big) 
	\arrow{d}{\Psi_{X/k}} \\
	\Map'
	\big( \theta_{X/k},\, \theta_{\cB/k} \big) 
	\arrow{r} 
	& \Map
	\big(\theta_{X/k},\, 0 \isomfrom 0\rightarrow V\otimes_k\cO_X \big)
\end{tikzcd}
\]
By construction, the top map is the adjunction equivalence of \Cref{prop:pd-cotangent-sheaf-adjunction}. Composing with $\RGamma(X,-)$ gives the desired commutative square.
\end{proof}

\begin{proof}[Proof of \Cref{thm:tangent-to-period-map}]
Lemmas~\ref{lemma:local-period-1}, \ref{lemma:local-period-2}, and \ref{lemma:key-lemma} give the following diagram:
\[
\begin{tikzcd}
	\Omega\Def_X(A) 
	\arrow{d}{\per_X} \arrow{r}{\simeq}
	& \Map_{\Mod(X,\cO_X)}\!\big(L_{X/k},\, V\otimes_k \cO_X \big)
	\arrow{d}{\Psi_{X/k}} \\
	\Omega\Per_X(A) \arrow{r}{\simeq} 
	& \Map_{\Fun(\cK,\Mod_k)}\big(
	\RGamma(\theta_{X/k}),\, 
	0 \isomfrom 0 \rightarrow V\otimes_k \RGamma(\cO_X) \big)
	\arrow{d}{\simeq} \\
	& \Map_{k}\!\big(
	\RGamma(\Fil^1\dR_{X/k})[1],\,V\otimes_k \RGamma(\cO_X) \big).
\end{tikzcd}
\]
The bottom right vertical map is the equivalence induced by passing to cofibres in the domain. To compute the composition of the right vertical maps, note that by \Cref{prop:dR-differential}, diagram~(\ref{eq:theta-to-L}) extends to a map of cofibre sequences
\[
\begin{tikzcd}
	\dR_{X/k} \arrow{r} \arrow{d} 
	& \cO_X \arrow{d}{d} \arrow{r} 
	& \Fil^1\dR_{X/k}[1] \arrow[dashed]{d}{\pi^1} \\
	0 \arrow{r} & L_{X/k} \arrow{r} & L_{X/k}
\end{tikzcd}
\]
This gives a commutative square
\[
\begin{tikzcd}
	\Omega\Def_X(A) \arrow{d}{\per_X} \arrow{r}{\simeq} 
	& \Map_{\cO_X}(L_{X/k},\, V\otimes_k \cO_X) \arrow{d}{\pi^1} \\
	\Omega\Per_X(A) \arrow{r}{\simeq} 
	& \Map_{k}(\RGamma(\Fil^1\dR_{X/k})[1],\, 
	V\otimes_k \RGamma(\cO_X)).
\end{tikzcd}
\]
The claim follows as  $\Omega F(\sqz^\pd V[1]) \simeq F(\sqz^\pd V)$ for $F$ in $\Moduli^\pd_W$.
\end{proof}

\section{Unobstructedness}
\label{sec:unobstructedness}

In this section, we leverage the period map of Section \ref{sec:period-map}
to prove  \Cref{mainthm:btt}. 
First, we establish conditions under which    the period domain $\Per_X$   is   unobstructed. Next, we   check that $\per_X$ 
often induces an injection on `divided power obstruction classes', which lets us deduce that $\Def_X$ is `divided power unobstructed'. Finally, we use a bootstrapping technique (inspired by Schr\"{o}er ~\cite{Schroeer2003}) to establish a general unobstructedness criterion in \Cref{thm:unobstructed}, from   which we then deduce \Cref{mainthm:btt} (using a result of Achinger--Suh \cite{AchingerSuh2023} at $p=2$).

Fix a perfect\vspace{0pt} field $k$ of characteristic $p$  and write $W$ for its ring of Witt vectors.

\subsection{Unobstructedness of the period domain}

Let $M$ and $N$ be eventually connective objects in $\Mod_W$, and write $M_0 := k\otimes_W M$
and $N_0 := k\otimes_W N$.
Using \Cref{cons:modfmp}, we consider the formal moduli problem $F\colon \CAlg^{\an,\art}_W \rightarrow \cat{S}$ whose $A$-points are given by the pullback square
\[
\begin{tikzcd}[column sep=42pt]
F(A) \arrow{d} \arrow{r} \arrow[rd, phantom, "\lrcorner",  at start] & 
\Def^\mod_{M_0 \oplus N_0 \rightarrow N_0}(A) \arrow{d}{\ev_0} \\
\{\ast\} \arrow{r}{A\otimes_W (M\oplus N)} 
& \Def^\mod_{M_0 \oplus N_0}(A).
\end{tikzcd}
\]
One can think of $F$ as the formal neighbourhood of a point in a derived  Grassmannian; the following lemma provides local coordinates on this Grassmannian: 

\begin{lemma}\label{lemma:local-coordinates-grassmanian}
There is a fibre sequence
\[
F(A) \rightarrow \Map_{W}(M,A\otimes_W N) \rightarrow \Map_{W}(M,N_0),
\]
functorial in $A\in \CAlg^{\an,\art}_W$.
\end{lemma}
\begin{notation}
If $c$ is an object  in an $\infty$-category $\cat{C}$,   let $\cat{C}_{c/}^{\simeq}$ be the  \mbox{subcategory of $  \cat{C}_{c/}$} whose objects are arrows $c\rightarrow d$  and whose morphisms are equivalences \mbox{between them.}
\end{notation}
\begin{proof}[Proof of \Cref{lemma:local-coordinates-grassmanian}]
Chasing the definitions, 
we see that $F(A)$ sits in a pullback\vspace{-2pt}
\[
\begin{tikzcd}
	F(A) \arrow{r} \arrow{d} \arrow[rd, phantom, "\lrcorner",  at start] 
	& \Mod_{A, A\otimes_W(M\oplus N)/} \arrow{d}{k\otimes_A -}^\simeq \\
	\{\ast\} \arrow{r}{N_0} 
	& \Mod_{k, (M_0\oplus N_0)/}^\simeq,
\end{tikzcd}
\]
functorially in $A$. Here the bottom map picks out the projection \mbox{$\pr_1 \colon M_0\oplus N_0 \rightarrow N_0$.}

Let $\Mod_{A,A\otimes_W (M\oplus N)/}^{\simeq,\ast}$ be the full subcategory  of those $A\otimes_W (M\oplus N) \rightarrow Q$ such that the induced map $A\otimes_W N \rightarrow Q$ is an equivalence. By the conservativity of $k\otimes_A -$, the above square factors as the composition of two pullback squares:
\[
\begin{tikzcd}
	F(A) \arrow{r} \arrow{d} \arrow[rd, phantom, "\lrcorner",  at start] 
	& \Mod_{A, A\otimes_W(M\oplus N)/}^{\simeq,\ast} 
	\arrow{d}{k\otimes_A -} \arrow[hook]{r} \arrow[rd, phantom, "\lrcorner",  at start] 
	& \Mod_{A, A\otimes_W(M\oplus N)/}^{\simeq}
	\arrow{d}{k\otimes_A -} \\
	\{\ast\} \arrow{r}{N_0}
	& \Mod_{k, (M_0\oplus N_0)/}^{\simeq,\ast} \arrow[hook]{r}
	& \Mod_{k, (M_0\oplus N_0)/}^{\simeq}.
\end{tikzcd}
\]
We note that the middle vertical map is given by the map
\[
\Map_{A}(A\otimes_W M, A\otimes_W N) 
\rightarrow
\Map_{k}(M_0, N_0),
\]
which in turn identifies with the map
\[
\Map_{W}(M, A\otimes_W N) \rightarrow \Map_{W}(M, N_0).\vspace{-14pt}
\]\end{proof}

\begin{lemma}
Assume moreover that the $W$-module $\pi_0 \Map_{W}(N,M)$ is projective and that $N$ is perfect. Then $F^\cl\colon \CRing^\art_W \rightarrow \Set$ (cf.\ \Cref{not:fmp-underlying-classical}) is unobstructed. 
\end{lemma}

\begin{proof}
Since $N$ is perfect, \Cref{lemma:local-coordinates-grassmanian} gives an equivalence
\[
F(A) \simeq \Map_{W}(\Hom_W(N,M), I_A),
\]
where $I_A$ denotes the fibre of the augmentation map $A\rightarrow k$. The lemma now follows from the following general fact applied to the mapping complex \mbox{$H:= \Hom_W(N,M)$. }

Let $H$ be an object of $\Mod_W$ with $\pi_0 H$ a projective $W$-module, and assume $P\rightarrow Q$ is a surjective map of discrete $W$-modules. Then the induced map
\[
\pi_0 \Map_{W}(H,P) \rightarrow \pi_0 \Map_{W}(H,Q)
\]
is surjective. Indeed,   by the universal coefficient theorem, it sits in a commutative diagram
\[
\begin{tikzcd}[column sep=16pt]  
	0 \arrow{r} 
	& \Ext^1_W(\pi_{-1} H, P) \arrow{r} \arrow[two heads]{d} 
	& \pi_0 \Map_{W}(H,P) \arrow{r} \arrow{d} 
	& \Hom_W(\pi_0 H, P) \arrow[two heads]{d} \arrow{r}
	& 0 \\
	0 \arrow{r}
	& \Ext^1_W(\pi_{-1} H, Q) \arrow{r}
	& \pi_0 \Map_{W}(H,Q) \arrow{r}  
	& \Hom_W(\pi_0 H, Q)  \arrow{r}
	& 0 
\end{tikzcd}
\]
with exact rows. The left vertical map is surjective as $\Ext^2_W(-.-)$ vanishes since $W$ is a principal ideal domain. 
The right vertical map is surjective since $\pi_0 H$ is projective. We conclude 
that also the middle vertical map must be surjective.
\end{proof}
Using that finitely generated torsion-free modules over PIDs are free,  {we   deduce:}
\begin{corollary}\label{cor:period-domain-unobstructed}
Let $X$ be a   smooth and proper scheme over $k$. Assume that there exists a decomposition
\[
\RGamma_\cris(X/W)\simeq M \oplus N
\]
in $\Mod_W^\wedge$ such that the following composition is an equivalence:
\[
k\otimes_W N \rightarrow k\otimes_W  \RGamma_\cris(X/W) 
\isomto \RGamma_{\dR}(X/k) \rightarrow \RGamma(X,\cO_X).
\]
Assume  that $\pi_0 \Map_{W}(N,M)$ is torsion-free. Then $\Per_X$ is unobstructed.  
\end{corollary}

We record the following special case: 

\begin{corollary}\label{cor:torsion-free-period-domain-unobstructed}
Let $X$ be a   smooth and proper  scheme over $k$   such that
\begin{enumerate}
	\item $\rH^i_\cris(X/W)$ is torsion-free for all $i$;
	\item the map $\rH_{\dR}^i(X/k) \rightarrow \rH^i(X,\cO_X)$ is 
	surjective for all $i$.
\end{enumerate}
Then $\Per_X$ is unobstructed. 
\end{corollary}

\begin{proof}
As $W$ is a PID and $X$ is smooth and proper, $\rH^i_\cris(X/W)$  is a finitely generated free $W$-module and  there is an equivalence
\[
\RGamma_\cris(X/W)\simeq \bigoplus_i \rH^i_\cris(X/W)[-i]
\]
in $\Mod_W$. For all $i$,  there is a canonical isomorphism $k\otimes_W \rH^{i}_{\cris}(X/W) \rightarrow \rH^i_{\dR}(X/k)$.  
Moreover, we can pick a splitting $\rH^i_\cris(X/W) = M_i \oplus N_i$ such that the induced map $k\otimes_W N_i \rightarrow \rH^i_{\dR}(X/k) \rightarrow \rH^i(X,\cO_X)$ is an isomorphism. Taking $M=\oplus_i M_i[-i]$ and $N=\oplus_i N_i[-i]$ gives a decomposition as in \Cref{cor:period-domain-unobstructed}.  
\end{proof}

\subsection{Injectivity on obstruction classes}

\begin{proposition}\label{prop:cy-obstruction-injectivity}
Let $X$ be a   smooth and proper scheme over $k$  such that
\begin{enumerate}
	\item the map
	$
	\rH^{d-1}(X,\Fil^1\Omega_{X/k}) \rightarrow
	\rH^{d-2}(X,\Omega^1_{X/k})
	$
	is surjective;
	\item the map
	$
	\Ext^{2}_{\cO_X}(\Omega^1_{X/k},\cO_X) \rightarrow 
	\Hom_k\big(\rH^{d-2}(X,\Omega^1_{X/k}),\,
	\rH^{d}(X,\cO_X) \big)
	$
	\mbox{is injective.}
\end{enumerate}
Then for all finite (discrete) $k$-modules $V$ the map
\[
\pi_0 \Def_X(\sqz_k V[1]) \rightarrow
\pi_0 \Per_X(\sqz^\pd_k V[1])
\]
induced by $\per_X$ is injective.
\end{proposition}

\begin{proof}
Using that $\per_X$ is a map of divided power formal moduli problems, we reduce to the case   $V=k$. 
By \Cref{thm:tangent-to-period-map}, the map then identifies with the map
\[
\Hom_X\big(\Omega^1_{X/k}[-1],\,\cO_X[1]\big) \longto
\Hom_k\big(\RGamma(\Fil^1\Omega^\bullet_{X/k}),\, 
\RGamma(\cO_X)[1] \big)
\]
induced by $\Fil^1\Omega^{\bullet}_{X/k} \rightarrow \Omega^1_{X/k}[-1]$. 
This is the top map in the commutative square
\[
\begin{tikzcd}
	\Hom_{\cO_X}\big(
	\Omega^1_{X/k}[-1],\,\cO_X[1]
	\big) \arrow{r}{  } \arrow[hook]{d}{\pi_{1-d} \circ R\Gamma}
	& \Hom_k\big(
	\RGamma(\Fil^1\Omega^\bullet_{X/k}),\, 
	\RGamma(\cO_X[1])
	\big) \arrow{d}{\pi_{1-d} } \\
	\Hom_k\big(\rH^{d-2}(\Omega^1_X),\,
	\rH^{d}(\cO_X) \big) \arrow[hook]{r}
	& \Hom_k\big( \rH^{d-1}(\Fil^1\Omega^\bullet_X),\,
	\rH^{d}(\cO_X) \big),
\end{tikzcd}
\]
where the horizontal maps are induced by $\Fil^1 \Omega^\bullet_{X/k}\rightarrow \Omega^1_{X/k}[-1]$. 
As the left and bottom maps are injective by assumption, so is the top map.
\end{proof}

\subsection{Bootstrapping from divided power rings}
\label{subsec:bootstrap}
We  will now prove  a criterion which allows us to test unobstructedness of a deformation functor  on its restriction to divided power rings. This criterion is a variation on a theorem of Schr\"oer~\cite{Schroeer2003}.

\begin{definition}\label{def:divided-power-small}
A map $A\rightarrow B$ in $\PDRing^{\art}_W$ is called a \emph{divided power elementary  extension} if it is surjective, and if the kernel $I:=\ker(A\rightarrow B)$ satisfies
\begin{enumerate}
	\item $\fm_A I=0$;
	\item $\gamma_n(x)=0$ for all $x\in I$ and $n\geq 2$.
\end{enumerate}
\end{definition}

The first condition guarantees that the $A$-module $I$ is in fact a $k$-module. Note that the second condition is not implied by the first. For example, $\bZ/4\bZ \rightarrow \bF_2$ is not a divided power elementary  extension, since $\gamma_2(2)=2$. This is in some sense the only counterexample relevant to us. Indeed, we have the following result:

\begin{proposition}\label{prop:pd-bootstrap}
Assume that $k$ is infinite.
Let $F\colon \CRing_W^\art \rightarrow \Set$ be a deformation functor (see \Cref{def:schlessinger}) such that 
\begin{enumerate}
	\item $F(k[\epsilon]/\epsilon^2)$ is finite-dimensional;
	\item for all divided power elementary  extensions $A'\rightarrow A$ in $\PDRing^{\art}_W$, the induced map $F(A')\rightarrow F(A)$ is surjective;
	\item if $p=2$ then $F(W/4)$ is non-empty.
\end{enumerate}
Then $F$ is unobstructed. 
\end{proposition}

\begin{remark}
It is crucial to consider mixed characteristic divided power rings, even if one is only interested in equal characteristic deformations. For example, let $F\colon \CRing_k^\art \rightarrow \Set$
be the functor represented by $k[t]/t^p$. Then 
$F(A') \rightarrow F(A)$ is surjective for all surjective $A'\rightarrow A$ in $\PDRing^{\art}_k$, yet $F$ is clearly obstructed.
\end{remark}

The remainder of this section is devoted to the proof of \Cref{prop:pd-bootstrap}. 

\begin{notation}
Let us  denote the quotient $W/p^n$ by $W_n$, and write $W_{n,d}$ for the truncated divided power polynomial ring
\[
W_{n,d} = W_n\langle t \rangle / (\gamma_i(t) : i \geq d ).
\]
\end{notation}

Note that as a $W_n$-module, $W_{n,d}$ is free of rank $d$ on generators  $1,t,\ldots,\gamma_{d-1}(t)$. We proceed to the technical heart of the proof of \Cref{prop:pd-bootstrap}:

\begin{lemma}\label{lemma:pd-bootstrap-heart}
Assume that $k$ is infinite.
Let $F\colon \CRing^{\art}_W \rightarrow \Set$ be a deformation functor. Assume that $F(k[\epsilon]/\epsilon^2)$ is finite-dimensional, and assume that for all $n\geq 1$ and $d\geq 2$ the map $F(W_{n,d}) \rightarrow F(W_{1,2})$
is surjective. Then $F$ is unobstructed.
\end{lemma}

\begin{proof}
First assume that $F$ is pro-represented by a complete Noetherian $W$-algebra $R$. By \cite[Lemma 1.1]{Schlessinger1968}, we can write $R=W[[x_1,\ldots,x_m]]/I$, and after possibly eliminating variables, 
we may further assume   that 
\[
I \subset (x_1,\ldots, x_m)^2 + (p).
\]
We then need to show that $I=0$. 

For the sake of contradiction, suppose that there exists a non-zero element $f\in I$. Write $f=p^eg + h$
with $g$ a homogeneous   degree $d$ polynomial which is 
non-zero mod $p$ and $h$ consisting of terms of higher degree. Choose $a_1,\ldots, a_m \in k$ such that \mbox{$g(a_1,\ldots, a_m)\not\equiv 0$ mod $p$.} By our assumption on $I$, the map
\[
W[[x_1,\ldots,x_m]] \rightarrow k[t]/t^2,\, x_i \mapsto a_i t
\]
factors through a map $\varphi\colon R\rightarrow k[t]/t^2$. Now choose $n\gg 0$ such that $p^e d!$ is non-zero in $W_n$.
By our hypothesis, $\varphi$ admits a lift $\widetilde\varphi \colon R \rightarrow W_{n,d+1}$.
This lift must satisfy
\[
\widetilde\varphi(x_i) = \widetilde{a}_i t + c_{i,2}\gamma_2(t) + \cdots
\]
where $\widetilde{a}_i$ reduces to $a_i$ mod $p$. As $\gamma_{d+1}(t)=0$ in $W_{n,d+1}$, we see that 
\[
\widetilde\varphi(f) = 
p^e t^d g(\widetilde{a}_1,\ldots,\widetilde{a}_m) =
p^e d! g(\widetilde{a}_1,\ldots,\widetilde{a}_m) \gamma_d(t)
\]
in $W_{n,d+1}$. By construction, $g(\widetilde{a}_1,\ldots,\widetilde{a}_m)$ is a $p$-adic unit, and $p^e d!$ is non-zero in $W_n$, so $\widetilde\varphi(f) \neq 0$ in $W_{n,d+1}$. But this  contradicts our assumption that $f$ lies in the ideal $I$. This finishes the proof in the pro-representable case.

Now if $F$ is not pro-representable, then Schlessinger's criterion \cite{Schlessinger1968} implies that $F$ admits a hull, i.e.\ a smooth map $G\rightarrow F$ with $G$ pro-representable by some $R$ as above, and such that $G(k[\epsilon]/\epsilon^2)\rightarrow F(k[\epsilon]/\epsilon^2)$ is a bijection. For every $n\geq1$ and $d\geq 2$, we have a commutative square
\[
\begin{tikzcd}
	G(W_{n,d}) \arrow{r} \arrow[two heads]{d}
	& G(W_{1,2}) \arrow{d}{\cong} \\
	F(W_{n,d}) \arrow[two heads]{r} 
	& F(W_{1,2}).
\end{tikzcd}
\]
The bottom map is surjective by assumption. The right map is a bijection since $G$ is a hull of $F$ (note that $W_{1,2}\cong k[\epsilon]/\epsilon^2$).
An easy induction on the length  shows that $G(A) \rightarrow F(A)$ is surjective for all $A $ in $\CRing^{\art}_W$. In particular, the left map is surjective. We conclude that the top map is surjective. Hence $G$ is unobstructed, and therefore   $F$ is also  unobstructed.
\end{proof}

To finish the proof of \Cref{prop:pd-bootstrap}, we observe that the maps $W_{n,d} \rightarrow W_{1,2}$ can (almost) be factored as a composition of  divided power elementary  extensions. 

\begin{lemma}
Let $A\rightarrow B$ be a surjective map in $\PDRing^{\art}_W$ with the property that $\gamma_m(x)=0$ for all $m\geq 2$ and all
$x$ in $I:=\ker(A\rightarrow B)$. Then $A\rightarrow B$ can be factored as a composition of divided power elementary  extensions.
\end{lemma}

\begin{proof}Assume $I \neq 0$. As $I$ is an Artinian $A$-module, there is a nonzero submodule $J\subset I$ with $\fm_{A} J =0$. The map $A\rightarrow B$ then factors over a divided power elementary  extension $A\rightarrow A/J$. Repeating this process gives the desired factorisation. 
\end{proof}

\begin{corollary}\label{cor:type-I-square-zero}
Let $d\geq 0$ and $n\geq 1$. Then the map $W_{n,d+1} \rightarrow W_{n,d}$ can be factored as a composition of divided power elementary  extensions.

\end{corollary}

\begin{corollary}\label{cor:type-II-square-zero}
Let $d \geq 0$, and  $n\geq 1$ such that $p\neq 2$ or $n>1$. Then the map $W_{n+1,d} \rightarrow W_{n,d}$ is a 
composition of divided power elementary  extensions.
\end{corollary}

\begin{proof}The corresponding ideal is generated by $p^n$, so the corollary follows from the fact that $\gamma_m(p^n)=p^{mn}/m!$ is divisible by $p^{n+1}$ as soon as $p^n\neq 2$ and $m \geq 2$, which can be readily verified using Legendre's formula for the valuation $\nu_p(m!)$.
\end{proof}

\begin{proof}[Proof of \Cref{prop:pd-bootstrap}]	
Observe that we can factor the map $W_{n,d} \rightarrow W_{1,2}$ as
\[
W_{n,d} \rightarrow W_{n,{d-1}} \rightarrow \cdots \rightarrow 
W_{n,2} \rightarrow W_{n-1,2} \rightarrow \cdots \rightarrow W_{1,2}.
\]

For  $p> 2$, the assumption on $F$ together with 
Corollaries~\ref{cor:type-I-square-zero} and~\ref{cor:type-II-square-zero} imply that $F(W_{n,d})\rightarrow F(W_{1,2})$ is surjective for all $n,d$. \Cref{lemma:pd-bootstrap-heart} then implies that $F$ is unobstructed.

For $p=2$,  it remains to show that $F(W_{2,2})\rightarrow F(W_{1,2})$ is surjective. The map $W_{2,2}\rightarrow W_{1,2}$ factors as
\[
W_{2,2} \cong W_{2}[\epsilon]/\epsilon^2 \surjto W_{2}[\epsilon]/(\epsilon^2,2\epsilon) \surjto k[\epsilon]/\epsilon^2 \cong W_{1,2}.
\]
The first map is divided power elementary, hence  $F(W_2[\epsilon]/\epsilon^2) \rightarrow F(W_2[\epsilon]/(2\epsilon,\epsilon^2)$ is surjective. The second map sits in a pullback square
\[
\begin{tikzcd}
	W_2[\epsilon]/(\epsilon^2,2\epsilon) \arrow{r} \arrow{d} \arrow[rd, phantom, "\lrcorner",  at start] 
	& W_2 \arrow{d} \\
	k[\epsilon]/\epsilon^2 \arrow{r} & k
\end{tikzcd}
\]
By the assumptions, there exists a lift $y$ in $F(W_2)$. Since $F$ is a deformation functor, for every $x\in F(k[\epsilon]/\epsilon^2)$ there exists an $\tilde{x}$ in $F(W_2[\epsilon]/(\epsilon^2,2\epsilon))$ lifting both $x$ and $y$. We conclude that $F(W_2[\epsilon]/(\epsilon^2,2\epsilon)) \rightarrow F(k[\epsilon]/\epsilon^2)$ is surjective, and therefore  $F(W_{2,2}) \rightarrow F(W_{1,2})$ is  surjective as well. 
\end{proof}

\subsection{Divided power elementary extensions as homotopy pullbacks}

In \Cref{def:divided-power-small}, we introduced the notation of a \emph{divided power elementary  extension}. We now relate this to divided power trivial square-zero functor $\sqz_k^\pd$ of \Cref{def:pd-square-zero}. 

\begin{proposition}
\label{prop:artinian-pd-small-extension}
Let $A\rightarrow B$ be a divided power elementary  extension in $\PDRing^{\art}_W$ with kernel $I$. Then $A\rightarrow B$ sits in a pullback square in $\CAlg^{\pd,\art}_W$:
\[
\begin{tikzcd}
	A \arrow{r} \arrow{d} \arrow[rd, phantom, "\lrcorner",  at start] & B \arrow{d} \\
	k \arrow{r} & \sqz^\pd_k(I[1]).
\end{tikzcd}
\]
\end{proposition}

\begin{proof}
\Cref{prop:pi1-of-pd-cotangent-complex} provides a map
$
L_{B/A}^{\pd} \rightarrow \tau_{\leq 1} L_{B/A}^{\pd} \simeq I[1].
$
Using  \Cref{cor:pd-derivations-sqz} and the fact that $I$ comes from a $k$-module, we obtain 
a map 
\[
B\longto \sqz^\pd_B(I[1]) \longto \sqz^\pd_k(I[1])
\]
in $\CAlg^{\pd}_{A/}$. Define $A'$ as  the following pullback square in in $\CAlg^\pd_{A/}$:
\[
\begin{tikzcd}
	A' \arrow{r} \arrow{d} \arrow[rd, phantom, "\lrcorner",  at start] & B \arrow{d} \\
	k \arrow{r} & \sqz^\pd_k(I[1]).
\end{tikzcd}
\]
It   suffices to show that the induced map $f\colon A\rightarrow A'$ is an equivalence in $\CAlg^{\pd}$. By construction, we know that $A'$ is discrete and that 
$A/I_A \rightarrow A'/I_{A'}$ is an equivalence. 
By \Cref{prop:forgetful-functor}, it then suffices to verify that the induced map $B/A \rightarrow B/A'$ on cofibres is an equivalence. Unravelling the definitions, we see that it is equivalent to the canonical map $B/A \rightarrow L_{B/A}^{\pd} \rightarrow \tau_{\leq 1}  L_{B/A}^{\pd} $.

But \Cref{prop:pi1-of-pd-cotangent-complex} shows that this map induces an equivalence on homotopy groups, since
$I = \pi_1(B/A) \rightarrow \pi_1(L_{B/A}^{\pd} ) \cong I/I^{[2]} = I$ given by the identity.
\end{proof}

\subsection{Unobstructedness of Calabi--Yau varieties}
Our results on injectivity of the period map on tangent fibres, unobstructedness of the period domain, and bootstrapping combine to give the following general criterion for unobstructedness: 
\begin{theorem}\label{thm:unobstructed}
Let $X$ be a   smooth and proper scheme over $k$ such that:
\begin{enumerate}
	\item there is a decomposition $\RGamma_\cris(X/W) \simeq M\oplus N$ such that the induced map
	\[
	k\otimes_W N \rightarrow \RGamma_{\dR}(X/k) \rightarrow \RGamma(X,\cO_X)
	\]
	is an equivalence;
	\item the $W$-module $\pi_0\Map_{W}(N,M)$ is torsion-free;
	\item the map $\rH^{d-1}(X,\Fil^1\Omega_{X/k}) \rightarrow \rH^{d-2}(X,\Omega^1_{X/k})$ is surjective;
	\item the map
	$
	\Ext^2_{\cO_X}(\Omega^1_{X/k},\cO_X) \rightarrow 
	\Hom_k( \rH^{d-2}(X,\Omega^1_{X/k}),\, \rH^d(X,\cO_X) )
	$
	\mbox{	is injective;}
	\item If $p=2$ then $X$ admits a flat lift to $W_2 = W/4$.
\end{enumerate}
Then deformations of $X$ over local Artinian   $W$-algebras are unobstructed.
\end{theorem}

\begin{proof}
Let us first assume that $k$ is infinite.
We will  verify that $\pi_0\Def_X$ satisfies the hypotheses of \Cref{prop:pd-bootstrap}.  

Let $A'\surjto A$ be an elementary extension in $\PDRing^{\art}_W$. By \Cref{prop:artinian-pd-small-extension}, we have a pullback square
\[
\begin{tikzcd}
	A' \arrow{r} \arrow{d} \arrow[rd, phantom, "\lrcorner",  at start] 
	& A \arrow{d} \\
	k \arrow{r} & \sqz_k^\pd I[1]
\end{tikzcd}
\]
in $\CAlg^{\pd,\art}_W$. The period map $\per_X$
induces a morphism of fibre sequences 
\[
\begin{tikzcd}
	\Def_X(A') \arrow{r}\arrow{d}
	& \Def_X(A) \arrow{r}\arrow{d}
	& \Def_X(\sqz_k I[1]) \arrow{d} \\
	\Per_X(A') \arrow{r} 
	& \Per_X(A) \arrow{r} 
	& \Per_X(\sqz_k I[1]) 
\end{tikzcd}
\]
Let  $x\in \pi_0 \Def_X(A)$ and consider its image  $y\in \pi_0 \Per_X(A)$. \Cref{cor:period-domain-unobstructed} and
conditions (1) and (2) guarantee that $\Per_X$ is unobstructed, so $y$ admits a lift to $\pi_0 \Per_X(A')$. Its image 
$\ob_{\Per_X}(y, A'\rightarrow A)$ in $\pi_0 \Per_X(\sqz_k I[1])$ therefore vanishes.

Conditions (3) and (4), together with \Cref{prop:cy-obstruction-injectivity} implies that the right vertical map is injective, so that also $\ob_{\Def_X}(x,A'\rightarrow A)$ vanishes. Using \Cref{prop:unobstructed-via-obstruction-classes},
we conclude that $x$ admits a lift to $\pi_0\Def_X(A')$.

When $k$ is finite,  
write $\overline{X}$ for the base change of $X$ to the algebraic closure   $\overline{k}$. Using flat base change  \cite[02KH]{stacks},  we check that  $\overline{X}$ satisfies the conditions of the theorem; it is therefore unobstructed.
The maps $\Def_X(A) \rightarrow \Def_{\overline{X}} (W(\overline{k}) \otimes_{W(k)} A)$  assemble into a map of formal moduli problems inducing an injection on obstruction spaces. We conclude by an  argument as above that $X$ is unobstructed.
\end{proof}

\begin{remark}\label{rmk:torsion-free-condition}
Using \Cref{cor:torsion-free-period-domain-unobstructed} instead of \Cref{cor:period-domain-unobstructed}, we can replace conditions (1) and (2) by the conjunction of the conditions:
\begin{enumerate}
	\item[(1')] $\rH^i_\cris(X/W)$ is torsion-free for all $i$;
	\item[(2')] the map $\rH^i_{\dR}(X/k) \rightarrow \rH^i(X,\cO_X)$ is surjective for all $i$.
\end{enumerate} 
\end{remark}

Condition (5) in \Cref{thm:unobstructed} is always satisfied for Calabi--Yau varieties with degenerating Hodge--de Rham spectral sequence. This follows from a more general result of Achinger--Suh; we thank  Sasha Petrov for bringing it to our attention.
\begin{theorem}[Achinger--Suh]  \label{AchingerSuh}
	Let $X$ be a smooth, proper, geometrically irreducible scheme of dimension $d$ over $k$  such that
	$\omega_{X/k} \cong \cO_X$. Assume moreover that   the Hodge--de Rham spectral sequence for $X$ degenerates at the $E_1$-page.  {Then $X$ admits a flat lift  to $W_2(k)$.}
\end{theorem}
\begin{proof} 
Let $\ob(X,W_2(k)) \in H^2(X,T_X)$ be the obstruction class  to \mbox{lifting $X$  to $W_2(k)$.}
As the Hodge--de Rham spectral  sequence degenerates, the conjugate spectral  sequence also degenerates.
By \cite[Theorem 1.3]{AchingerSuh2023}, this implies that $\ob(X,W_2(k))$ lies in the kernel of the cup product map 
$  
H^2(X,T_X) \rightarrow \Map(H^{d-2}(X, \Omega^{1}_X) , H^{d}(X, \cO_X) )$, which implies the claim by Serre duality. 
\end{proof}

Combining \Cref{thm:unobstructed} and \Cref{AchingerSuh}, we obtain:
\begin{corollary}[\Cref{mainthm:btt}]
Let $X$ be a smooth and proper, geometrically irreducible scheme of dimension $d$ over $k$ such that
\begin{enumerate}
	\item $\omega_{X/k} \cong \cO_X$;
	\item $\rH^i_\cris(X/W)$ is torsion-free for all $i$;
	\item the Hodge--de Rham spectral sequence for $X$ degenerates at the $E_1$-page.
\end{enumerate}
Then deformations of $X$ over $W$ are unobstructed.
\end{corollary}

\begin{proof}
We will verify that $X$ satisfies conditions (1') and (2') from \Cref{rmk:torsion-free-condition} and (3)--(5) from \Cref{thm:unobstructed}. 

Condition (1') is satisfied by assumption, and conditions (2') and (3) follow from the degeneration of the Hodge--de Rham spectral sequence. Finally, using the assumption that the canonical bundle of $X$ is trivial, we can identify the map in (4) with
the natural map
\[
\Ext^2_X(\Omega^1_{X/k}, \omega_{X/k}) \rightarrow
\Hom_k\big( \rH^{d-2}(X,\Omega^1_{X/k}),\, 
\rH^d(X,\omega_{X/k})\big),  
\]
which is an isomorphism by Serre duality.

For $p=2$, condition (5) follows from \cite[Theorem 1.3]{AchingerSuh2023}  by Serre duality, as the conjugate spectral sequence degenerates.
\end{proof}

\section{Serre--Tate coordinates}
\label{sec:serre-tate}
The main aim of this section is to prove that the deformation space of a Bloch--Kato $2$-ordinary Calabi--Yau variety is a formal group of multiplicative type. To this end, we will use sheaves of multiplicative units  to construct a second \mbox{period map.}

\subsection{The multiplicative group} We begin  by introducing the derived multiplicative group, following \cite[25.1.5]{LurieSAG}. 
Consider the   functor
$
\bZ[-]\colon \Mod_{\bZ}^\cn \rightarrow \CAlg^{\an}
$
that  preserves sifted colimits and maps a finite free abelian group $G$ to its group ring $\bZ[G]$. This functor preserves colimits, as it preserves coproducts of finite free abelian groups. We denote its right adjoint by
\[
\bG_m\colon \CAlg^\an \rightarrow \Mod_{\bZ}^\cn.
\] 

\begin{proposition}\label{prop:Gmproperties}
	The  multiplicative group $	\bG_m$  has the following properties:$~$
\begin{enumerate}
	\item the composition of $\bG_m$ with the forgetful functor $\Omega^{\infty}:\Mod_{\bZ}^\cn \rightarrow \cat{S}$ is corepresented by $\bZ[t,t^{-1}]$;
	\item  the canonical arrow  $A^\times \xrightarrow{\simeq }\bG_m(A) $  is an equivalence for all $A\in \CRing$;
	\item the functor $\bG_m$ is left Kan extended from its restriction to the category $\CRing^{\sm}$ of (discrete)  smooth   $\bZ$-algebras. 
\end{enumerate}
\end{proposition}

\begin{proof}
The first statement follows from the equivalences
\[
\Omega^{\infty}	\bG_m(-)\simeq \Map_{\Mod_{\bZ}^\cn}(\bZ,\bG_m(-)) \simeq \Map_{\CAlg^\an}(\bZ[\bZ], -)
\]
in $\Fun(\CAlg^\an,\cat{S})$.  
The second statement is a formal consequence of the first. 
For the final statement,  first note that $\Omega^{\infty} \bG_m$ is left Kan extended from $\CRing^{\sm}$, as  $\bZ[t,t^{-1}]$ is smooth. The claim then follows from the colimit formula for left Kan extensions, using 
that $\Omega^{\infty}$ reflects sifted colimits and  that for any $B \in \CAlg^{\an}$, the $\infty$-category $\CRing^{\sm} \times_{ \CAlg^{\an}}  \CAlg^{\an}_{/B}$ is sifted, as it admits finite coproducts.
\end{proof}

For every smooth  $\bZ $-algebra $A\in \CRing^{\sm}$ and every discrete $A$-module $M$, we have a direct sum decomposition
\[
(\sqz_AM)^\times = (A\oplus M)^\times \longisomto A^\times \oplus M,\, \ 
(a,m)\mapsto (a,a^{-1}m),
\]
which is functorial in $(A,M)$. By left Kan extension, this induces decompositions
\begin{equation}\label{eq:multiplicative-square-zero}
\bG_m(\sqz_A M) \longisomto \bG_m(A) \oplus M,
\end{equation}
defining a direct sum decomposition in $\Fun(\CAlg^\an\Mod^\cn,\,\Mod^\cn_\bZ)$.

\begin{definition}[$\dlog$]
The map $\dlog\colon \bG_m \rightarrow L_{-/\bZ}$ in $\Fun(\CAlg^\an,\Mod^\cn_\bZ)$ is the left Kan extension of the following  natural transformation on smooth $\bZ$-algebras $A$:
\[
\dlog\colon A^\times \rightarrow \Omega^{1}_{A/\bZ}\colon 
a \mapsto \frac{da}{a}.
\] 
\end{definition}

Recall from \Cref{not:map-prime} that we write $\Map'(x,y)$ for the space of sections of a given map $y\rightarrow x$.

\begin{proposition}\label{prop:Gm-deformation}
Let $A$ be an animated ring, and $M$ a connective $A$-module. Then there is a commutative square 
\[
\begin{tikzcd}
	\Map'_{\CAlg^\an}(A,\sqz_A M) \arrow{r}{ \simeq} \arrow{d}{\bG_m}
	& \Map_{\Mod_A^\cn}(L_{A/\bZ}, M) \arrow{d}{\dlog} \\
\Map'_{\Mod_\bZ^\cn}(\bG_m(A),\bG_m(\sqz_A M))	 \arrow{r}{\simeq} & 
	\Map_{\Mod_\bZ^\cn}(\bG_m(A), M),
\end{tikzcd}
\]
functorial in $(A,M) \in \CAlg^\an\Mod^\cn$. Here the upper horizontal map is the equivalence from \Cref{cor:derivations-sqz}, and the lower horizontal map is the equivalence induced by the decomposition~(\ref{eq:multiplicative-square-zero}).
\end{proposition}

\begin{proof}
If $A$ is smooth and $M$ discrete, then all the mapping spaces in the statement are discrete and the resulting diagram of sets commutes by a direct inspection. The general statement follows by left Kan extension.
\end{proof}

If $X$ is a scheme,
we denote the small \'etale site of $X$  by $X_\et$ (cf.\ \Cref{def:smet}).
In this section, we will simply write $\cO_X$ for the 
 \'{e}tale structure sheaf of $X$. Let
\[
\Mod(X_\et,\bZ_p) := \lim_n \Mod(X_\et,\bZ/p^n)
\]
be the stable $\infty$-category of $p$-adic sheaves on $X_\et$. 
\begin{remark} We slightly abuse notation, as $	\Mod(X_\et,\bZ_p) $ is different from the $\infty$-category of  sheaves of modules over the constant sheaf $\bZ_p$ in the sense of \Cref{sec:sheavesofmodules}.\end{remark}

Objects of $\Mod(X_\et,\bZ_p)$ are systems $M_\bullet$ with $M_n$ in $\Mod(X_\et,\bZ/p^n\bZ)$ together with equivalences $\bZ/p^{n-1}\bZ \otimes_{\bZ/p^n\bZ}M_n \simeq M_{n-1}$. Any map of schemes $f:X\rightarrow S$ defines a pushforward functor\vspace{-1pt}
\[
\rR f_\ast \colon \Mod(X_\et,\bZ_p) \rightarrow \Mod(S_\et,\bZ_p), 
\]
which admits a left adjoint 
\[
f^{-1} \colon \Mod(S_\et,\bZ_p) \rightarrow \Mod(X_\et,\bZ_p).
\]
We also  have a limit-preserving  completion functor\vspace{-1pt}
\[
(-)^\wedge\colon \Mod(X_\et,\bZ) \rightarrow \Mod(X_\et,\bZ_p),\,
\ M \mapsto (\bZ/p^n \bZ \otimes_{\bZ} M)_n.
\]
Completion also preserves colimits, and  restricts to an equivalence between  $\Mod(X_\et,\bZ_p)$ and the full subcategory of all $p$-complete objects in $\Mod(X_\et,\bZ)$, see  \cite[2.3.1.8]{GaitsgoryLurie}.
\begin{definition}  
The \textit{$p$-completed multiplicative units functor}  $\bG^{\wedge}_{m}$  is   the \mbox{composite}\vspace{-1pt}
\[
\Shv(X_\et,\CAlg^\an) \overset{\bG_m}\longto
\Mod(X_\et,\bZ) \overset{(-)^\wedge}{\longto}
\Mod(X_\et,\bZ_p).
\]\end{definition} 

We will need the following observation about the   sheaf $\cW$ from 	\Cref{not:cW}:
\begin{proposition} \label{prop:inverseandgm}
Given any $A \in \CAlg^{\an,\art}_W$, there are canonical equivalences 
{$\bG_m(f^{-1}\cW \otimes_WA) \simeq 
	f^{-1}\bG_m(\cW \otimes_WA) $}  and 
{$\bG^{\wedge}_m(f^{-1}\cW \otimes_WA) \simeq 
		f^{-1}\bG^{\wedge}_m(\cW \otimes_WA) $}. 
\end{proposition}
\begin{proof}
As  completion functor $ \Mod(X_\et,\bZ) \rightarrow \Mod(X_\et,\bZ_p)$
and the sheafification functor $\Shv(X_\et, \Mod_{\bZ}^{\cn}) \rightarrow \Shv(X_\et, \Mod_{\bZ}) \simeq  \Mod(X_\et,\bZ)  $
 are compatible with inverse image functors, it suffices to check that the following square commutes: 
\[
\begin{tikzcd}
	\Shv(S_\et, \CAlg^{\an}) \arrow{r} \arrow{d}{f^{-1}}
	&  \Shv(S_\et, \Mod^{\cn})  \arrow{d}{f^{-1}} \\
	\Shv(X_\et, \CAlg^{\an}) \arrow{r} & 	  \Shv(X_\et, \Mod^{\cn}) 
\end{tikzcd}
\]  
This follows from \cite[Proposition 4.1.7]{LurieDAG} as the functor $\bG_m$ is represented by  the compact object $\bZ[t,t^{-1}]$ by \Cref{prop:Gmproperties}.
\end{proof}	

\begin{remark}
Alternatively, one could work in the pro-\'etale topology, and define $\bG^{\wedge}_{m}(\cA)$ as the $p$-adic completion of $\bG_m(\cA)$ in $\Shv(X_{\textrm{pro\'et}},\bZ)$.
\end{remark}

\subsection{Bloch--Kato $n$-ordinary varieties}
Let $k$ be a perfect field of characteristic $p$, and $X$ a smooth and proper scheme over $k$. Write  $B^r_{X/k}$ for the image of $d\colon \Omega^{r-1}_{X/k}\rightarrow \Omega^r_{X/k}$. Following \cite{BlochKato86}, one says that $X$ is \emph{Bloch--Kato ordinary} if $\rH^i(X,B^r_{X/k})$ vanishes for all $i$ and $r$. Inspired by Achinger and Zdanowicz~\cite{AchingerZdanowicz2021}, we introduce a weaker ordinariness condition: 

\begin{definition}We call $X$ \emph{Bloch--Kato $n$-ordinary} if $\rH^\ast(X,B^r_{X/k})\cong 0$ \mbox{for all $r\leq n$.}
\end{definition}

	\begin{proposition}Assume that the first $n$ slopes of the Newton and the Hodge polygon of $X$ agree. 
Then $X$ is Bloch--Kato $n$-ordinary. If $\rH^\ast_\cris(X/W)$ is torsion-free, the  \mbox{converse   holds true.}
\end{proposition}

\begin{proof}This follows from \cite[Lemma~7.1(iv)]{BlochKato86}.
\end{proof}

\begin{remark}
In the terminology of Achinger and Zdanowicz \cite[Section 2.1]{AchingerZdanowicz2021}, this becomes the following statement: If $\rH^\ast_\cris(X/W)$ is torsion-free, then $X$ is Bloch--Kato $n$-ordinary if and only if the $F$-crystal $\rH^i_\cris(X/W)$ is $n$-ordinary\vspace{-1pt} for all $i$. 
\end{remark}

Bloch--Kato $1$-ordinary and $2$-ordinary varieties will play an important role in \Cref{mainthm:Serre-Tate}. Writing $S=\Spec (k)$, and
$f\colon X\rightarrow S$ for the structure map, we have:

\begin{proposition}\label{prop:1-ordinary}\label{prop:2-ordinary}
The smooth and proper  scheme  $X$ over $k$ is Bloch--Kato $1$-ordinary if and only if the following natural map is an equivalence in $\Mod(S_\et,k)$:
\[
\cO_S\otimes_{\bF_p }\rR f_\ast \bF_p \rightarrow \rR f_\ast \cO_X.
\]
It is  Bloch--Kato $2$-ordinary if and only if it is Bloch--Kato $1$-ordinary, and the map
\[
\cO_S\otimes_{\bZ }  
\rR f_\ast \bG_{m,X} \rightarrow \rR f_\ast \Omega^1_{X/k}
\]
induced by $\dlog$ is an equivalence.
Here $\cO_{S}$ is the \'{e}tale structure \mbox{sheaf of $S=\Spec(k)$.}
\end{proposition}

\begin{proof}
It suffices to check that both maps induce  equivalences on   $\Spec(\overline{k}) \rightarrow \Spec( {k})$. Here, the claim 
follows from  \cite[Lemma~7.3(3)]{BlochKato86}. We have used that all sheaves in sight are coconnective and hence hypercomplete, so we can work in the derived $\infty$-category of the ordinary category of discrete sheaves \mbox{of $k$-modules on $S_\et$.}
\end{proof}
\begin{remark} 
For $k=\overline{k}$, the sheaf $\cO_S$ in  \Cref{prop:1-ordinary} is  \mbox{the constant sheaf $\overline{k}$.}
	\end{remark}

\subsection{Construction of Serre--Tate coordinates}\label{sec:stcoord}

Let $f:X\rightarrow \Spec(k)$ be a smooth, proper, and Bloch--Kato $1$-ordinary scheme over a perfect field $k$ of characteristic $p$. In this section, we will construct a map from the formal moduli problem $\Def_X$ to an auxiliary  (derived) formal group $\ST'_X$. In the subsequent \Cref{subsec:cancoord}, we will then construct the  Serre-Tate period domain $\ST_X$ from 
$\ST'_X$.
The key ingredient for our construction is the fact  for any deformation
$\mathcal{B} $ of $\cO_X$ over \mbox{$A \in \CAlg_W^{\an,\art}$}, 
the cohomology of the $p$-completed sheaf of principal units $\fib(\bG_m^{\wedge}(\mathcal{B}) \rightarrow \bG_m^{\wedge}(\cO_X))$
 {only depends on $A$}, see \Cref{rem:stinformal}.

To algebraise our formal lifts, we will also want to lift line bundles. To this end, we will construct a formal moduli problem $\Def_{X,\lambda}$ for any class $\lambda$ in $ \Pic(X)$, encoding deformations of the pair $(X,\lambda)$. We will then construct a map from $\Def_{X,\lambda}$ to an  
auxiliary  (derived) \mbox{formal group  $\ST'_{X,\lambda}$} (and its variant $\ST_{X,\lambda}$ in \Cref{subsec:cancoord}).

But first, let us construct the formal moduli problem $\widetilde{\Def}_{X,\lambda}$.

\begin{construction}[$\widetilde{\Def}_{X,\lambda}$]
Fix  a map  $\lambda\colon \bZ[-1] \rightarrow \bG_{m,X}$ in $\Shv(X_\et,\Mod_\bZ)$,
where  $ \bG_{m,X} := \bG_m(\cO_X)$.
Recall from \Cref{constr:def-an-tilde} the $\infty$-category\vspace{-3pt}
\[
\widetilde{\cat{C}}_{\et} := \CAlg^{\an}_{W} \times_{  \Fun(\Delta^1, \Shv(X_{\et},\CAlg^\an))} 
\Fun(\Delta^2,\Shv(X_{\et},\CAlg^\an))
\]
whose objects are pairs consisting of
\begin{enumerate}
	\item an object $A$ in $\CAlg^\an$;
	\item a  morphism $ A \rightarrow f^{-1}\cW\otimes_W {A}  \rightarrow \cB$ in $\Shv(X_\et,\CAlg^\an)$
\end{enumerate}

Using the functor $
\bG_m \colon \Shv(X_\et,\CAlg^\an) \rightarrow \Shv(X_\et,\Mod_\bZ)$ above,  we define
\[
\widetilde{\cat{C}}_{\et}' := \widetilde{\cat{C}}_{\et} \times_{\Shv(X_\et,\Mod_\bZ)} 
\Fun(\Delta^1,\Shv(X_\et,\Mod_\bZ)) \times_{\Shv(X_\et,\Mod_{\bZ})} 
\{\bZ[-1]\}
\]
as the  $\infty$-category  of triples consisting of $A$ and $ (A \rightarrow f^{-1}\cW\otimes_W {A}  \rightarrow \cB)$  as in (1)--(2) above, \mbox{together with}
\begin{enumerate}
	\item[(3)] a map $\bZ[-1] \rightarrow \bG_m(\cB)$ in $\Mod(X_\et,\bZ)$.
\end{enumerate} 
We have projections
$
\widetilde{\cat{C}}_{\et}' \overset{q'}\rightarrow \widetilde{\cat{C}}_{\et} \overset{q}\rightarrow \CAlg^\an,
$
and we readily check that every arrow in $\widetilde{\cat{C}}_{\et}' $ is $q'$-cocartesian.
The triple $(k, \cO_X, \lambda)$ fixed above determines an object of $\widetilde{\cat{C}}_{\et}'$, and we obtain left fibrations \vspace{-3pt}
\[
(\widetilde{\cat{C}}_{\et}')^\cocart_{/(k,\cO_X,\lambda)} \overset{q'}\longto 
(\widetilde{\cat{C}}_{\et} )^\cocart_{/(k,\cO_X)} \overset{q}\longto \CAlg^\an_{/k}
\]
by \cite[Corollary 2.4.2.5]{LurieHTT}; the notation $(-)^\cocart$ is used as in \Cref{constr:def-an}.
These  straighten to a morphism $\widetilde{\Def}_{X,\lambda} \rightarrow  \widetilde{\Def}_X \simeq {\Def}_X$
in $\Fun(\CAlg^{\an,\art}_W,\widehat{\cat{S}})$, where we have used the equivalence \vspace{-4pt} in \Cref{prop:equivetale}.
\end{construction}

\begin{remark}[Informal description]
Objects of $\widetilde{\Def}_{X,\lambda}(A)$ are
pairs consisting of a  composition
$\left(A \rightarrow f^{-1}\cW \otimes_WA \rightarrow \cB\right)$ lifting $\left(k \rightarrow f^{-1}\cW \otimes_Wk \rightarrow \cO_X)\right)$
   and a commutative triangle in $\Mod(X_\et,\bZ)$: 
\[
\begin{tikzcd}	
	& \bG_m(\cB) \arrow{d} \\
	\bZ[-1] \arrow[dashed]{ru}{\tilde\lambda} \arrow{r}{\lambda}
	& \bG_m(\cO_X).
\end{tikzcd}\vspace{-3pt}
\]

\end{remark}

\begin{lemma} \label{lemma:DefXlambda}
The functor $\widetilde{\Def}_{X,\lambda}$ is a formal moduli problem.
\end{lemma}

\begin{proof} 
This follows from \Cref{prop:defan}, \Cref{prop:equivetale}, 
and the fact that for any $A$ in $\CAlg^{\an, \art}_W$ and  $\cB$ in $\Shv(X_\et,\CAlg^\an_A)$, the functor
\[
\CAlg^{\an, \art}_A \rightarrow \Mod(X_\et,\bZ),\, A' \mapsto \bG_m(A'\otimes_A \cB)
\]
preserves pullbacks of pairs $A_0 \rightarrow A_{01}, A_1 \rightarrow A_{01}$ inducing surjections on $\pi_0$. Indeed,  let us fix such a pair; write $(\ast)$ for the  associated pullback square in 
$\CAlg^{\an, \art}_A $.
For any \'etale open $U \rightarrow X$,
$\pi_0$-surjectivity implies that $(\ast) \otimes_A \cB(U)$ is a pullback in $\CAlg^{\an}_{A}$ whose leg $A_0\otimes_A \cB(U) \rightarrow A_{01}\otimes_A \cB(U) $ (and also $A_1 \otimes_A \cB(U)\rightarrow A_{01}\otimes_A \cB(U)$) induces a  surjection with nilpotent kernels on $\pi_0$.

A  truncated geometric series shows  that \mbox{$\bG_m(A_0\otimes_A \cB(U)) \rightarrow\bG_m(A_{01}\otimes_A \cB(U))$} is  also surjective on $\pi_0$. This implies that $\bG_m((\ast) \otimes \cB(U))$ is a pullback, i.e.\ a pushout, in $\Mod_{\bZ}$-valued presheaves.
Its sheafification $\bG_m((\ast) \otimes \cB)$ is a pushout in $\Mod_{\bZ}$-sheaves and so a  pullback by stability.  
\end{proof}

\begin{construction}[$\ST'_{X,\lambda}$ and $\ST'_X$] \label{cons:aux}
Using the composite functor \vspace{-2pt}
\[
\CAlg^\an \xrightarrow{ \ A \mapsto f^{-1}\cW \otimes_WA\ }  \Shv(X_\et,\CAlg^\an)
\xrightarrow{\bG^{\wedge}_{m}}  \Mod(X_\et,\bZ_p)
\xrightarrow{\rR f_\ast \ } \Shv(S_\et,\bZ_p),
\]
we define the $\infty$-category 
\[
\cat{D} := \CAlg^\an \times_{\Mod(S_\et,\bZ_p)}
\Fun(\Delta^1,\Mod(S_\et,\bZ_p))
\]
whose objects are pairs consisting of 
\begin{enumerate}
	\item an object $A$ in $\CAlg^\an$;
	\item a map $\rR f_\ast \bG^{\wedge}_{m}(f^{-1} \cW \otimes_W A)\rightarrow N$ in 
	$\Mod(S_\et,\bZ_p)$.
\end{enumerate}
We also define the $\infty$-category \vspace{-1pt}
\[
\cat{D}' := \cat{D}\times_{\Mod(S_\et,\bZ_p)} 
\Fun(\Delta^1,\Mod(S_\et,\bZ_p)) 
\times_{\Mod(S_\et,\bZ_p)} 
\{ f_\ast \bZ_p[-1] \}
\]
whose objects are pairs as in (1)--(2) above, together with
\begin{enumerate} 
	\item[(3)] a map $\rR f_\ast \bZ_p[-1] \rightarrow N$ in $\Mod(S_\et,\bZ_p)$.
\end{enumerate}
Consider the projection maps $
\cat{D}' \overset{s'}\longto
\cat{D} \overset{s}\longto \CAlg^\an.$
A map in $\cat{D}$ is $s$-cocartesian if and only if the underlying square  
\[
\begin{tikzcd}
	\rR f_\ast\bG^{\wedge}_{m}(f^{-1}\cW \otimes_WA) \arrow{r} \arrow{d}
	& N \arrow{d} \\
	\rR f_\ast\bG^{\wedge}_{m}(f^{-1}\cW \otimes_WA') \arrow{r} & N'\arrow[lu, phantom, "\ulcorner", at start]
\end{tikzcd}
\]
is a pushout in $\Mod(S_\et,\bZ_p)$. Every map in $\cat{D}'$ is $s'$-cocartesian.  

Let $\cat{D}^\cocart$ be  the subcategory of $ \cat{D}$  spanned by all objects and only   $s$-cocartesian morphisms. Write $\cat{D}'^\cocart  $ for the subcategory of $\cat{D}'$ spanned by  $(s\circ s')$-cocartesian morphisms. Set $\bG^{\wedge}_{m,X}  := \bG_{m}^{\wedge}(\cO_X)$. 
As in \Cref{cons:modfmp}, we   have left fibrations
\[
\cat{D}'^\cocart_{/(k,0\rightarrow \rR f_\ast  \bG^{\wedge}_{m,X}  ,\widehat\lambda)}
\longto \cat{D}^\cocart_{/(k,0\rightarrow  \rR f_\ast  \bG^{\wedge}_{m,X}  )}
\longto \CAlg^{\an}_{/k};
\]
we have used \Cref{prop:inverseandgm} and that $k$ is perfect to see $\rR f_\ast\bG^{\wedge}_{m}(f^{-1}\cW \otimes_Wk) \simeq 0$.

Straightening and restricting gives a morphism  $ \ST_{X,\lambda}' \rightarrow \ST_{X}'$ in  $\Fun(\CAlg^{\an,\art}_W,\widehat{\cat{S}})$.
\end{construction}

\begin{remark}[Informal description]
Objects in $\ST'_{X}(A)$ are fibre sequences
\[
\rR f_\ast \bG^{\wedge}_{m}(f^{-1}\cW \otimes_W A)  \longto N \longto 
\rR f_\ast  \bG^{\wedge}_{m,X}  .
\]
in $\Mod(S_\et,\bZ_p)$. Objects in $\ST'_{X,\lambda}(A)$ are such fibre sequences  with a lift $\tilde{\lambda}$ of $\lambda$:
\[
\begin{tikzcd}
	& & \rR f_\ast \bZ_p[-1] \arrow{d}{\lambda} 
	\arrow[dashed,swap]{dl}{\tilde\lambda} \\
	\rR f_\ast\bG^{\wedge}_{m}(f^{-1}\cW \otimes_W A) \arrow{r} 
	& N \arrow{r}
	& \rR f_\ast  \bG^{\wedge}_{m,X}  .
\end{tikzcd}
\]
\end{remark}

We can give a more direct description of $\ST'_X$ and $\ST'_{X,\lambda}$. To this end, let us define  $P_{X,\lambda}$ in $\Mod(X_\et,\bZ_p)$ by the cofibre sequence:
\begin{equation}\label{eq:def-Q}
\bZ_p[-1] \overset{\lambda}\longto 
\bG^{\wedge}_{m,X}   \longto
P_{X,\lambda}.
\end{equation} 

\begin{lemma}\label{lemma:ST-explicit}
There is a commutative square
\[
\begin{tikzcd}
	\ST'_{X,\lambda}(-) \arrow{r}{\simeq} \arrow{d}
	& \Map_{\Mod(S_\et,\bZ_p)}\big(\rR f_\ast P_{X,\lambda},\,
	\rR f_\ast\bZ_p\otimes_\bZ \bG_m( \cat{W} \otimes_W-)[1]\big)
	\arrow{d} \\
	\ST'_X(-) \arrow{r}{\simeq}
	& \Map_{\Mod(S_\et,\bZ_p)}\big(\rR f_\ast  \bG^{\wedge}_{m,X}  ,\,
	\rR f_\ast \bZ_p  \otimes_\bZ \bG_m( \cat{W} \otimes_W-)[1] \big)
\end{tikzcd}
\]
in $\Fun(\CAlg^{\an,\art}_W,\widehat{\cat{S}})$, in which  the horizontal maps are equivalences.
\end{lemma}

\begin{proof}This follows from the definitions, together with the fact that the map
\[
(\rR f_\ast \bZ_p) \otimes_{\bZ} \bG_m(\cat{W} \otimes_W A) \longrightarrow
\rR f_\ast \bG^{\wedge}_{m}(f^{-1} \cat{W} \otimes_W A)
\]
is an equivalence in $\Mod(S_\et,\bZ_p)$ by \Cref{prop:inverseandgm} and the projection formula in \cite[0F0F]{stacks}. We have used that $\bG_{m}(A)$ is perfect as $A \in \CAlg^{\an,\art}_W$ is Artinian. 
\end{proof}

\begin{corollary}$\ST'_{X}$ and $\ST'_{X,\lambda}$ are formal moduli problems. 
\end{corollary}
\begin{proof}
As in the proof of  \Cref{lemma:DefXlambda}, the functor $\bG_m(\cat{W}(U) \otimes_W-) $ sends 
	$\pi_0$-surjective maps $A \rightarrow A'$ in $\CAlg^{\an,\art}_W$  to $\pi_0$-surjections in $\Mod_{\bZ}$ for each $U \in S_{\et}$. The claim follows as pullbacks in sheaves are computed in presheaves.
	\end{proof}

We now come to the construction of the maps $\st'_{X,\lambda}$ and $\st'_X$. The key input is:

\begin{proposition}\label{prop:serre-tate-pullback}
Let  $f\colon X\rightarrow S=\Spec (k)$ be smooth, proper, and Bloch--Kato $1$-ordinary. Fix
a morphism $A'\rightarrow A$  in $\CAlg^{\an,\art}_W$. Given  a \mbox{pushout square}	\begin{equation}\label{square:serre-tate-pullback}
	\begin{tikzcd}
	f^{-1} \cat{W} \otimes_W	A' \arrow{d} \arrow{r} &f^{-1} \cat{W} \otimes_W A  \arrow{d} \\
		\cB' \arrow{r} & \cB \arrow[lu, phantom, "\ulcorner", at start]
	\end{tikzcd}
\end{equation}
in $\Shv(X_\et,\CAlg^\an)$, the induced square 	in $\Mod(S_\et,\bZ_p)$ is a pullback:  
\[
\begin{tikzcd}
	\rR f_{\ast }\bG^{\wedge}_{m}(f^{-1} \cat{W} \otimes_WA') \arrow{r} \arrow{d} \arrow[rd, phantom, "\lrcorner",  at start] 
	& 	\rR f_{\ast }\bG^{\wedge}_{m}(f^{-1} \cat{W} \otimes_WA) \arrow{d} \\
	\rR f_\ast \bG^{\wedge}_{m}(\cB') \arrow{r}
	& 	\rR f_\ast \bG^{\wedge}_{m}(\cB).
\end{tikzcd}
\]

\end{proposition}
\begin{proof}
Call a square in $\Shv(X_\et,\CAlg^\an)$  \textit{good} if 
applying $ \bG^{\wedge}_{m}$ to it  {gives a pullback.} 

First, assume that $A'\rightarrow A$ arises as in \Cref{cor:artinian-small} as a pullback of a cospan $A \rightarrow \sqz_k(V[n+1]) \leftarrow k$. 
Then
$A'\rightarrow A$,  $\cB'\rightarrow \cB$  sit in compatible \mbox{pullback squares}
\[\hspace{-15pt}
\begin{tikzcd}
	f^{-1} \cat{W} \otimes_W	A' \arrow{r} \arrow{d} \arrow[rd, phantom, "\lrcorner",  at start]  & 	f^{-1} \cat{W} \otimes_WA\arrow{d} 
	& &
	\cB' \arrow{r} \arrow{d} \arrow[rd, phantom, "\lrcorner",  at start]  & \cB \arrow{d} \\
	f^{-1} \cat{W} \otimes_W	k \arrow{r} &	f^{-1} \cat{W} \otimes_W \sqz_k(V[n+1]) 
	& &
	\cO_X \arrow{r} & \sqz_k(V[n+1]) \otimes_k\cO_X .
\end{tikzcd}
\]
These squares are  both good by the same argument as in \Cref{lemma:DefXlambda}, as  the functor $\rR f_\ast \bG_{m}(-):
\CAlg^{\an} \rightarrow \Mod^{\cn}$ preserves limits and the lower horizontal maps are  \mbox{$\pi_0$-surjective.} We now observe that the square
\[
\begin{tikzcd}
	\rR f_\ast \bG^{\wedge}_{m} (	f^{-1} \cat{W} \otimes_W k) \arrow{r} \arrow{d} & 
	\rR f_\ast \bG^{\wedge}_{m} ( 	f^{-1} \cat{W} \otimes_W \sqz_k(V[n+1]) ) \arrow{d} \\
	\rR f_\ast \bG^{\wedge}_{m}(\cO_X) \arrow{r} 
	& \rR f_\ast \bG^{\wedge}_{m}(\sqz_{k}(V[n+1])\otimes_k \cO_X)
\end{tikzcd}
\]
is a pullback square in $\Mod(S_\et, \bZ_p)$. Indeed,    the decomposition \eqref{eq:multiplicative-square-zero} above shows 
that the  induced map between the cofibres of the horizontal maps is given by
\[
V\otimes_k  \cO_S \otimes_{k} \rR f_\ast k \rightarrow V\otimes_k \rR f_\ast \cO_X,
\]
which is an equivalence by $1$-ordinariness, see \Cref{prop:1-ordinary}.
The `pasting law for pullbacks' shows that the proposition holds true for all $A'\rightarrow A$  as above.

Inducting on the length via \Cref{cor:artinian-small} proves proposition for all $\pi_0$-surjections.

We can write a general $f:A' \rightarrow A$ as a composition $ A' \rightarrow A' \otimes_W A \rightarrow A$, where the first map is the inclusion and the second map is  given by $(f,\id_A)$. By our previous considerations, it then suffices to prove that the   square 
\[	\begin{tikzcd}
		f^{-1} \cat{W} \otimes_W A' \arrow{d} \arrow{r} & 	f^{-1} \cat{W} \otimes_W (A' \otimes_W A)  \arrow{d} \\
	\cB' \arrow{r} & \cB'  \otimes_W A \arrow[lu, phantom, "\ulcorner", at start].
\end{tikzcd}\]
is good. Let us write $A$ as a pullback of a cospan   $A_0 \rightarrow \sqz_k(V[n+1]) \leftarrow k$ as in \Cref{cor:artinian-small}, where $A_0$ has shorter length.
We now consider the diagram 
\[	\begin{tikzcd}
	f^{-1} \cat{W} \otimes_W	A' \arrow{d} \arrow{r} & 	f^{-1} \cat{W} \otimes_W(A' \otimes_W A)   \arrow{d}\arrow{r} &
	f^{-1} \cat{W} \otimes_W (A' \otimes_W A_0) \arrow{d}  \\
	\cB'   \arrow{r} & \cB' \otimes_W A   \arrow{r} &
	\cB'\otimes_W A_0 .
\end{tikzcd}\]
The right square is  good by $\pi_0$-surjectivity of 
$A' \otimes_W A \rightarrow A' \otimes_W A_0$.
As the  outer  square is good by induction on the length, the left square is good.
\end{proof}

\begin{construction}[The maps   $\st_X'$ and $\st_{X,\lambda}'$]
\label{constr:st} 
Let $X\rightarrow S$ be smooth, proper and Bloch--Kato $1$-ordinary, and let $\lambda \colon \bZ[-1] \rightarrow \bG_{m,X}$ be a map in $\Mod(X_\et,\bZ)$.

By construction, we have the following  commutative square over $\CAlg^\an$:
\[
\begin{tikzcd}[column sep=36pt]
	\widetilde{\cat{C}}_{\et}' \arrow{r}{\rR f_\ast \bG^{\wedge}_{m}} \arrow{d}{q'}
	& \cat{D}' \arrow{d}{s'} \\
	\widetilde{\cat{C}}_{\et}\arrow{r}{\rR f_\ast \bG^{\wedge}_{m}} 
	& \cat{D}.
\end{tikzcd}
\]
The lower horizontal map first 
sends $(A \rightarrow f^{-1} \cW\otimes_W A \rightarrow \cB)$ to \mbox{$(f^{-1} \cW\otimes_W A \rightarrow \cB)$} and then applies   $\rR f_\ast \bG^{\wedge}_{m}$.

Restricting to fibres over  
$\CAlg^{\an,\art}_W$ and only  the cocartesian \mbox{morphisms}, we obtain a diagram
\[
\begin{tikzcd}[column sep=36pt]
	\widetilde{\cat{C}}_{\et}'^{\cocart, \art}
	\arrow{r}{\rR f_\ast \bG^{\wedge}_{m}} \arrow{d}{q'}
	& \cat{D}'^{\cocart, \art} \arrow{d}{s'} \\
	\widetilde{\cat{C}}_{\et}^{\cocart, \art}\arrow{r}{\rR f_\ast \bG^{\wedge}_{m}} 
	& \cat{D}^{\cocart, \art},
\end{tikzcd}
\] where we have used that by 
 \Cref{prop:serre-tate-pullback}, all functors in  the above square   preserve cocartesian edges over $\CAlg^{\an,\art}_W$.
Taking slice categories,   straightening, and using \Cref{prop:equivetale}, we obtain a commutative square of formal moduli problems in $\Moduli^\an_{W}$:
\begin{equation}\label{eq:Serre-Tate-square}
	\begin{tikzcd}[column sep=36pt]
		\widetilde{\Def}_{X,\lambda} 
		\arrow{d} \arrow{r}{\st'_{X,\lambda}}
		& \ST'_{X,\lambda} \arrow{d} \\
		\Def_X \simeq 	\widetilde{\Def}_{X} \arrow{r}{\st'_X} & \ST'_X.
	\end{tikzcd}
\end{equation}
\end{construction}

\begin{remark}[Informal description] \label{rem:stinformal}
The map $\st'_{X,\lambda}(A)$ takes a\vspace{-2pt}\vspace{-2pt}   pair 
\[
\big(\cB,\,
\bZ[-1] \overset{\tilde{\lambda}}\rightarrow \bG_m(\cB)
\big)
\]
in $\widetilde{\Def}_{X,\lambda}(A)$ to the resulting diagram in $\Mod(S_\et,\bZ_p)$: \vspace{-5pt}
\[
\begin{tikzcd}
	& & \rR f_\ast \bZ_p[-1] \arrow{d}{\lambda} 
	\arrow[swap]{dl}{\tilde\lambda} \\
	\rR f_\ast\bG^{\wedge}_{m}(f^{-1}\cW \otimes_W A) \arrow{r} 
	& \rR f_\ast \bG^{\wedge}_{m}(\cB) \arrow{r}
	& \rR f_\ast  \bG^{\wedge}_{m,X}  
\end{tikzcd}
\]
The horizontal sequence is a fibre sequence by \Cref{prop:serre-tate-pullback} as $\bG^{\wedge}_{m,X}  (k) \simeq 0$. 
\end{remark}

\subsection{Tangent to the Serre--Tate map} Let $X$ be as in \Cref{sec:stcoord}.
We will now compute the effect on tangent fibres of the maps 
$\st'_X$ and $\st'_{X,\lambda}$  in \Cref{constr:st}. 
 
\begin{notation}
	Fix $V \in \Mod_k^{\cn }$ perfect and 
write $A:=\sqz_k V$ for  and \mbox{$\cB := A\otimes_k \cO_X$.}
\end{notation}

The base point in the space $\ST'_X(A)$ corresponds to the fibre sequence 
\begin{equation}\label{eq:split-extension-ST}
\rR f_\ast \bG^{\wedge}_{m}(f^{-1}\cW \otimes_WA)  
\longto \rR f_\ast \bG^{\wedge}_{m}(\cB)
\longto \rR f_\ast  \bG^{\wedge}_{m,X}.
\end{equation}
The canonical section of $A\rightarrow k$ induces a section
$\rR f_\ast  \bG^{\wedge}_{m,X}   \rightarrow \rR f_\ast \bG^{\wedge}_{m}(\cB)$ of the right map in \eqref{eq:split-extension-ST}, 
which in turn induces an equivalence \vspace{-1pt}
\[
\Omega\ST'_X(A) \isomto 
\Map'_{\Mod(S_\et,\bZ_p)}\big(\rR f_\ast  \bG^{\wedge}_{m,X}  ,\,
\rR f_\ast \bG^{\wedge}_{m}(\cB) \big),
\]
Informally, it sends an automorphism $\alpha$ of the
extension (\ref{eq:split-extension-ST}) to the composite
\[
\rR f_\ast  \bG^{\wedge}_{m,X}   \rightarrow 
\rR f_\ast \bG^{\wedge}_{m}(\cB) \overset{\alpha}\rightarrow
\rR f_\ast \bG^{\wedge}_{m}(\cB)
\]
in $\Mod(S_\et,\bZ_p)$. Unravelling the construction of  $\st'_X$ (cf.\  \Cref{constr:st}),  \mbox{we find:}

\begin{lemma}\label{lemma:serre-tate-tangent}
There is a commutative square
\[
\begin{tikzcd}
	\Omega\Def_X(A) \arrow{r}{\simeq} \arrow{d}{\st'_X(A)} 
	& \Map'_{\Shv(X_\et,\CAlg^\an_k)}\big(\cO_X,\, \cB\big)
	\arrow{d}{\rR f_\ast \bG^{\wedge}_{m}} \\
	\Omega\ST'_X(A) \arrow{r}{\simeq} 
	& \Map'_{\Mod(S_\et,\bZ_p)}\big(
	\rR f_\ast  \bG^{\wedge}_{m,X}  , 
	\rR f_\ast \bG^{\wedge}_{m}(\cB) \big).
\end{tikzcd}
\]
The top  equivalence appears in \Cref{prop:defan-tangent-space}; the bottom  equivalence is as above.  
\end{lemma}

Using this we can now compute the\vspace{-2pt}   effect of $\st'_X$ on tangent fibres:
\begin{proposition}[Tangent to $\st'_X$]\label{prop:tangent-to-ST}
\label{prop:tangent-to-st-X}
Assume $k$ is algebraically closed. Then there is a commutative square
\[
\begin{tikzcd}
	\Omega\Def_X(A) \arrow{r}{\simeq} \arrow{d}{\st'_X(A)} 
	& \Map_{\Mod(X_\et,\cO_X)}(\Omega^1_{X/k},\, V\otimes_k \cO_X)
	\arrow{d}{  \dlog} \\
	\Omega\ST'_X(A) \arrow{r}{\simeq} 
	& \Map_{k}\!\big(k\otimes_\bZ\RGamma(X_\et,\bG_{m,X}),\,
	V \otimes_k \RGamma(X,\cO_X)\big)
\end{tikzcd}
\]
in which the horizontal maps are equivalences.
\end{proposition}

\begin{proof}
\Cref{prop:Gm-deformation} gives rise to a commutative square
\[
\begin{tikzcd}
	\Map'_{\Shv(X_\et,\CAlg^\an_k)}(\cO_X,\cB) \arrow{d}{\bG_m} \arrow{r}{\simeq} 
	& \Map_{\cO_X}(L_{X/k},\,\cO_X) \arrow{d}{\dlog} \\
	\Map'_{\Mod(X_\et,\bZ)}(\bG_{m,X},\bG_m(\cB)) \arrow{r}{\simeq} 
	& \Map_{\Mod(X_\et,\bZ)}(\bG_{m,X},\, V\otimes_k \cO_X).
\end{tikzcd}
\]
Postcomposing with $Rf_\ast$, we obtain a   commutative square 
\[
\begin{tikzcd}
	\Map'_{\Shv(X_\et,\CAlg^\an_k)}(\cO_X,\cB) \arrow{d}{\rR f_\ast\bG^{\wedge}_{m}} \arrow{r}{\simeq} 
	& \Map_{\cO_X}(L_{X/k},\,\cO_X) \arrow{d}{} \\
	\Map'(\rR f_\ast  \bG^{\wedge}_{m,X}  ,
	\rR f_\ast \bG^{\wedge}_{m}(\cB)) \arrow{r}{\simeq} 
	& \Map(\rR f_\ast  \bG^{\wedge}_{m,X}  ,\, V\otimes_k \rR f_\ast \cO_X).
\end{tikzcd}
\]
Here the bottom mapping spaces are computed  in $\Mod(S_\et,\bZ_p)$. As $k$ is algebraically closed  and  $V\otimes_k \cO_X$ is killed by $p$, the bottom right space is   \mbox{equivalent to}
\[
\Map_{k}\!\big(
k\otimes_\bZ \RGamma(X_\et,\bG_{m,X}),\,
V\otimes_k \RGamma(X,\cO_X)
\big).
\]
Combining this with \Cref{lemma:serre-tate-tangent} gives the result.
\end{proof}

Next, we compute the effect of the map  $\st'_{X,\lambda}: \widetilde{\Def}_{X,\lambda} \rightarrow \ST'_{X,\lambda}$ on \mbox{tangent fibres.}

To begin with, note that the restriction of any formal moduli problem to $k$-algebras is naturally pointed, and therefore defines a functor $\CAlg^{\an,\art}_k \rightarrow \cat{S}_\ast$. We write $\widetilde{\Def}_\lambda$ for the fibre of the restriction of $\widetilde{\Def}_{X,\lambda} \rightarrow \widetilde{\Def}_{X}$ to $\CAlg^{\an,\art}_k$,  and   define $\ST_\lambda$ analogously. 
Evaluating at $A=\sqz_k V$, we obtain a commutative diagram
\begin{equation}\label{eq:fibre-sequences-X-lambda}
\begin{tikzcd}
	\widetilde{\Def}_\lambda(A) \arrow{r} \arrow{d}{\st'_\lambda(A)} 
	& \widetilde{\Def}_{X,\lambda}(A) \arrow{r} \arrow{d}{\st'_{X,\lambda}(A)} 
	& \widetilde{\Def}_X(A) \arrow{d}{\st'_X(A)} \\
	\ST'_\lambda(A) \arrow{r}  
	& \ST'_{X,\lambda}(A) \arrow{r} 
	& \ST'_X(A)  
\end{tikzcd}
\end{equation}
of pointed spaces, in which both rows are fibre sequences.  

\begin{proposition}[Tangent to $\st'_\lambda$]
\label{prop:tangent-to-st-lambda}
There is a commutative square
\[
\begin{tikzcd}
\widetilde{\Def}_{\lambda}(A) \arrow{r}{\simeq} \arrow{d}{\st'_\lambda(A)} 
	& \Map_{\cO_X}(\cO_X, V[1]\otimes_k \cO_X)
	\arrow{d}{\rR f_\ast} \\
	\ST'_\lambda(A) \arrow{r}{\simeq} 
	& \Map_{\cO_S}\big( \rR f_\ast \cO_X, V[1]\otimes_k \rR f_\ast\cO_X\big)
\end{tikzcd}
\]

in which the horizontal maps are equivalences.
\end{proposition}

\begin{proof}
Using the decomposition $\bG_m(\cB) \simeq \bG_{m,X} \oplus (V\otimes_k \cO_X)$ induced by (\ref{eq:multiplicative-square-zero}), we find a  commutative square
\[
\begin{tikzcd}
\widetilde{\Def}_{\lambda}(A) \arrow{r}{\simeq} \arrow{d}{\st'_\lambda(A)} 
	& \Map_{\Mod(X_\et,\bZ)}\big(\bZ[-1],\,
	V\otimes_k \cO_X \big)
	\arrow{d}{} \\
	\ST'_\lambda(A) \arrow{r}{\simeq} 
	& \Map_{\Mod(S_\et,\bZ_p)}\big(\rR f_\ast \bZ_p[-1],\,
	V\otimes_k \rR f_\ast \cO_X \big)
\end{tikzcd}
\]
in which the horizontal maps are equivalences, and the right map is induced by 
\[
\Mod(X_\et,\bZ) \longto \Mod(X_\et,\bZ_p)
\overset{\rR f_\ast}{\longto} \Mod(S_\et,\bZ_p).
\]
The lower right  space is equivalent to
$\Map_{\cO_S}\big( \rR f_\ast \cO_X, V[1]\otimes_k \rR f_\ast\cO_X\big)$ by 
  \Cref{prop:1-ordinary}.
\end{proof}

\subsection{Canonical coordinates for ordinary Calabi--Yau varieties}\label{subsec:cancoord}
Let $X/k$ be proper, smooth of relative dimension $d$, and Bloch--Kato $1$-ordinary, where $k$ is a perfect field of characteristic $p$. Let $\lambda\colon \bZ[-1] \rightarrow \bG_{m,X}$ in $\Mod(X_\et,\bZ)$. 
 Consider the truncation $H := \tau_{\leq -d} \rR f_\ast \bZ_p$ in $\Mod(S_\et,\bZ_p)$.  \Cref{prop:1-ordinary} shows that $H$ is concentrated in cohomological degree $d$, and 
 \mbox{that its stalk at $\bar{k}$ is $\rH^d(X_{\bar k,\et},\bZ_p)[-d]$.}
\begin{definition}[Serre--Tate period domain and period maps]
Consider the formal moduli problems $\ST_X$ and $\ST_{X,\lambda}$ defined by
\begin{eqnarray*}
\ST_X(A) &:=& \Map_{\Mod(S_\et,\bZ_p)}\big(
\rR f_\ast \widehat{\bG}_{m,X}, \,
H \otimes_{\bZ} {\bG}_m(   \cW \otimes_W A)[1] 
\big) \\
\ST_{X,\lambda}(A) &:=& \Map_{\Mod(S_\et,\bZ_p)}\big(
\rR f_\ast P_{X,\lambda}, \  \,
H \otimes_{\bZ} {\bG}_m(   \cW \otimes_W A)[1] 
\big),
\end{eqnarray*}
where $P_{X,\lambda}$ is defined by the cofibre sequence~(\ref{eq:def-Q}) above. 

The truncation map,  \Cref{lemma:ST-explicit}, and \Cref{prop:equivetale} give a commutative diagram in $\Moduli^\an_W$:
\[
\begin{tikzcd}
\widetilde{\Def}_{X,\lambda} \arrow{r}{\st'_{X,\lambda}} \arrow{d}
& \ST'_{X,\lambda} \arrow{d} \arrow{r}
& \ST_{X,\lambda} \arrow{d} \\
\Def_X \simeq \widetilde{\Def}_X \arrow{r}{\st'_X} & \ST'_X \arrow{r} & \ST_X.
\end{tikzcd}
\]
The horizontal compositions are the \textit{Serre--Tate period maps} $\st_{X,\lambda}$ and $\st_X$. 
\end{definition} 

\begin{theorem}[Theorem B]\label{thm:serre-tate}
 Let $X/k$ be smooth, proper, and geometrically irreducible  
 of dimension $d$ with $\omega_X \cong \cO_X$. Let $\lambda\colon \bZ[-1] \rightarrow \bG_{m,X}$ in $\Mod(X_\et,\bZ)$. Assume moreover that
\begin{enumerate}
	\item $X$ is Bloch--Kato $2$-ordinary;
	\item $\rH^d(X_{\bar k,\et},\bZ_p)$ is torsion-free.
\end{enumerate}
Then $\st_{X,\lambda}$ and $\st_X$ are equivalences.
\end{theorem}

\begin{proof}
Using \Cref{prop:conservativity} and the fibre sequences in~(\ref{eq:fibre-sequences-X-lambda}) it suffices to show that the maps $\st_{X}(\sqz_k V)$ and $\st_{\lambda}(\sqz_k V)$ are equivalences for all perfect $V \in \Mod_k^{\cn}$. First, let us assume that $k$ is algebraically closed.

By   (2)  and \Cref{prop:1-ordinary}, the natural map $k\otimes_{\bZ_p} H \rightarrow \rH^d(X,\cO_X)[-d]$ is an equivalence.
Using \Cref{prop:tangent-to-st-X}, we see that $\st_X(\sqz_kV)$ is given by the map
\[
\Map_{\cO_X}\!\big(\Omega^1_{X/k},\, 
V[1]\otimes_k \cO_X\big) \longto
\Map_{k}\!\big( k\otimes_{\bZ} \RGamma(\bG_{m,X}),\,
V\otimes_k \rH^d(\cO_X)[1-d] \big)
\]
induced by $\RGamma$, $\dlog$, and the truncation. By (1) and  \Cref{prop:2-ordinary}, the map
\[
\dlog \colon k\otimes_{\bZ} \RGamma(X_\et,\bG_m) \rightarrow 
\RGamma(X,\Omega^1_{X/k})
\]
is an equivalence in $\Mod_k$. Hence, $\st_X(\sqz_k V)$ can be identified with the map
\[
\Map_{\cO_X}\big(\Omega^1_{X/k},\, 
V[1]\otimes_k \cO_X\big) \longto
\Map_{k}\!\big( \RGamma(\Omega^1_X),\, 
V\otimes_k \rH^d(X,\cO_X)[1-d] \big)
\]
induced by $\RGamma$ and truncation. As the canonical bundle is trivial, this map is an equivalence by Serre duality, so $\st_X(\sqz_k V)$ is an equivalence.

Using \Cref{prop:tangent-to-st-lambda}, we see that $\st_\lambda(\sqz_k V)$ identifies with the map
\[
\Map_{\cO_X}(\cO_X, V[1]\otimes_k \cO_X) \longto
\Map_k\!\big( \RGamma(\cO_X),\, V \otimes_k \rH^d(\cO_X)[1-d] \big)
\]
induced by $\RGamma$ and truncation. This is again an equivalence by Serre duality, and the assumption that the canonical bundle is trivial. 

For a general perfect field $k$ of characteristic $p$, it suffices to check that $\st_{X}$ and $\st_{X,\lambda}$ induce equivalences on tangent fibres after applying the functor $\overline{k} \otimes_k (-)$.
By \Cref{rem:tangentfibre}, this amounts to showing that the transformations $\Omega\st_{X}(\sqz_k \overline{V})$ and $\Omega\st_{X,\lambda}(\sqz_k \overline{V})$ are equivalences for all perfect $\overline{V} \in \Mod_{\overline{k}}$.
But \Cref{prop:pd-cotangent-sheaf-adjunction}, \Cref{lemma:serre-tate-tangent}, \Cref{prop:tangent-to-st-lambda}, 
and \Cref{prop:2-ordinary} together imply that these transformations are equivalent to 
$\Omega\st_{\overline{X}}(\sqz_k \overline{V})$ and $\Omega\st_{\overline{X},\lambda}(\sqz_k \overline{V})$, respectively, where $\overline{X}$ is the base change of $X$ to $\overline{k}$.
Hence they are equivalences  by our previous discussion.
\end{proof}

\begin{remark}
If $\rH^{d-1}(X,\cO_X)=0$, then  
Condition (2) in 
 \Cref{thm:serre-tate} \mbox{is  satisfied.}
\end{remark}

\begin{remark}[Structure of $\pi_0\ST_X$ and $\pi_0 \ST_{X,\lambda}$] \label{rmk:underived-ST}
Assume that $k$ is algebraically closed.
If $X$ is Bloch--Kato $2$-ordinary, then by \cite[Lemma~7.1]{BlochKato86}, we have 
\[
\rH^\ast(X_\et,\bG^{\wedge}_{m}) \cong
\rH^\ast(X,W\Omega^1_{X/k})^{F=1}
\]
as $\bZ_p$-modules. The group $\pi_0 \ST_X(A)$ can then  be described explicitly as
\[
\pi_0 \ST_X(A) \cong
\Hom_{\bZ_p}\!\big( \rH^{d-1}(X,W\Omega^1_{X/k})^{F=1},\,
\rH^d(X_\et,\bZ_p) \otimes_{\bZ} \bG_m(A) \big).
\]
Moreover, the homomorphism $\varphi\colon \pi_0 \ST_{X,\lambda}(A) \rightarrow \pi_0 \ST_X(A)$ satisfies
\begin{eqnarray*}
	\ker(\varphi) &\cong& 
	\Hom_{\bZ_p}\!\big(\rH^{d-1}(X_\et,\bZ_p),\,
	\rH^d(X_\et,\bZ_p) \otimes_{\bZ} \bG_m(A) \big) \\
	\coker(\varphi) &\cong&
	\Hom_{\bZ_p}\!\big(\rH^{d-2}(X_\et,\bZ_p),\,
	\rH^d(X_\et,\bZ_p) \otimes_{\bZ} \bG_m(A) \big).
\end{eqnarray*}
If $X$ is a K3 surface, then this recovers Deligne's description \cite[2.2.2]{Deligne81} of the relation between canonical coordinates and deformations of line bundles on K3 surfaces.
\end{remark}

\begin{definition}[Canonical lifts]
Let $X/k$ be as in \Cref{thm:serre-tate}. For every $n$, the unit of $\ST_X(W/p^n)$ corresponds to a lift $X_n$ over $W/p^n$, and we denote the formal scheme $\lim_n X_n$ over $W$
by $X^\can$. We refer to it as the \emph{canonical lift} of $X$.
\end{definition}

\begin{corollary}[\Cref{mainthm:line-bundles}]If $X$ is projective, then so is $X^\can$.
\end{corollary}

\begin{proof}
Assume $\lambda \colon \bZ[-1]\rightarrow \bG_{m,X}$ corresponds to an ample line bundle $\cL$ on $X$. For each $n$, the unit section of $\ST_{X,\lambda}(W/p^n)$ corresponds to a lift $\lambda_n \colon \bZ[-1] \rightarrow \bG_{m,X_n}$ of $\lambda$, and hence to a (necessarily ample) line bundle $\cL_n$ on $X_n$. These form a compatible system, giving rise to an ample line bundle $\cL^\can$ on $X^\can$.
\end{proof}

\begin{remark}[Canonical lifts of Brauer classes]
There is nothing special about the shift $[-1]$ in the definition of $\ST_{X,\lambda}$, and in the proof of \Cref{thm:serre-tate}. The same reasoning can be applied to deformations of maps $\bZ[-2]\rightarrow \bG_{m,X}$, and this shows that Brauer classes on $X$ admit canonical lifts to $X^\can$.
\end{remark}

\begin{remark}[Complex multiplication]\label{rmk:complex-multiplication}
Choose an embedding $W(k) \rightarrow \bC$. Then every $X/k$ satisfying the conditions of \Cref{mainthm:Serre-Tate} gives rise to a complex variety $X^\can_{\bC}$. If $X$ is an abelian variety or a K3 surface, then $X^\can_{\bC}$ is known to have complex multiplication. More generally, this should be the case for families parametrised by Shimura varieties. 
We do not expect canonical lifts of arbitrary ordinary Calabi--Yau varieties to be of CM type. In fact, this would contradict conjectures on CM points in variations of Hodge structures. Indeed, consider for example quintic threefolds. There exist an algebraically closed field $k$ of positive characteristic, and a quintic threefold $X/k$ for which the Hodge and Newton polygons coincide, and by Grothendieck's semicontinuity theorem \cite[Theorem 2.3.1]{Katz79}, this is true for an open dense subset of the quintic threefolds over $k$. By \Cref{mainthm:line-bundles}, they all have canonical lifts which are quintic threefolds. If these are all CM, then the moduli space of quintic threefolds over $\bC$ would have a dense collection of CM points, contradicting \emph{e.g.}\ a conjecture by Klingler \cite[Conjecture 5.6]{Klingler17}.
\end{remark}

\section{Spectral lifts}
\label{sec:spectral-lifts}
In the previous sections, we have used techniques from  {derived} algebraic geometry to prove that large classes of Calabi--Yau varieties admit lifts to the Witt vectors. One could ask if one can use  {spectral} algebraic geometry to prove that these even lift to the spherical Witt vectors, as is the case for ordinary abelian varieties. We will show that this rarely happens.

Let $k$ be a perfect field of characteristic~$p$. Let $\SW(k)$ be the $\bE_\infty$-ring spectrum of spherical Witt vectors \cite[\S 5.2]{LurieEllipticII}. This is a $p$-complete $\bE_\infty$-ring spectrum, flat over the sphere spectrum $\bS$,   such that $\bF_p\otimes_{\bS} \SW(k) \simeq k$. We have $\pi_0 \SW(k)\cong W(k)$.

\begin{proposition}\label{prop:no-spherical-lifts}
Let $X$ be a smooth projective irreducible scheme over an algebraically closed field $k$ of characteristic $p$ with $\omega_{X/k}\cong \cO_X$. If $X$ lifts to a formal spectral scheme over $\SW(k)$, then $X$ admits finite \'etale cover by an abelian variety.
\end{proposition}

This shows in particular that a simply connected Calabi--Yau variety such as a K3 surface or a quintic threefold cannot lift to $\SW(k)$.

The proof is based on the fact that flat commutative $\SW(k)$-algebras carry canonical Frobenius lifts. 
Let us write $\CAlg_{\SW(k)}^\fl \subset \CAlg_{\SW(k)}$ for the full subcategory spanned by all $\bE_\infty$-$\SW(k)$-algebras which are flat in the sense of \cite[7.2.2.10]{LurieSAG}.
Write $\CR_{W(k)}$ for the category of ordinary  (commutative) $W(k)$-algebras, and 
$\CR_{W(k)}^{\fl,\delta}$ for the  category of flat  $W(k)$-algebras $A$ equipped with a ring homomorphism $\phi\colon A\rightarrow A$ such that $\phi(x) \equiv x^p \bmod{p}$.

\begin{lemma}\label{lemma:frobenius-lift-from-spherical-witt}
The functor $\pi_0\colon \CAlg_{\SW(k)}^\fl \rightarrow \CR_{W(k)}$ admits a lift to $\CR^{\fl,\delta}_{W(k)}$.
\end{lemma}

\begin{proof} Let $K$ be the Morava $E$-theory corresponding to the formal multiplicative group over $k$. We have $\pi_* K \cong W(k)[u,u^{-1}]$ with $|u|=2$. Using the universal property of the $\bE_\infty$-ring of spherical Witt vectors, we obtain   maps of $\bE_\infty$-rings
\[
\SW(k) \rightarrow \tau_{\geq 0} K  \rightarrow K
\]
inducing isomorphisms on $\pi_0$. 
Consider the functor $ \CAlg_{\SW(k)}^\fl \rightarrow \CAlg_K^{\wedge}$, 
$B \mapsto B' := (K \otimes_{\SW(k)} B)^{\wedge}_p$. Here $\CAlg_K^{\wedge}$ denotes  to the $\infty$-category of $p$-complete, i.e.\ $K(1)$-local, $\EE_\infty$-$K$-algebras.
By \cite[\S 3]{Hopkins2014} \cite{McClure86}, the ring $\pi_0(B')$ carries a natural Frobenius lift, functorial in $B'$. As $B'$ is flat over $\SW(k)$,  the  map $\pi_0(B)\rightarrow \pi_0(B')$ is an isomorphism.
\end{proof}

\begin{proof}[Proof of \Cref{prop:no-spherical-lifts}]
Let $(X,\cB)$ be a formal spectral scheme over $\SW(k)$ lifting
$(X,\cO_X)$. A similar argument as in \Cref{cor:flat-scheme-over-artinian} shows that the $\SW(k)$-module $\cB(U)$ is flat for every affine open $U \subset X$. By \Cref{lemma:frobenius-lift-from-spherical-witt}, the rings $\pi_0 \cB(U)$ acquire compatible lifts of Frobenius. In particular, $X$  lifts to a scheme \mbox{$(X,W_2(k) \otimes_{W(k)} \pi_0 \cB)$}  over $W_2(k)$, equipped with a Frobenius lift. But then it follows from \cite[Proposition 3.3.1.c]{AchingerWitaszekZdanowicz2021} and \cite[Theorem 2]{MehtaSrinivas1997} that $X$ admits a finite \'etale cover by an ordinary abelian variety. 
\end{proof}

\end{document}